\documentclass[11pt]{amsproc}
\pdfoutput=1
\usepackage{amsmath,amssymb,amsthm,eucal, eufrak,amscd,tikz,tikz-cd,mathtools}
\usepackage{xcolor}
\usepackage{enumitem}
\usepackage[shortalphabetic]{amsrefs}
\usepackage{xypic}
\usepackage{graphicx}
\definecolor{allrefcolors}{rgb}{0,0.2,0.5}
\usepackage[linktocpage=true,colorlinks=true,allcolors=allrefcolors,bookmarksopen,bookmarksdepth=3]{hyperref}
\usepackage[margin=1.25in]{geometry}

\newcommand{\Hom}{\operatorname{Hom}}

\def\Z{{\mathbb Z}}
\def\R{{\mathbb R}}
\def\C{{\mathbb C}}
\def\P{{\mathbb P}}

\def\L{{\mathbb L}}
\def\K{{\mathbf{k}}}
\def\w{\mathcal{W}(E)}
\def\A{\mathcal{A}}
\def\B{\mathcal{B}}

\def\n{\mathcal{N}}
\def\o{\mathcal{O}}

\def\Y{\mathcal{Y}}
\def\cP{\mathcal{P}}
\def\cQ{\mathcal{Q}}
\def\cR{\mathcal{R}}

\def\cZ{\mathcal{Z}}
\def\e{\epsilon}
\def\cc{\mathcal{C}}
\def\dd{\mathcal{D}}
\def\ee{\mathcal{E}}
\def\oc{\mathcal{OC}}

\def\mod{\mathrm{\!-\! mod}}
\def\rmod{\mathrm{mod\!-\!}}
\def\bimod{\mathrm{\!-\! mod\!-\!}}

\def\e{\epsilon}

\def\r#1{\mathrm{#1}}

\def\mc#1{\mathcal{#1}}

\def\w{\mathcal{W}}
\def\rw{\mathcal{RW}}
\def\k{\kappa}
\def\d{\delta}
\def\D{\Delta}
\def\M{\mathcal{M}}

\def\ainf{A_\infty}

\def\z2{\Z / 2\Z}

\def\id{\mathrm{id}}

\def\p{\partial}
\def\A{\mc{A}}
\def\ob{\mathrm{ob\ }}
\def\ch{\mathrm{Ch}}
\def\perf{\mathrm{Perf}}
\def\calk{\mathrm{Calk}}
\def\fun{\mathrm{Fun}}
\def\cone{\mathrm{cone}}

\def\cinf{\widehat{\cc}_{\infty}}
\def\winf{\widehat{\w}_{\infty}}
\def\n{\mathcal{N}}

\def\ck{\mathfrak{c}}
\def\sk{\mathfrak{s}}

\newtheorem{lem}{Lemma}[section]
\newtheorem{prop}[lem]{Proposition}
\newtheorem{thm}[lem]{Theorem}
\newtheorem{cor}[lem]{Corollary}
\newtheorem{conj}[lem]{Conjecture}
\newtheorem{defn}[lem]{Definition}

\newtheorem{rem}[lem]{Remark}
\def\e{\epsilon}

\theoremstyle{remark}

\newtheorem{example}{Example}[section]
\numberwithin{equation}{section}
\begin{document}
\begin{abstract}
Inspired by the simple fact that a compact $n$-dimensional manifold-with-boundary which satisfies Poincar\'{e}-Lefschetz duality of dimension $n$ has a boundary which itself satisfies Poincar\'{e} duality of dimension $n-1$,
we show that the categorical formal punctured neighborhood of infinity,
a canonical categorical construction associated to every $\ainf$-category,
has a weak proper Calabi-Yau structure of dimension $n-1$ whenever the original $\ainf$-category admits a weak smooth Calabi-Yau structure of dimension $n$.
Applications include proper Calabi-Yau structures on Rabinowitz Fukaya category of a Liouville manifold and Orlov's singularity category of a proper singular Gorenstein scheme of finite type.
\end{abstract}

\title{Abstract categorical residues and Calabi-Yau structures}
\author{Yuan Gao} \thanks{University of Georgia, Athens, GA 30602, USA. Email: \url{yuangaohl@gmail.com}}
\email{yuangaohl@gmail.com}
\address{Department of Mathematics, University of Georgia, Athens, GA 30602}

\maketitle
\tableofcontents

\section{Introduction}

\subsection{Rabinowitz Fukaya categories}

The Rabinowitz Fukaya category $\rw(X)$ of a Liouville manifold $X$,
introduced by \cite{GGV} and appearing in \cite{BJK},
 is an $\ainf$-category which is geometrically constructed using methods of Floer theory and measures the failure of Poincar\'{e} duality to hold for the wrapped Fukaya category $\w(X)$ \cites{abouzaid-seidel, abouzaid_gc, ganatra}.
This measurement originates from a classical analogy in topology, 
which says cochains on the boundary of a compact oriented manifold-with-boundary measures the failure of Poincar\'{e} duality between chains and cochains on the compact manifold-with-boundary.
As noted, the boundary of a compact manifold-with-boundary is itself a closed manifold,
on which Poincar\'{e} duality always holds.
Since the Rabinowitz Fukaya category behaves in a manner as an invariant of the boundary-at-infinity and is also a measurement for failure of Poincar\'{e} duality,
 a natural question arises: does the Rabinowitz Fukaya category have a similar property, i.e., satisfy Poincar\'{e} duality itself?

The construction of the chain-level $\ainf$-operations on the morphism spaces in $\rw(X)$,
named the Lagrangian Rabinowitz Floer complexes defined by mapping cones of Floer continuation maps from a very negative Hamiltonian to a very positive one,
was realized by \cite{GGV} using variants of popsicles originating from \cite{abouzaid-seidel, seidel6}.
Algebraically, compared to works of \cites{COcone, CHOduality} which study product and coproduct structures on homology of mapping cones aiming at understanding algebraic structures on the closed-string Rabinowitz Floer cohomology,
the algebraic structures in the open-string theory tend to have a particularly simple construction realizing all of them simultaneously on the chain level.

Return to the question about Poincar\'{e} duality.
Numerous important works concerning this question on higher-algebra levels are ultimately related to geometry,
including but not limited to \cites{fukaya_cyclic, cho, tradler, tradler-zeinalian, kontsevich-soibelman} and many more for which one cannot name all,
lead us to consider the notion of a {\it weak proper Calabi-Yau structure} on an $\ainf$-category,
a homotopy version of Poincar\'{e} duality for $\ainf$-categories that are mostly essentially well-defined up to homotopy.

\begin{defn}[Definition \ref{def: cy structures}]\label{def: weak proper cy}
A weak proper Calabi-Yau structure on an $\ainf$-category $\cc$ of dimension $n$ is a chain map of degree $-n$
\begin{equation}\label{weak proper cy map}
\phi: \r{CC}_{*}(\cc, \cc) \to \K[-n],
\end{equation}
or equivalently a cycle in $\r{CC}_{*}(\cc, \cc)^{\vee}[-n]$,
such that the induced pairing
\begin{equation}\label{induced pairing}
H^{*}(\hom_{\cc}(X, Y)) \otimes H^{n-*}(\hom_{\cc}(Y, X)) \stackrel{[\mu_{\cc}^{2}]}\to H^{n}(\hom_{\cc}(Y, Y)) \stackrel{[i]}\to \r{HH}_{n}(\cc) \stackrel{[\phi]}\to \K
\end{equation}
is nondegenerate,
where $\mu_{\cc}^{2}$ is the product, as part of the $\ainf$-structure maps of $\cc$,
and $i: \hom_{\cc}(Y, Y) \to \r{CC}_{*}(\cc)$ is the canonical inclusion of chain complexes.
\end{defn}

In this paper, we study further categorical structures on the Rabinowitz Fukaya category of a Liouville manifold and prove that the Rabinowitz Fukaya category, as an $\ainf$-category, admits a chain-level weak proper Calabi-Yau structure $\phi$.
Recall from \cite{ganatra} that a Liouville manifold is said to be {\it nondegenenrate}, 
if it admits a finite collection $\B$ of exact cylindrical Lagrangians satisfying Abouzaid's generation criterion \cite{abouzaid_gc},
meaning that the {\it open-closed map} (a canonical geometrically defined map for Fukaya categories which will be recalled in \S\ref{section: open-closed})
\begin{equation}
\oc: \r{HH}_{*-n}(\B, \B) \to SH^{*}(X)
\end{equation}
hits the unit in symplectic cohomology.
We note that Definition \ref{def: weak proper cy} is meaningful even for non-proper $\ainf$-categories,
although it does not say in the usual case of a proper $\ainf$-category that the weak smooth Calabi-Yau structure induces a quasi-isomorphism from the diagonal bimodule $\cc_{\D}$ to its linear dual $\cc^{\vee}$.

\begin{thm}\label{thm: rw cy}
Let $\K$ be a field of arbitrary characteristic. 
Let $X$ be a non-degenerate Liouville manifold with $c_{1}(X) = 0$.
Then the Rabinowitz Fukaya category $\rw(X)$ admits a weak proper Calabi-Yau structure of dimension $n-1$.
\end{thm}

Some comments on Theorem \ref{thm: rw cy} are in order.
This theorem seems quite puzzling at the first glance,
because the Rabinowitz Fukaya category $\rw(X)$ is almost always non-proper,
as is the wrapped Fukaya category $\w(X)$ \cite{Gnonproper, GGV}.
But as is already mentioned above, 
nondegeneracy of the pairing \eqref{induced pairing} does not rely on properness of the category,
yet for a non-proper category it might not be equivalent to the data of a quasi-isomorphism between the diagonal bimodule and the linear dual bimodule in the usual sense.
A related issue was addressed in \cite{CO_tate}, 
where it is shown that the (closed-string) Rabinowitz Floer cohomology of any Liouville domain is a self-dual locally linearly compact vector space in the sense of Lefschetz \cite{lefschetz}, or equivalently a Tate vector space in the sense of Beilinson-Feigin-Mazur \cite{Beilinson-Feigin-Mazur}.
However, our result is independent of the use of the theory of topological vector spaces,
but is expected to reflect this structure in an appropriate sense.
For example, we include a discussion on the local linear compactness of the Rabinowitz Floer complex that we considered in this paper (Proposition \ref{prop: rab tate}).

In \cite{BJK}, it is proved that if the wrapped Floer cohomology for every pair of Lagrangians is finite dimensional in each degree,
then the derived Rabinowitz Fukaya category, as an ordinary category, 
is Calabi-Yau of dimension $n-1$ in the sense that there exists a non-degenerate pairing of degree $1-n$ between the morphism spaces.
They constructed a pairing on the Rabinowitz Floer complexes object-wise,
and showed that the induced maps are quasi-isomorphisms on cohomology;
however, this type of duality statement is irrelevant to the $\ainf$-structure of $\rw(X)$.
It is not clear from the construction of the object-wise map whether there are higher order pieces of information that would yield a closed element in the Hochschild chain complex.
In this sense, Theorem \ref{thm: rw cy} can be regarded as a chain-level improvement that is compatible with the $\ainf$-structure.
But as it turns out, a direct geometric construction of the desired trace map \eqref{weak proper cy map} using holomorphic disk counts seems to be out of reach of current techniques.

Theorem \ref{thm: rw cy} is in fact much stronger than a cohomology-level statement,
as the induced pairing on $\rw(X)$ has to be compatible with the $\ainf$-structure in a very strong way.
In addition, the existence of a Calabi-Yau structure implies non-trivial relations between Hochchild invariants of the category, 
on which there are TQFT operations \cite{kontsevich-takeda-vlassopoulos}.
In addition, a proper Calabi-Yau structure on the Rabinowitz Fukaya category appears as a first piece of evidence of the existence of a more general type of algebraic structures on the wrapped Fukaya category,
namely the pre-Calabi-Yau structures introduced by \cite{kontsevich-takeda-vlassopoulos}.

\begin{rem}
We expect that Theorem \ref{thm: rw cy} holds even over $\Z$.
But since the duality issues are more delicate, 
and since we also follow the construction of the weak smooth Calabi-Yau structure for the wrapped Fukaya category in \cite{ganatra} which only did it over a field (though seems valid over $\Z$),
we will restrict ourselves to the case where $\K$ is a field.
\end{rem}

Going in another direction, we ask if there is a lift of the weak proper Calabi-Yau structure in Theorem \ref{thm: rw cy} to a {\it strong} proper Calabi-Yau structure.
Recall that 

\begin{defn}
A strong proper Calabi-Yau structure of dimension $n$ on $\cc$ is a chain map
\begin{equation}
\tilde{\phi}: \r{CC}_{*}(\cc, \cc)_{hS^{1}} \to \K[-n]
\end{equation}
such that the composition $\tilde{\phi} \circ pr$ is a weak proper Calabi-Yau structure,
where 
\begin{equation}
pr: \r{CC}_{*}(\cc, \cc) \to \r{CC}_{*}(\cc, \cc)_{hS^{1}}
\end{equation}
is the canonical projection to homotopy orbits.
\end{defn}

\begin{conj}\label{lift to strong cy}
The weak proper Calabi-Yau structure from Theorem \ref{thm: rw cy} admits a lift to a strong proper Calabi-Yau structure on $\rw(X)$.
\end{conj}

Although it is not a priori guaranteed that the compatibility between the $\ainf$-structure and the pairing on $\rw(X)$ as strong as the conditions required by a {\it cyclic $A_{\infty}$-category} (\cites{costello, fukaya_cyclic, cho-lee},
if Conjecture \ref{lift to strong cy} is proved, over a field of characteristic zero there is not much essential difference due to \cite{kontsevich-soibelman},
which roughly says that the $\ainf$-structure can be made strictly compatible with the pairing up to some canonical quasi-equivalence.
In addition, it is shown in \cite{kontsevich-takeda-vlassopoulos} that an $n$-dimensional pre-Calabi-Yau structure on an $\ainf$-algebra $A$ is equivalent to a cyclic $\ainf$-algebra structure on $A^{\vee}[1-n] \oplus A$ of dimension $n-1$.
The Rabinowitz Fukaya category behaves in a way similar to $A^{\vee}[1-n] \oplus A$ (as we will see through \S\ref{section: rw relation to w}),
but it is too much to ask for strict cyclic symmetry as well as a full pre-Calabi-Yau structure in the relevant setups of Floer theory in the non-compact case.
In this sense Theorem \ref{thm: rw cy} provides a weak version, even in the infinite-dimensional case,
of the existence of an additional structure on the wrapped Fukaya category $\w(X)$ that would be close to an honest pre-Calabi-Yau structure (which is expected to exist but hard to be implemented due to lack of symmetries).

Some careful treatment of the chain-level circle actions on the wrapped and Rabinowitz Fukaya categories and also on their closed-string counterparts, using ideas and techniques of e.g \cite{ganatra_cyclic} and/or even \cite{AGV},
should lead to an answer to Conjecture \ref{lift to strong cy}.

\subsection{Orlov's singularity category}

Somewhat surprisingly, our method for proving Theorem \ref{thm: rw cy} is very general,
and will lead to other important examples from algebraic geometry,
among which we'd like to address the Orlov's category of singularities \cite{orlov}.
Recall that for a separated Noetherian scheme $X$ over a field $\K$,
the Orlov's singularity category is defined to be the Verdier quotient
\[
D^{b}_{sg}(X):= D^{b}Coh(X) / \mathrm{perf}(X),
\]
the bounded derived category of coherent sheaves on $X$ modulo the full subcategory of perfect complexes (in the derived category).
This has a natural dg enhancement
\begin{equation}
\dd^{b}_{sg}(X) = Coh(X)/ \perf(X)
\end{equation}
given by the Drinfeld quotient of the dg enhancement $Coh(X)$ by the full dg subcategory of perfect complexes.

It is known by \cite{murfet} that if $X$ is finite type over $\K$, {\it Gorenstein} and has isolated singularities,
the bounded derived category $D^{b}_{sg}(X)$ has a Serre functor given by a shift $(-)[n-1]$,
where $n$ is the (Krull) dimension of $X$,
and moreover provides a closed formula for the non-degenerate pairing.
The existence of duality actually reduces to a local computation dating back to \cite{auslander}.
It is therefore an interesting question whether this can be promoted to a proper Calabi-Yau structure on the dg enhancement.
In fact, we deduce the following result, without assuming singularities being isolated,
but will require a global trivialization of the canonical bundle:

\begin{thm}\label{thm: sing cy}
Let $\K$ be a perfect field,
and $X$ a Gorenstein scheme of finite type over $\K$, of dimension $n$.
Suppose the canonical bundle of $X$ admits a trivialization. 
Then the natural dg enhancement $\dd^{b}_{sg}(X)$ has a proper Calabi-Yau structure of dimension $n-1$.
\end{thm}

The existence of such a proper Calabi-Yau structure can be seen as a dg enhancement of a Serre functor on the derived category of singularities,
and will follow from the general abstract machinery Theorem \ref{thm: cinf cy},
to be discussed in the following subsection,
as well as the feature that the Orlov's category is in the form of a dg quotient.
Importantly, we will find a chain-level graded symmetric non-degenerate pairing on $\Hom$ spaces of some other category,
quasi-equivalent to the dg category $\dd^{b}_{sg}(X)^{op}$, 
the opposite of Orlov's singularity category.
More details will be given in \S\ref{section: singularity category}.

On the other hand, Orlov's singularity category is also closely related to matrix factorization category.
In fact, Orlov \cite{orlov} proved that for $W: U \to \mathbb{A}^{1}_{\K}$ a regular function from an affine scheme $U$ over a field $\K$ of characteristic zero,
the matrix factorization category $MF_{w}(U, W)$,
as a triangulated category,
is exact equivalent to the singularity category $D^{b}_{sg}(X_{w})$,
where $X_{w} = W^{-1}(w)$,
for any critical value $w$ of $W$ (in fact, $MF_{w}(U, W)$ is trivial if $w$ is a regular value).
It is proved in \cite{shklyarov_mf} that over $\K = \C$,
the $\Z/2$-graded dg category of matrix factorizations has a proper Calabi-Yau structure in $\Z/2$-graded sense,
by constructing an explicit cyclically symmetric pairing.
It would be interesting to understand this from a conceptual point of view using Theorem \ref{thm: sing cy} and Orlov's equivalence.
Although our method does not apply immediately because the category of matrix factorizations is $\Z/2$-graded,
and Orlov's equivalence also breaks the $\Z$-grading down to $\Z/2$-grading,
we expect that the method (i.e., Theorem \ref{thm: cinf cy}) based on which Theorem \ref{thm: sing cy} is proved,
can be extended to the $\Z/2$-graded setting.

\subsection{The categorical residue}

As already mentioned previously, 
such a Calabi-Yau structure in Theorem \ref{thm: rw cy} does not seem to be reachable by a direct Floer-theoretic construction via holomorphic curves.
It is therefore inevitable to reinvestigate the problem in a more general and abstract framework,
where the duality result is independent of the presentation of the relevant $\ainf$-category as a Fukaya-type category.
The strategy of proof of Theorem \ref{thm: rw cy} is therefore to appeal to a purely categorical structure that may be interpreted as the `boundary' of an $\ainf$-category with a weak smooth Calabi-Yau structure.
The nature of the construction,
for which details are carried out throughout \S\ref{section: neighborhood of infinity} also based on results in \S\ref{section: ainf},
 suggests an appropriate name for this invariant - {\it the residue}.
One reason for this is that the chain level maps are defined in a way similar to Tate's construction of abstract residues \cite{tate_residue} using traces of certain operators on infinite-dimensional spaces,
where the theory is further generalized by Beilinson \cite{beilinson}.

The {\it residue} on the {\it categorical formal punctured neighborhood of infinity}, 
is a canonical Hochschild invariant associated with any $\ainf$-category equipped with a weak smooth Calabi-Yau structure.
The notion of the categorical formal punctured neighborhood of infinity is first introduced by \cite{efimov2} in the case of dg categories,
which has a straightforward generalization to $\ainf$-categories further discussed in \cite{GGV}.
Recall that for an $\ainf$-category $\cc$, 
its {\it (algebraizable) categorical formal punctured neighborhood of infinity}, $\cinf$,
is defined to be the essential image of the induced Yoneda functor
\begin{equation}\label{calkin yoneda}
\bar{y}: \cc \to \rmod \cc = \fun(\cc^{op}, \ch_{\K}) \to \fun(\cc^{op}, \ch_{\K} / \perf_{\K}),
\end{equation}
where $\perf_{k}$ is the subcategory of $\ch_{\K}$ consisting of perfect chain complexes.
The target $\ch_{\K}/\perf_{\K}$ is an algebraic analogue of the notion of the Calkin algebra \cite{calkin},
measuring in the non-proper case the difference between all `bounded linear operators' and `compact operators'.
In this sense, the non-triviality of $\cinf$ measures the failure of $\cc$ to be proper.

The categorical formal punctured neighborhood of infinity has received recent attentions in both symplectic geometry \cite{GGV} and string topology \cite{rivera-takeda-wang}.
The general philosophy along the lines of these results says that Calabi-Yau structures on $\cc$ induce product structures on $\r{HH}_{*-n}(\cc, \cinf)$, or equivalently $\r{HH}^{*}(\cc, \cinf)$
which in turn induce products on Hochschild homology of $\cc$ and coproducts on reduced Hochschild homology.
This naturally leads to algebraic approaches to understanding structures on and relations between symplectic cohomology, symplectic homology, Rabinowitz Floer cohomology v.s. loop homology, loop cohomology, Rabinowitz loop cohomology (\cite{CHOduality}),
where the two sides of the theories are related by Viterbo's theorem \cites{Viterbo1, abouzaid_sh, BDHO_dg}.

The definition $\cinf$ as the essential image of \eqref{calkin yoneda} makes sense for any $\ainf$-category, smooth or not.
However, the structure of $\cinf$ is not always well behaved for arbitrary $\cc$.
We will explore its properties and find, when $\cc$ is a weak smooth Calabi-Yau category of dimension $n$,
that there is a canonical map
\begin{equation}\label{cinf res}
\r{res}: \r{CC}_{n-1}(\cinf) \to \K
\end{equation}
defined in \S\ref{section:res}.

The following theorem gives a purely categorical account for the philosophy that the measurement of the failure of Poincar\'{e} duality for Poincar\'{e}-Lefschetz duality must itself satisfy Poincar\'{e} duality of one dimension less.

\begin{thm}\label{thm: cinf cy}
Suppose $\cc$ is a smooth $\ainf$-category having a weak smooth Calabi-Yau structure of dimension $n$.
Then its categorical formal punctured neighborhood $\cinf$ carries a weak proper Calabi-Yau structure of dimension $n-1$,
given by the residue $\r{res}$ \eqref{cinf res}.
\end{thm}

\begin{rem}
Despite the similarity of the names,
the reader should not confuse the induced pairing (which would naturally be called the residue pairing) with K. Saito's higher residue pairing \cite{saito} or Shklyarov's categorical analogue on negative/period cyclic homology \cite{shklyarov2}.
\end{rem}

To relate this algebraic invariant to geometry, in particular the situation of Theorem \ref{thm: rw cy},
we introduce the following definition.

\begin{defn}[Definition \ref{def: weak proper cy functor}]\label{def: cy fun}
Let $\cc, \dd$ be $\ainf$-categories with weak proper Calabi-Yau structures $\phi, \psi$ of the same dimension $n$.
An $\ainf$-functor $F: \cc \to \dd$ is called a weak proper Calabi-Yau functor if $[\psi \circ F] = [\phi] \in \r{HH}_{n}(\cc)$.
$F$ is called weak proper Calabi-Yau equivalence if there exists a weak Calabi-Yau functor $G: \dd \to \cc$ such that the two way compositions are both homotopic to the identity.
\end{defn}

Now let us go back from homological algebra to geometry.
Another main result of \cite[Theorem 1.1]{GGV} is to give an algebraic interpretation of the Rabinowitz Fukaya category, 
which says there is a canonical $\ainf$-functor
\begin{equation}\label{rw to winf}
\Phi: \rw(X) \to \widehat{\w(X)}_{\infty},
\end{equation}
which is a quasi-equivalence whenever the Liouville manifold $X$ has $c_{1}(X)=0$ and is nondegenerate in the sense of \cite{ganatra}; see Theorem \ref{thm: rw=winf}.
We shall use this quasi-equivalence to produce a weak proper Calabi-Yau structure on $\rw(X)$ by pulling back the residue from Theorem \ref{thm: cinf cy}.
In \S\ref{section: residue on rw} we verify that the induced pairing agrees with the tautological pairing, 
and is therefore nondegenerate.

By construction, the functor $\Phi$ will be a weak proper Calabi-Yau equivalence in the sense of Definition \ref{def: cy fun} (Definition \ref{def: weak proper cy functor}). 
Nonetheless, we expect there is a possible way of constructing a proper Calabi-Yau structure purely in terms of Floer theory,
which would potentially lead and/or be related to a geometric pre-Calabi-Yau structures on the wrapped Fukaya category.
This geometric proper Calabi-Yau structure should induce the same pairing, 
such that the quasi-equivalence $\Phi$ is a weak proper Calabi-Yau functor between these weak proper Calabi-Yau structures.
We leave this to a future research as the current method of using popsicles (\S\ref{section: popsicles}) to construct the $\ainf$-structures does not seem to be applicable to disks with more than two outputs.

Regarding the algebro-geometric result Theorem \ref{thm: sing cy},
it is shown by \cite[Theorem 9.2]{efimov2} that for a separated scheme $X$ over finite type over a perfect field $\K$,
$\dd^{b}_{sg}(X)^{op}$ is quasi-equivalent to $\widehat{\dd^{b}Coh(X)}_{\infty}$.
For $X$ Gorenstein of dimension $n$ with a trivialization of the canonical bundle,
the dg category $\dd^{b}Coh(X)$ admits a smooth Calabi-Yau structure of dimension $n$.
So by Theorem \ref{thm: cinf cy}, we get a proper Calabi-Yau structure of dimension $n-1$ on $\dd^{b}_{sg}(X)$.
More details will be given in \S\ref{section: singularity category}

Theorem \ref{thm: cinf cy} will follow from a more general study of a system of product structures on Hochschild cochain complexes with arbitrary bimodule coefficients,
as well as pairings and their induced maps on Hochschild chain complexes,
We carry these out throughout \S\ref{section: ainf} and \S\ref{section: neighborhood of infinity},
especially \S\ref{section: products on Hochschild}, \S\ref{section: cy structures}, \S\ref{section: taut pairing}, \S\ref{section: bilinear pairings}, \S\ref{section:res}, \S\ref{section: pairing nondegenerate}.

\subsection{Relation to closed-string theory}

Rabinowitz Floer theory originates from a closed-string theory \cite{CF, CFO},
which assigns to every Liouville manifold $X$ a (co)chain complex
\begin{equation}
RFC^{*}(X)
\end{equation}
called the {\it Rabinowitz Floer complex},
whose cohomology is called the {\it Rabinowitz Floer cohomology} and denoted by
\begin{equation}
RFH^{*}(X).
\end{equation}
The definitions will be given in \S\ref{section: closed-string rfc}.
On the Rabinowitz Floer complex there is also a tautological pairing 
\begin{equation}\label{pairing on rfc}
\langle \cdot, \cdot \rangle_{taut}: RFC^{*}(X) \otimes RFC^{2n-1-*}(X) \to \K
\end{equation}
defined in \eqref{taut pairing on rfc},
which is nondegenerate.

This is also closely related to our theory of Rabinowitz Fukaya categories.
In \cite[Theorem 1.6]{GGV} we showed that when the Liouville manifold $X$ is nondegenerate and has $c_{1}(X)=0$,
then there is an isomorphism
\begin{equation}
\r{HH}_{*-n}(\w(X), \rw(X)) \cong RFH^{*}(X),
\end{equation}
whenever the usual open-closed map to symplectic cohomology
\[
\oc: \r{HH}_{*-n}(\w(X), \w(X)) \to SH^{*}(X)
\]
is an isomorphism.
But our first approach in \cite{GGV} is somewhat indirect; 
in particular, we did not construct a geometric map realizing this isomorphism, 
and therefore could not discuss its further properties.
One additional structure, in comparison to \eqref{pairing on rfc}, is that the Hochschild chain complex $\r{CC}_{*-n}(\w(X), \rw(X))$ also has a canonical pairing
\begin{equation}\label{pairing on wrw}
\langle \cdot, \cdot \rangle_{\sigma}: \r{CC}_{*-n}(\w(X), \rw(X)) \otimes \r{CC}_{(2n-1-*)-n}(\w(X), \rw(X)) \to \K,
\end{equation}
defined in terms of the weak smooth Calabi-Yau structure $\sigma$ on $\w(X)$;
see \S\ref{section: relating pairings} for the detailed construction.
For every exact cylindrical submanifold $L$, 
this induces a pairing
\begin{equation}\label{pairing on wrw for L}
\langle \cdot, \cdot \rangle_{\sigma, L}: \rw^{*-n}(L, L) \otimes \rw^{n-1-*}(L, L) \to \K
\end{equation}
via the canonical inclusion map of chain complexes $i: \rw^{*}(L, L) \to \r{CC}_{*}(\w(X), \rw(X))$.
A priori, this could be different from the pairing on $\rw$ induced from the weak proper Calabi-Yau structure on $\rw$ from Theorem \ref{thm: rw cy}.
To better understand the relationship between the Rabinowitz Fukaya category and the Rabinowitz Floer cohomology,
especially the fruitful algebraic structures on both sides, 
we have the following:

\begin{thm}\label{thm: oc respects pairing}
Floer theory gives rise to a geometrically defined chain map
\begin{equation}\label{wrw oc1}
\oc_{R}: \r{CC}_{*-n}(\w(X), \rw(X)) \to RFC^{*}(X).
\end{equation}
If $X$ is a nondegenerate Liouville manifold with $c_{1}(X) = 0$,
then $\oc_{R}$ is a quasi-isomorphism and respects the pairing \eqref{pairing on wrw} on $\r{CC}_{*-n}(\w(X), \rw(X))$ and the pairing on Rabinowitz Floer cohomology $RFH^{*}(X)$ induced by the tautological pairing $\langle \cdot, \cdot \rangle_{taut}$ \eqref{pairing on rfc}.
\end{thm}

We expect that the induced map on homology 
\begin{equation}
\oc_{R}: \r{HH}_{*-n}(\w(X), \rw(X)) \to RFH^{*}(X),
\end{equation}
is an isomorphism of unital rings,
where the left hand side carries a product dual to the cup product on $\r{HH}^{*}(\w(X), \rw(X))$ defined in \eqref{cup product on cinf} in \S\ref{section: cup product}.
In addition, the full open-closed map \eqref{wrw oc1} should respect the pairings on the two-sides,
where the pairing on the left hand side is induced by the product structure as well as the weak proper Calabi-Yau structure on $\rw$.
There is a chain-level construction on a product on $\r{CC}_{*-n}(\cc, \cinf)$ in \cite{rivera-takeda-wang} which nonetheless requires that $\cc$ have a pre-Calabi-Yau structure (at least up to $4$-truncation),
but that has not been fully established on the wrapped Fukaya category.

\subsection*{Overview of the paper}

In \S\ref{section: ainf}, we provide the necessary homological algebra background for $\ainf$-categories and multi-modules that will be relevant to various key notions throughout the paper.
In particular, we include a discussion of multi-tensor products and Hochschild homology/cohomology with coefficients in multi-modules,
and provide a general framework for product structures on Hochschild cochain complexes with various bimodule coefficients.
In \S\ref{section: neighborhood of infinity}, we recall the definition of the categorical formal punctured neighborhood of infinity, 
and construct our main invariant - the residue on Hochschild chain complex of the categorical formal punctured neighborhood of infinity,
and prove Theorem \ref{thm: cinf cy}.
In \S\ref{section: singularity category}, we include a concise discussion on the Calabi-Yau structure on Orlov's singularity category.
In \S\ref{section: rabinowitz fukaya category}, we review the construction of the Rabinowitz Fukaya category of a Liouville manifold,
as well as some key features of it, among which in particular is the quasi-equivalence between $\rw(X)$ and $\widehat{\w(X)}_{\infty}$. 
In \S\ref{section: residue on rw}, we use the results developed earlier to endow $\rw(X)$ a residue which induces a nondegenerate pairing on cohomology,
and therefore complete the proof of Theorem \ref{thm: rw cy}.
In \S\ref{section: closed-string rfc}, we provide a chain-level definition of closed-string Rabinowitz Floer cohomology compatible with our geometric framework of the Rabinowitz Fukaya category.
Based on this,  in \S\ref{section: open-closed} we explore an open-closed relationship between the Rabinowitz Fukaya category and the closed-string Rabinowitz Floer cohomology,
and define an extension of the usual open-closed map.

\subsection*{Acknowledgements} 

The author wishes to thank Sheel Ganatra for helpful discussions about Calabi-Yau structures on $\ainf$-categories and the relevant notions of residues in complex and algebraic geometry,
and Manuel Rivera for answering some questions related to pre-Calabi-Yau structures and trivializations of the Chern character of the diagonal.

\section{Homological algebra related to $\ainf$-categories}\label{section: ainf}

The goal of this section is to review several important definitions and results surrounding $\ainf$-categories,
and develop algebraic tools using which we will prove our main results Theorem \ref{thm: rw cy} and Theorem \ref{thm: cinf cy}.
In addition, we discuss a system of operations on Hochschild chain and cochain complexes with different bimodule coefficients,
and relate such to pairings between bimodules.
We shall be very careful about signs and the opposite categories.

\subsection{$\ainf$-categories, functors and modules}

We begin by recalling some basic definition about $\ainf$-categories, functors, modules and bimodules,
for which references are by now relatively standard, 
e.g. \cites{lefevre, seidel_book, seidel_ainf, kontsevich-soibelman, FOOO1, fukaya_fun, ganatra},
while mention that there are many more references.

Fix a ground field $\K$ of arbitrary characteristic. 
In this paper all $\ainf$-categories will be small, linear over $\K$, graded by $\Z$, and cohomologically unital.
All unspecified tensor products are understood to be taken over $\K$.
The notation for the grading of an element $x$ is
\[
\deg(x) = |x| \in \Z.
\]

\begin{rem}
Many of the constructions work for $\Z/2$-graded $\ainf$-categories,
including the construction of the Rabinowitz Fukaya category (\S\ref{section: rabinowitz fukaya category}),
the closed-string Rabinowitz Floer cohomology (\S\ref{section: closed-string rfc}),
and the open-closed map (\S\ref{section: open-closed}).
However, the statements on quasi-equivalence/quasi-isomorphisms are not valid without a Calabi-Yau structure on the wrapped Fukaya category,
which only exists when $c_{1}(X) = 0$ and the category is $\Z$-graded.
\end{rem}

Let $\cc$ be a small $\ainf$-category, 
which has a set of objects $\ob \cc$,
and composition maps of degree $2-k$
\begin{equation}\label{ainf maps}
\mu_{\cc}^{k}: \cc(X_{k-1}, X_{k}) \otimes \cdots \otimes \cc(X_{0}, X_{1}) \to \cc(X_{0}, X_{k})[2-k],
\end{equation}
or equivalently the following maps of degree $0$
\begin{equation}\label{ainf maps shifted}
\mu_{\cc}^{k}: \cc[1](X_{k-1}, X_{k}) \otimes \cdots \otimes \cc[1](X_{0}, X_{1}) \to \cc[1](X_{0}, X_{k}),
\end{equation}
satisfying the $\ainf$-equations 
\begin{equation}\label{a-infinity relation}
\sum_{i, j} (-1)^{\maltese_{1;i}} \mu_{\cc}^{k-j+1}(x_{k}, \ldots, x_{i+j+1}, \mu_{\cc}^{j}(x_{i+j}, \ldots, x_{i+1}), x_{i}, \ldots, x_{1}),
\end{equation}
where 
\begin{equation}\label{koszul sign}
\maltese_{1;i} = |x_{1}| + \cdots + |x_{i}| - i
\end{equation}
is the Koszul sign obtained by commuting $\mu_{\cc}^{j}$ to the right end of the expression \eqref{a-infinity relation}.

$\cc$ is said to be {\it strictly unital} if for every $X \in \ob \cc$ there exists $e_{X} \in \cc(X, X)$ such that 
\begin{equation}
\begin{split}
\mu^{1}_{\cc}(e_{X}) &= 0,\\
(-1)^{|x|} \mu^{2}_{\cc}(e_{X}, x) &= \mu^{2}_{\cc}(x, e_{X}) = x, \\
\mu^{k}(\cdots, e_{X}, \cdots) &= 0 \text{ for all } k >2.
\end{split}
\end{equation}
$\cc$ is said to be {\it cohomologically unital} if the cohomology category $H(\cc)$ is unital as an ordinary graded category.
In \S\ref{section: ainf}, we shall assume all $\ainf$-categories in consideration are cohomologically unital.
For notational simplicity, we define
\begin{equation}
\cc(X_{0}, \ldots, X_{k}) := \cc(X_{k-1}, X_{k}) \otimes \cdots \otimes \cc(X_{0}, X_{1}).
\end{equation}

\begin{defn}\label{def: opposite category}
For an $\ainf$-category $\cc$, its {\it opposite category} 
\begin{equation}
\cc^{op}
\end{equation}
is defined to be the $\ainf$-category with $\ob \cc^{op} = \ob \cc$,
morphism spaces
\begin{equation}
\cc^{op}(X, Y) = \cc(Y, X),
\end{equation}
and $\ainf$-structure maps
\begin{equation}
\mu^{k}_{\cc^{op}}(x_{1}, \ldots, x_{k}) = (-1)^{\maltese_{1;k}} \mu^{k}_{\cc}(x_{k}, \ldots, x_{1}).
\end{equation}
\end{defn}

Let $\cc, \dd$ be $\ainf$-categories. 
An $\ainf$-functor $F: \cc \to \dd$ consists of a function $F: \ob \cc \to \ob \dd$,
together with maps
\begin{equation}
F^{k}: \cc(X_{k-1}, X_{k}) \otimes \cdots \otimes \cc(X_{0}, X_{1}) \to \dd(F(X_{0}), F(X_{k}))[1-k]
\end{equation} 
for all $X_{0}, \ldots, X_{k} \in \ob \cc$,
satisfying the $\ainf$-functor equations
\begin{equation}
\begin{split}
& \sum_{s \ge 1} \sum_{i_{1} + \cdots + i_{s} = d} \mu_{\dd}^{s}(F^{i_{s}}(x_{d}, \cdots, x_{i_{1} + \cdots + i_{s-1} + 1}), \cdots, F^{i_{1}}(x_{i_{1}}, \cdots, x_{1})) \\
= & \sum_{i, j} (-1)^{\maltese_{1;i}} F^{d - j + 1}(x_{d}, \cdots, x_{i + j + 1}, \mu_{\cc}^{j}(x_{i + j}, \cdots, x_{i+1}), x_{i}, \cdots, x_{1}).
\end{split}
\end{equation}
The functor $F: \cc \to \dd$ is said to be {\it cohomological unital}, 
if the induced functor on cohomology categories $H(F): H(\cc) \to H(\dd)$ is unital.
In \S\ref{section: ainf}, we shall only consider cohomologically unital $\ainf$-functors.

The set of all $\ainf$-functors from $\cc$ to $\dd$ is denoted by $\fun(\cc, \dd)$, which also forms an $\ainf$-category (\cite{seidel_book}),
in which the degree $g$ part of the morphism space is
\begin{equation}\label{hom in fun}
\hom_{\fun(\cc, \dd)}^{g}(F, G) = \prod_{X_0, \ldots, X_k \in \ob \cc} \hom_{\K}(\cc(X_{k-1}, X_{k}) \otimes \cdots \otimes \cc(X_{0}, X_{1}), \dd(F(X_{0}), G(X_{k}))[g - k]).
\end{equation}
Elements in $\hom_{\fun(\cc, \dd)}(F, G)$ are called $\ainf$-pre-natural transformations from $F$ to $G$.
The $\ainf$-structure maps are
\begin{equation}\label{mu1 in fun}
    \begin{split}
   & \mu^{1}_{\fun(\cc, \dd)}(T)^{k}(x_k, \ldots, x_1)\\
   = &\sum_{r, i} \sum_{\substack{s_{1}, \ldots, s_{r} \ge 0\\s_{1}+\cdots+s_{r}=k}} (-1)^{\dagger} \mu_{\dd}^{r}(G^{s_{r}}(x_{k}, \ldots, x_{s_{1}+\cdots+s_{r-1}+1}), \ldots, \\
    &G^{s_{i+1}}(x_{s_{1}+\cdots+s_{i+1}}, \ldots, x_{s_{1}+\cdots+s_{i}+1}), T^{s_{i}}(x_{s_{1}+\cdots+s_{i}}, \ldots, x_{s_{1}+\cdots+s_{i-1}+1}),\\
    & F^{s_{i-1}}(x_{s_{1}+\cdots+s_{i-1}}, \ldots, x_{s_{1}+\cdots+s_{i-2}+1}) , \ldots, F^{s_{1}}(x_{s_{1}}, \ldots, x_{1})) \\
    &-\sum_{m,n} (-1)^{*_{n}+|T|-1}T^{k-m+1}(x_{k},\ldots, x_{n+m+1}, \mu_{\cc}^{m}(x_{n+m},\ldots,x_{n+1}) x_{n}, \ldots, x_{1}),
    \end{split}
\end{equation}
where
\begin{equation}
\dagger = (|T|-1)(|x_{1}|+\cdots+|x_{s_{1}+\cdots+s_{i-1}}|-s_{1}-\cdots-s_{i-1}),
\end{equation}
and 
\begin{equation}\label{muk in fun}
    \begin{split}
    & \mu^{k}_{\fun(\cc,\dd)}(T_k, \ldots, T_1)^{k}(x_k, \ldots, x_1) \\
    = &\sum_{r, i_{1},\ldots,i_{d}} \sum_{\substack{s_{1},\ldots,s_{r} \ge 0\\s_{1}+\cdots+s_{r}=k}} (-1)^{\circ} \mu^{r}_{\dd}( F_{k}^{s_{r}}(x_{k}, \ldots, x_{k-s_{r}+1}), \ldots, F_{k}^{s_{i_{k}+1}}(\ldots),\\
       &T_{k}^{s_{i_{k}}}(\ldots), \ldots, T_{2}^{s_{i_{2}}}(\ldots), F_1^{s_{i_{2}-1}}(\ldots), \ldots, F_{1}^{s_{i_{1}+1}}(\ldots), \\
    &T_{1}^{s_{i_{1}}}(x_{s_{1}+\cdots+s_{i_{1}}}, \ldots, x_{s_{1}+\cdots+s_{i_{1}-1}+1}), F_0^{s_{i_{1}-1}}(\ldots), \ldots, F_0^{s_{1}}(x_{s_{1}},\ldots,x_{1}))
\end{split}
\end{equation}
where
\begin{equation}\label{signforfunctorcomp}
\circ = \sum_{p=1}^{k}(|T_{p}|-1) \cdot \maltese_{1; s_{1}+\cdots+s_{i_{p-1}}}
\end{equation}
and $1 \le i_{1} \le i_{2} \le \cdots \le i_{k} \le r$.

A left (resp. right) $\ainf$ $\cc$-module $\n$ is an $\ainf$-functor $\cc \to \ch_{\K}$ (resp. $\cc^{op} \to \ch_{\K}$).
The category $\cc \mod = \fun(\cc, \ch_{\K})$ (resp. $\rmod \cc = \fun(\cc^{op}, \ch_{\K})$) is called the category of left (resp. right) modules over $\cc$, or simply $\cc$-modules.
There are Yoneda functors
\begin{equation}
\begin{split}
y^{l}: \cc & \to \cc \mod, \\
y^{r}: \cc & \to \rmod \cc,
\end{split}
\end{equation}
which are both cohomologically fully faithful (since everything is assumed to be cohomologically unital);
see e.g. \cite[\S(2g)]{seidel_book}.
The image of an object $X \in \ob \cc$ under these are denoted by $\Y^{l}_{X}$ and resp. $\Y^{l}_{X}$,
called the left and resp. right Yoneda module of $X$.

\subsection{$A_{\infty}$ mult-functors and multi-modules}

An $\ainf$ $\cc-\dd$-bimodule $\cP$ is an $\ainf$-bilinear functor (or a bi-functor) $\cc \times \dd^{op} \to \ch_{\K}$ ($A_{\infty}$-morphism with two entries in the sense of \cite{lyubashenkomulti}).
To give a systematic treatment of bi-functors, bimodules and even more general objects, 
we recall the notion of $\ainf$ multi-functors by \cite{lyubashenkomulti}; see also \cites{sheridan, fukaya_fun}.

\begin{defn}
Let $m \ge 1$ be a positive integer and $\cc_{1}, \ldots, \cc_{m}$ and $\dd$ be $\ainf$-categories.
An $\ainf$ multi-functor $F: \cc_{1} \times \cdots \times \cc_{m} \to \dd$ consists of a map $F: \prod_{i=1}^{m} \ob \cc_{i} \to \ob \dd$, 
together with a collection of linear maps
\begin{equation}\label{maps of multi-functor}
F^{k_{1}, \ldots, k_{m}}: \bigotimes_{i=1}^{m} \cc_{i}(X_{i, 1}, \ldots, X_{i, k_{i}}) \to \dd(F(X_{1, 1}, \ldots, X_{m, 1}), F(X_{1, k_{1}}, \ldots, X_{m, k_{m}})),
\end{equation}
satisfying the following equations
\begin{equation}
\begin{split}
& \sum_{\substack{s \ge 1\\ i_{j, 1} + \cdots + i_{j, s} = k_{j}}} (-1)^{*}  \mu_{\dd}^{s}(F(x_{1, i_{1,1} + \cdots + i_{1, s}}, \ldots, x_{1, i_{1, 1} + \cdots + i_{1, s-1}+1}; \ldots; x_{m, i_{m,1} + \cdots + i_{m, s}}, \ldots, x_{m, i_{m,1} + \cdots + i_{m, s-1}+1}), \\
& \ldots, F(x_{1, i_{1,1}}, \ldots, x_{1,1}; x_{2, i_{2,1}}, \ldots, x_{2, 1}; \ldots; x_{m, i_{m,1}}, \ldots, x_{m, 1})) \\
= & \sum_{i, j, r}  (-1)^{\maltese} F^{k_{1}, \ldots, k_{i} - j + 1, \ldots, k_{m}}(x_{1, k_{1}}, \ldots; x_{i, k_{i}}, \ldots, \mu_{\cc_{i}}^{j}(x_{i, r+j}, \ldots, x_{i, r+1}), x_{i, r}, \ldots, x_{i, 1}; \ldots; \ldots, x_{m, 1}).
\end{split}
\end{equation}
Here the sign $*$ is the Koszul sign associated to re-ordering the elements in the expression to appear in the standard order
\[
x_{1, k_{1}}, \ldots, x_{1, 1}; \ldots; x_{m, k_{m}}, \ldots, x_{m, 1},
\]
and the sign $\maltese$ is the Koszul sign similar to \eqref{koszul sign}, 
but obtained by commuting $\mu_{\cc_{i}}^{j}$ all the way to the right end of the expression.
\end{defn}

Let $\cc_{1}, \ldots, \cc_{r}$ and $\dd_{1}, \ldots, \dd_{s}$ be $\ainf$-categories.
For notational simplicity we denote them by ordered tuples of $\ainf$-categories
\begin{align}
\overrightarrow{\cc} & = (\cc_{1}, \ldots, \cc_{r}), \\
\overrightarrow{\dd} & = (\dd_{1}, \ldots, \dd_{s}). 
\end{align}

\begin{defn}
An $\ainf$ $\overrightarrow{\cc} - \overrightarrow{\dd}$ multi-module is an $\ainf$ multi-linear functor
\begin{equation}
\A: \prod_{i=1}^{r} \cc_{i} \times \prod_{j=1}^{s} \dd^{op}_{i} \to \ch_{\K}.
\end{equation}
\end{defn}

To make formulas for bimodule structure maps look concise and clean, we introduce some notations.
Let $k, l \ge 0$ be non-negative integers.
For $x_{i} \in \cc(X_{i-1}, X_{i}), x'_{s} \in \dd(Y_{s}, Y_{s-1})$, 
define
\begin{align}
\mathbf{x}_{i; j} & = x_{j} \otimes \cdots \otimes x_{i}, \vspace{0.5cm} 1 \le i \le j \le k, \label{tensor symbol 1} \\
\mathbf{x}'_{s; t} & = x'_{s} \otimes \cdots \otimes x'_{t}, \vspace{0.5cm} 1 \le s \le t \le l, \label{tensor symbol 2}
\end{align}
where if $i > j$ or $k=0$, $\mathbf{x}_{i; j}$ is understood to be empty,
if $s>t$ or $t=0$, $\mathbf{x}'_{s; t}$ is understood to be empty.

\begin{defn}\label{def: diagonal bimodule}
The {\it diagonal bimodule} $\cc_{\D}$ is the canonical $\cc-\cc$ bimodule associate with an $\ainf$-category $\cc$,
with underlying complex $\cc_{\D}(X, Y) = \cc(Y, X)$ and bimodule structure maps,
\begin{equation}
\mu^{k, l}_{\cc_{\D}}: \cc(X_{0}, \ldots, X_{k}) \otimes \cc_{\D}(X_{0}, Y_{0}) \otimes \cc(Y_{l}, \ldots, Y_{0}) \to \cc_{\D}(X_{k}, Y_{l})
\end{equation}
\begin{equation}\label{structure maps for diagonal}
\mu^{k, l}_{\cc_{\D}}(\mathbf{x}_{1;k}, c, \mathbf{x}'_{1;l}) = (-1)^{\maltese'_{1;l}+1} \mu_{\cc}^{k+l+1}(\mathbf{x}_{1;k}, c, \mathbf{x}'_{1;l}),
\end{equation}
where $x_{i} \in \cc(X_{i-1}, X_{i}), y'_{j} \in \cc(Y_{j}, Y_{j-1}) = \cc^{op}(Y_{j-1}, Y_{j})$, and 
\begin{equation}
\maltese'_{1;l} = \sum_{j=1}^{l} |x'_{j}| - l.
\end{equation}
\end{defn}

\begin{defn}\label{def: linear dual bimodule}
The {\it linear dual bimodule} $\cc^{\vee}$ is the $\cc-\cc$ bimodule whose underlying complex is
\begin{equation}
\cc^{\vee}(X, Y) = \hom_{\K}(\cc^{-*}(X, Y), \K)
\end{equation}
with the opposite grading on $\cc(X, Y) = \cc_{\D}(Y, X)$,
and bimodule structure maps induced by the $\ainf$-structure maps of $\cc$ in the following way
\begin{equation}
\mu^{k, l}_{\cc^{\vee}}: \cc(X_{0}, \ldots, X_{k}) \otimes \cc^{\vee}(X_{0}, Y_{0}) \otimes \cc(Y_{l}, \ldots, Y_{0}) \to \cc^{\vee}(X_{k}, Y_{l})
\end{equation}
\begin{equation}\label{structure maps for linear dual}
\mu^{k, l}_{\cc^{\vee}}(\mathbf{x}_{1;k}, f, \mathbf{x}'_{1;l})(w) = (-1)^{\maltese_{1;k} \cdot (|f| + \maltese'_{1;l} + |w| - 1) + |w| -1}f(\mu_{\cc}^{k+l+1}(\mathbf{x}'_{1;l}, w, \mathbf{x}_{1;k})).
\end{equation}
The sign comes from the Koszul sign by moving $\mathbf{x}_{1;k}$ on the left hand side of \eqref{structure maps for linear dual} all the way to the right of the expression, 
where one can apply the action of $\mu^{k+l+1}_{\cc}$ as done on the right hand side of \eqref{structure maps for linear dual}.
\end{defn}

\begin{defn}\label{def: yoneda modules}
For an object $X \in \ob \cc$, its left Yoneda module is $\Y^{l}_{X}$ is the left $\cc$-module with chain complexes $\Y^{l}_{X}(Y) = \cc(X, Y)$ and structure maps
\begin{equation}\label{structure maps for left yoneda}
\mu^{k}_{\Y^{l}_{X}}(\mathbf{x}_{1;k}, w) = \mu_{\cc}^{k+1}(\mathbf{x}_{1;k}, w).
\end{equation}
Its right Yoneda module $\Y^{r}_{X}$ is the right $\cc$-module with chain complexes $\Y^{r}_{X}(Y) = \cc(Y, X)$ and structure maps 
\begin{equation}
\mu^{l}_{\Y^{r}_{X}}(w, \mathbf{x}'_{1;l}) = (-1)^{\maltese'_{1;l}} \mu_{\cc}^{l+1}(w, \mathbf{x}'_{1;l}).
\end{equation}
\end{defn}

\begin{defn}
For a $\cc_{1}-\dd_{1}$ bimodule $\cP$ and a $\cc_{2}-\dd_{2}$ bimodule $\cQ$,
the {\it linear tensor product}
\begin{equation}
\cP \otimes_{\K} \cQ
\end{equation}
is a $(\cc_{1}, \dd_{1}^{op})-(\cc_{2}^{op}, \dd_{2})$ quadmodule,
whose underlying chain complex is 
\begin{equation}
(\cP \otimes_{\K} \cQ)(X_{0}, Y_{0}; X_{1}, Y_{1}) = \cP(X_{0}, Y_{0}) \otimes_{\K} \cQ(X_{1}, Y_{1}),
\end{equation}
where $X_{i} \in \ob \cc_{i}, Y_{i} \in \ob \dd_{i}$,
with multi-module structure maps twisted by the usual rule for Koszul signs.
\end{defn}

\begin{defn}\label{example: linear tensor bimodule}
Let $\cP = \cc^{op}_{\D}$ and $\cQ=(\cc^{op})^{\vee}$,
both $(\cc^{op}, \cc^{op})$-bimodules where $\cc^{op}$ is the opposite category of $\cc$.
We get a $(\cc^{op}, \cc^{op})-(\cc^{op}, \cc^{op})$ quadmodule $\cc^{op}_{\D} \otimes_{\K}(\cc^{op})^{\vee}$,
which by insertion specializes to a $\cc^{op}-\cc^{op}$ bimodule 
\begin{equation}\label{finite hom bimodule}
\Y^{l}_{X_{1}} \otimes_{\K} (\Y^{l}_{X_{0}})^{\vee} = (\cc^{op}_{\D} \otimes_{\K}(\cc^{op})^{\vee}) (X_{1}, -; X_{0}, -),
\end{equation}
whose underlying chain complex, for $Y_{1}, Y_{0} \in \ob \cc^{op}$, is
\begin{equation}
(\Y^{l}_{X_{1}} \otimes_{\K} (\Y^{l}_{X_{0}})^{\vee} ) (Y_{1}, Y_{0}) = \Y^{l}_{X_{1}}(Y_{1}) \otimes_{\K} (\Y^{l}_{X_{0}})^{\vee} (Y_{0}) = \cc^{op}(X_{1}, Y_{1}) \otimes_{\K} \cc^{op}(X_{0}, Y_{0}) = \cc(Y_{1}, X_{1}) \otimes_{\K} \cc(Y_{0}, X_{0}).
\end{equation}
Here the notation $\Y^{l}_{X}$ denotes the left Yoneda $\cc^{op}$-module $\Y^{l}_{X}(K) = \cc^{op}(X, K) = \cc(K, X)$.
The bimodule structure maps are induced from the structure maps on the linear tensor product over $\K$ of Yoneda modules, 
which are
\begin{equation}\label{bimodule maps for finite linear hom}
\mu^{k,l}_{\Y^{l}_{X_{1}} \otimes (\Y^{l}_{X_{0}})^{\vee}}(\mathbf{x}_{1;k}, z \otimes f, \mathbf{x}'_{1;l})(w) = 
\begin{cases}
 (-1)^{|f| - 1} \mu^{1}_{\cc^{op}}(z) \otimes f(w) + (-1)^{|w| - 1} z \otimes f(\mu^{1}_{\cc^{op}}(w))  , & \text{ if } k = l = 0, \\
(-1)^{|f| - 1} \mu^{k+1}_{\cc^{op}}(\mathbf{x}_{1;k}, z) \otimes f(w) , & \text{ if } k>0, l=0,  \\
(-1)^{|w|-1} z \otimes f(\mu^{l+1}_{\cc^{op}}(\mathbf{x}'_{1;l}, w)) , & \text{ if } k=0, l>0, \\
0, & \text{ if } k > 0, l > 0.
\end{cases}
\end{equation}
Here $\mathbf{x}_{1;k}$ are tensors of composable elements in $\cc^{op}$, 
and $\mathbf{x}'_{1;l}$ are tensors of composable elements in $(\cc^{op})^{op} = \cc$.
\end{defn}

\begin{defn}
For a $\cc_{0}-\dd_{0}$ bimodule $\cP$ and a $\cc_{1}-\dd_{1}$ bimodule $\cQ$,
the space of {\it $\K$-linear homomorphisms} from $\cP$ to $\cQ$ is a $(\dd_{1}, \dd_{0}^{op})-(\cc_{0}, \cc_{1}^{op})$ quadmodule,
whose underlying chain complex is 
\begin{equation}
\hom_{\K}(\cP, \cQ)(Y_{1}, Y_{0}; X_{0}, X_{1}) = \hom_{\K}(\cP(X_{0}, Y_{0}), \cQ(X_{1}, Y_{1})).
\end{equation}
\end{defn}

\begin{defn}\label{example: linear hom bimodule}
Let $\cP = \cQ = \cc^{op}_{\D}$, both $(\cc^{op}, \cc^{op})$-bimodules.
We get  a $(\cc, \cc^{op})-(\cc, \cc^{op})$ quadmodule $\hom_{\K}(\cc^{op}_{\D}, \cc^{op}_{\D})$,
which by insertion specializes to a $\cc^{op} - \cc^{op}$-bimodule
\begin{equation}\label{hom bimodule}
\hom_{\K}(\Y^{l}_{X_{0}}, \Y^{l}_{X_{1}}) = \hom_{\K}(\cc^{op}_{\D}, \cc^{op}_{\D}) (-, X_{0}; -, X_{1}).
\end{equation}
Here the right $\cc^{op}$-action acts on the domain of $\hom_{\K}$, i.e., on $\Y^{l}_{X_{0}}$.
In particular, the chain complex underlying this bimodule is defined, for each $Y_{1}, Y_{0} \in \cc^{op}$, as
\begin{equation}
\hom_{\K}(\Y^{l}_{X_{0}}, \Y^{l}_{X_{1}}) (Y_{1}, Y_{0}) = \hom_{\K}(\Y^{l}_{X_{0}}(Y_{0}), \Y^{l}_{X_{1}}(Y_{1})) = \hom_{\K}(\cc^{op}(X_{0}, Y_{0}), \cc^{op}(X_{1}, Y_{1})).
\end{equation}
The bimodule structure maps are
\begin{equation}\label{bimodule maps for linear hom}
\mu^{k, l}_{\hom_{\K}(\Y^{l}_{X_{0}}, \Y^{l}_{X_{1}})}(\mathbf{x}_{1;k}, \phi, \mathbf{x}'_{1;l})(w) = 
\begin{cases}
(-1)^{|w|-1} \phi(\mu^{1}_{\cc^{op}}(w)) + (-1)^{ |w| - 1} \mu^{1}_{\cc^{op}}(\phi(w)), & \text{ if } k=l=0, \\
(-1)^{|w|-1} \mu^{k+1}_{\cc^{op}}(\mathbf{x}_{1;k}, \phi(w)), & \text{ if } k >0, l = 0. \\
(-1)^{|w|-1} \phi(\mu^{l+1}_{\cc^{op}}(\mathbf{x}'_{1;l}, w)) , & \text{ if } k=0, l >0. \\
0, & \text{ if } k > 0, l > 0.
\end{cases}
\end{equation}
 \end{defn}
 
 \begin{rem}
 The signs are particularly simple because the bimodule structures are induced from the left Yoneda module structures,
 whose structure maps are given by \eqref{structure maps for left yoneda} without any signs.
 \end{rem}
 
 An elementary but important property is that there is a canonical map of quadmodules
\begin{equation}\label{quadmodule i map}
i: \cc^{op}_{\D} \otimes_{\K} (\cc^{op})^{\vee} \to \hom_{\K}(\cc^{op}_{\D}, \cc^{op}_{\D}),
\end{equation}
which specializes to a map of bimodules
\begin{equation}\label{bimodule i map}
i:  \Y^{l}_{X_{1}} \otimes(\Y^{l}_{X_{0}})^{\vee} \to \hom_{\K}(\Y^{l}_{X_{0}}, \Y^{l}_{X_{1}}),
\end{equation}
such that for every pair of objects $Y_{1}, Y_{0} \in \ob \cc^{op}$, the map on chain complexes is
\begin{equation}\label{bimodule i map 00}
i^{0, 0}: (\Y^{l}_{X_{1}} \otimes (\Y^{l}_{X_{0}})^{\vee})(Y_{1}, Y_{0}) \to (\hom_{\K}(\Y^{l}_{X_{0}}, \Y^{l}_{X_{1}}))(Y_{1}, Y_{0})
\end{equation}
\begin{equation}\label{sign for i map}
z \otimes f \mapsto i(z \otimes f) = \phi_{z, f}, \text{ where } \phi_{z, f}(w) = f(w) z.
\end{equation}
The higher order terms of \eqref{quadmodule i map} and \eqref{bimodule i map} are all zero.

\begin{lem}\label{lem: bimodule i map}
The maps $i^{r, s}$ with $i^{0, 0}$ being \eqref{bimodule i map 00} and $i^{r, s} = 0$ for $r>0$ or $s>0$ form an $\ainf$-bimodule morphism \eqref{bimodule i map}.
\end{lem}
\begin{proof}
This is a straightforward computation based on the formulas \eqref{bimodule maps for finite linear hom} and \eqref{bimodule maps for linear hom}. 
One sees that the assignment $z \otimes f \mapsto \phi_{z, f}$ intertwines all the bimodule structure maps,
by keeping track of the signs following $|\phi_{z, f}| = |z| + |f|$.
\end{proof}

\begin{rem}
The map $i^{0, 0}$ is the canonical inclusion of subspace of finite rank linear homomorphisms to the space of all linear homomorphisms. 
Under this identification, the bimodule structures for $\Y^{l}_{X_{1}} \otimes(\Y^{l}_{X_{0}})^{\vee}$ and $\hom_{\K}(\Y^{l}_{X_{0}}, \Y^{l}_{X_{1}})$ are in fact exactly the same, 
up to the sign taken care by the sign twist of the map $i$.
In other words, $i$ is the strict inclusion of a subspace which is closed under bimodule structure maps of $\hom_{\K}(\Y^{l}_{X_{0}}, \Y^{l}_{X_{1}})$.
This is why all higher order terms vanish.
\end{rem}

\subsection{Tensor products and duality}

Various notions of tensor products play important roles in studying $\ainf$-bimodules and mult-modules.
The first one is tensor product of two multi-modules over a single side, known as the {\it convolution tensor product}.
To write signs in a consistent manner, we introduce the following notations
\begin{align}\label{koszul sign for cc}
\maltese_{i; j} & = \sum_{s=i}^{j} |x_{s}| - (j-i+1), \\
\maltese'_{i;j} & = \sum_{t=i}^{j} |x'_{t}| - (j-i+1),
\end{align}
where we notice that $\maltese_{1; i}$ is the sign \eqref{koszul sign}.

\begin{defn}\label{def: convolution tensor product}
Define the convolution tensor product of a $\cc-\dd$-bimodule $\cP$ and $\dd-\ee$-bimodule $\cQ$
\begin{equation}
\cP \otimes_{\dd} \cQ
\end{equation}
to be a $\cc-\ee$-bimodule whose underlying chain complex is
\begin{equation}
(\cP \otimes_{\dd} \cQ)(X, Z) = \bigoplus_{\substack{l \\ Y_{0}, \ldots, Y_{l} \in \ob \dd}} \cP(X, Y_{0}) \otimes \dd(Y_{0}, Y_{1}) \otimes \cdots \otimes \dd(Y_{l-1}, Y_{l}) \otimes \B(Y_{l}, Z),
\end{equation}
where $X \in \ob \cc, Y_{0}, \ldots, Y_{l} \in \ob \dd, Z \in \ob \ee$.
The grading of an element $p \otimes \mathbf{x}'_{1;l} \otimes q$ is 
\begin{equation}\label{grading on convolution tensor product}
|p \otimes \mathbf{x}'_{1;l} \otimes q| = |p| + \maltese'_{1;l} + |q|.
\end{equation}

The structure maps $\mu_{\cP \otimes_{\cc} \cQ}^{0, 0}$ are defined as follows.
The $(0, 0)$-th order structure map is
\begin{equation}\label{structure maps for convolution of bimodules 0}
\begin{split}
\mu_{\cP \otimes_{\dd} \cQ}^{0, 0} = & \sum (-1)^{\maltese'_{j+1;m} + |q| - 1} \mu_{\cP}^{0, j}(p, \mathbf{x}'_{1;j}) \otimes \mathbf{x}'_{j+1; m} \otimes q \\
+ & \sum p \otimes \mathbf{x}'_{1; m-i} \otimes \mu_{\cQ}^{i, 0}(\mathbf{x}'_{m-i+1;m}, q) \\
+ & \sum (-1)^{\maltese'_{i+j+1;m} +|q| - 1} p \otimes \mathbf{x}'_{1;i} \otimes \mu_{\dd}^{j}(\mathbf{x}'_{i+1;i+j}) \otimes \mathbf{x}'_{i+j+1;m} \otimes q,
\end{split}
\end{equation}
\begin{equation}\label{structure maps for convolution of bimodules 1}
\mu_{\cP \otimes_{\dd} \cQ}^{k, 0}(\mathbf{x}_{1;k}, p, \mathbf{x}'_{1;m}, q) = \sum (-1)^{\maltese'_{j+1;m} + |q| - 1} \mu_{\cP}^{k, j}(\mathbf{x}_{1;k}, p, \mathbf{x}'_{1;j}) \otimes \mathbf{x}'_{j+1;m} \otimes q.
\end{equation}
\begin{equation}\label{structure maps for convolution of bimodules 2}
\mu_{\cP \otimes_{\dd} \cQ}^{0, l}(p, \mathbf{x}'_{1;m}, q, \mathbf{y}'_{1; l})= \sum p \otimes \mathbf{x}'_{1;m-i} \otimes \mu_{\cQ}^{i, l}(\mathbf{x}'_{m-i+1;m}, q, \mathbf{y}'_{1;l})
\end{equation}
and $\mu_{\cP \otimes_{\dd} \cQ}^{k, l} = 0$ for $k>0, l>0$.
\end{defn}

There is another way of forming a two-sided tensor product,
which gives rise to a chain complex, defined as follows.

\begin{defn}\label{def: bimodule tensor product}
For a $\cc-\dd$-bimodule $\cP$ and a $\dd-\cc$-bimodule $\cQ$, 
their bimodule tensor product
\begin{equation}
\cP \otimes_{\cc-\dd} \cQ
\end{equation}
is defined to be the chain complex whose underlying graded vector space is
\begin{equation}
\begin{split}
 \cP \otimes_{\cc-\dd} \cQ 
 = \bigoplus_{\substack{k \ge 0\\ X_{0}, \ldots, X_{k} \in \ob \cc}} \bigoplus_{\substack{l \ge 0\\ Y_{0}, \ldots, Y_{l} \in \ob \dd}} & \cP(X_{0}, Y_{0}) \otimes \dd(Y_{0}, Y_{1}) \otimes \cdots \otimes \dd(Y_{l-1}, Y_{l}) \\
& \otimes \cQ(Y_{l}, X_{k}) \otimes \cc(X_{k-1}, X_{k}) \otimes \cdots \otimes \cc(X_{0}, X_{1}),
\end{split}
\end{equation}
and whose differential is
\begin{equation}\label{bimodule tensor product differential}
\begin{split}
& d_{\cP \otimes_{\cc-\dd} \cQ}(p \otimes \mathbf{x}'_{1;l} \otimes q \otimes \mathbf{x}_{1;k}) \\
=  & \sum (-1)^{\#_{r,s}} \mu_{\cP}^{k-r, l-s}(\mathbf{x}_{1; k-r}, p, \mathbf{x}'_{1; l-s}) \otimes \mathbf{x}'_{l-s+1; l} \otimes q \otimes \mathbf{x}_{k-r+1; k} \\
+ & \sum (-1)^{\maltese_{1;k} + |q| + \maltese'_{i+j+1; l}} p \otimes \mathbf{x}'_{1; i} \otimes \mu_{\dd}^{j}(\mathbf{x}'_{i+1; i+j}) \otimes \mathbf{x}'_{i+j+1; l} \otimes q \otimes \mathbf{x}_{1; k} \\
+ & \sum (-1)^{\maltese_{1; k-r}} p \otimes \mathbf{x}'_{1; l-s} \otimes \mu_{\cQ}^{s, r}(\mathbf{x}'_{l-s+1; l}, q, \mathbf{x}_{k-r+1; k}) \otimes \mathbf{x}_{1; k-r} \\
+ & \sum (-1)^{\maltese_{1; i}} p \otimes \mathbf{x}'_{1; l} \otimes q \otimes \mathbf{x}_{i+j+1; k} \otimes \mu^{j}_{\cc}(\mathbf{x}_{i+1; i+j}) \otimes \mathbf{x}_{1; i},
\end{split}
\end{equation}
where
\begin{equation}
\#_{r, s} = \maltese_{1; k-r} \cdot (|p| + \maltese'_{1; l} + |q| + \maltese_{k-r+1; k}) + \maltese'_{l-s+1; l} + |q| - 1 + \maltese_{k-r+1; k}.
\end{equation}

\end{defn}

For any $\cc$-bimodule $\cP$, there is a canonical collapse map
\begin{equation}\label{bimodule collapse map}
\mu_{\D, \cP}: \cc_{\D} \otimes_{\cc} \cP \stackrel{\sim}\to \cP,
\end{equation}
defined by 
\begin{equation}\label{bimodule collapse map formula}
\mu^{s; k, l}_{\D, \cP}(x_{k}, \ldots, x_{1}, c, z_{1}, \ldots, z_{s}, p, y_{1}, \ldots, y_{l}) = (-1)^{\circ^{s}_{-l}}\mu_{\cP}^{k+s+1, l}(x_{k}, \ldots, x_{1}, c, z_{1}, \ldots, z_{s}, p, y_{1}, \ldots, y_{l}),
\end{equation}
where
\begin{equation}
\circ^{s}_{-l} = \sum_{i=1}^{s} |z_{i}| - s + |p| - 1+ \sum_{j=1}^{l} |y_{j}| - l.
\end{equation}
There is a similar collapse map
\begin{equation}\label{bimodule collapse map2}
\nu_{\D, \cP}: \cP \otimes_{\cc} \cc_{\D} \stackrel{\sim}\to \cP
\end{equation}
with the Koszul sign from elements to the right hand side of the input from $\cc_{\D}$.

\begin{lem}[{\cite[Proposition 2.2]{ganatra}, \cite[Lemma 2.6]{GGV}}]\label{lem: collapse iso}
The collapse maps $\mu_{\D, \cP}$ \eqref{bimodule collapse map} and $\nu_{\D, \cP}$ \eqref{bimodule collapse map2} are quasi-isomorphisms.
\end{lem} \qed

The {\it bimodule dual} $\cP^{!}$ is defined to be
\begin{equation}
\cP^{!} = \hom_{\cc \bimod \cc}(\cP^{!}, \cc_{\D} \otimes_{\K} \cc_{\D}).
\end{equation}
If $\cP$ is perfect so is $\cP^{!}$, and there is a natural quasi-isomorphism
\begin{equation}
\cP^{!} \otimes_{\cc-\cc} \cQ \simeq \hom_{\cc \bimod \cc}(\cP, \cQ).
\end{equation}
For an $\ainf$-category $\cc$, its {\it inverse dualizing bimodule} is defined to be
\begin{equation}\label{inverse dualizing bimod}
\cc^{!}:= \hom_{\cc \bimod \cc}(\cc_{\D}, \cc_{\D} \otimes_{\K} \cc_{\D}).
\end{equation}

\subsection{Hochschild invariants} 

Let $\cc$ be an $\ainf$-category, and $\cP$ an $\ainf$ $\cc$-bimodule.
The Hochschild chain complex of $\cc$ with coefficients in $\cP$ is defined to be 
\begin{equation}
\r{CC}_{*}(\cc, \cP) = \bigoplus_{\substack{k \ge 0\\X_{0}, \ldots, X_{k} \in \ob \cc}} \cP(X_{0}, X_{k}) \otimes \cc(X_{k-1}, X_{k}) \otimes \cdots \otimes \cc(X_{0}, X_{1})
\end{equation}
with grading
\begin{equation}
|p \otimes \mathbf{x}_{1; k}| = |p| + \maltese_{1; k},
\end{equation}
where $\mathbf{x}_{1; k}$ is \eqref{tensor symbol 1}.
For the ease in keeping track of grading,
we may also write it using a shifted complex:
\begin{equation}\label{cc shifted form}
\r{CC}_{*}(\cc, \cP) = \bigoplus_{\substack{k \ge 0\\X_{0}, \ldots, X_{k} \in \ob \cc}} \cP(X_{0}, X_{k}) \otimes \cc[1](X_{k-1}, X_{k}) \otimes \cdots \otimes \cc[1](X_{0}, X_{1}).
\end{equation}
The differential on $\r{CC}_{*}(\cc, \cP)$ is called the {\it Hochschild chain differential}, 
defined by
\begin{equation}\label{hochschild chain differential}
\begin{split}
d_{\r{CC}_{*}}(p \otimes \mathbf{x}_{1; k}) = & \sum (-1)^{\star_{i}^{j}} \mu^{i, j}_{\cP}(\mathbf{x}_{1; i}, p, \mathbf{x}_{k-j+1; k})\otimes \mathbf{x}_{i+1; k-j} \\
+ & \sum (-1)^{\maltese_{1}^{i}} p \otimes \mathbf{x}_{i+j+1; k} \otimes \mu^{j}_{\cc}(\mathbf{x}_{i+1; i+j}) \otimes \mathbf{x}_{1; i},
\end{split}
\end{equation}
where 
\begin{equation}
\star_{i}^{j} = \maltese_{1;i} \cdot (|p| + \maltese_{i+1; k}) + \maltese_{i+1; k-j}.
\end{equation}

The Hochschild cochain complex of $\cc$ with coefficients in $\cP$ is defined to be
\begin{equation}\label{hochschild cochains}
\r{CC}^{*}(\cc, \cP) = \prod_{\substack{k \ge 0\\X_{0}, \ldots, X_{k} \in \ob \cc}} \hom_{\K}(\cc(X_{k-1}, X_{k}) \otimes \cdots \otimes \cc(X_{0}, X_{1}), \cP(X_{k}, X_{0}))
\end{equation}
with grading
\begin{equation}
\r{CC}^{s}(\cc, \cP) = \prod_{\substack{k \ge 0\\X_{0}, \ldots, X_{k} \in \ob \cc}} \hom_{\K}(\cc(X_{k-1}, X_{k}) \otimes \cdots \otimes \cc(X_{0}, X_{1}), \cP(X_{k}, X_{0})[k-s])
\end{equation}
where $\hom^{0}_{\K}(\cdot, \cdot)$ denotes the subspace of degree zero maps between graded modules.
The differential on $\r{CC}^{*}(\cc, \cP)$ is called the {\it Hochschild cochain differential} is defined as follows.
\begin{equation}\label{hochschild cochain differential}
\begin{split}
d_{\r{CC}^{*}}(\phi)^{k}(\mathbf{x}_{1; k}) = & \sum (-1)^{\dagger_{s}} \mu_{\cP}^{r, s}(\mathbf{x}_{k-r+1; k}, \phi^{k-r-s}(\mathbf{x}_{s+1; k-r}), \mathbf{x}_{1;s}) \\
- & \sum (-1)^{\maltese_{1;i}} \phi^{k-j+1}(\mathbf{x}_{i+j+1; k}, \mu_{\cc}^{j}(\mathbf{x}_{i+1; i+j}), \mathbf{x}_{1; i}).
\end{split}
\end{equation}
where 
\begin{equation}
\dagger_{s} = (|\phi|-1)\maltese_{1; s}.
\end{equation}

\begin{lem}\label{lem: cc computes hom in fun}
Let $F, G: \cc \to \dd$ be $\ainf$-functors. Then we have a strict equality of chain complexes
\begin{equation}
\r{CC}^{*}(\cc, (F, G)^{*}\dd_{\D}) = \hom_{\fun(\cc, \dd)}(F, G),
\end{equation}
where the right hand side is the morphism space in the functor category $\fun(\cc, \dd)$ carrying the differential $\mu^{1}_{\fun(\cc, \dd)}$ \eqref{mu1 in fun}.
\end{lem}
\begin{proof}
This follows immediately from the definition of the morphism space in $\fun(\cc, \dd)$ and the definition of the Hochschild cochain complex \eqref{hochschild cochains} applied to the pullback diagonal bimodule $(F, G)^{*}\dd_{\D}$.
\end{proof}

There is a canonical quasi-isomorphism of chain complexes
\begin{equation}\label{2pt to 1pt}
\cc_{\D} \otimes_{\cc-\cc} \cP \stackrel{\sim}\to \r{CC}_{*}(\cc, \cP) 
\end{equation}
induced by the collapse quasi-isomorphism \eqref{bimodule collapse map}.
There is also a canonical chain map
\begin{equation}\label{1pt to 2pt}
\r{CC}^{*}(\cc, \cP) \to \hom_{\cc \bimod \cc}(\cc_{\D}, \cP)
\end{equation}
This map is a quasi-isomorphism if $\cc$ is cohomologically unital (\cite[Proposition 2.5]{ganatra}).

\begin{lem}\label{lem: cc is bimodule tensor product}
Let $\cP, \cQ$ be $\cc-\cc$-bimodules. 
Then the differential $d_{\r{CC}_{*}(\cc, \cP \otimes_{\cc} \cQ)}$ on $\r{CC}_{*}(\cc, \cP \otimes_{\cc} \cQ)$ defined by \eqref{hochschild chain differential} applied to the bimodule $\cP \otimes_{\cc} \cQ$ with bimodule structure maps \eqref{structure maps for convolution of bimodules 0}, \eqref{structure maps for convolution of bimodules 1}, \eqref{structure maps for convolution of bimodules 2} exactly agrees with the differential $d_{\cP \otimes_{\cc - \cc} \cQ}$ \eqref{bimodule tensor product differential}.
\end{lem}
\begin{proof}
This follows from the definition of the bimodules structure maps for $\cP \otimes_{\cc} \cQ$ in \eqref{structure maps for convolution of bimodules 0}, \eqref{structure maps for convolution of bimodules 1}, \eqref{structure maps for convolution of bimodules 2},
and the general formula \eqref{hochschild chain differential} for differential on the Hochschild chain complex.

First recall from \eqref{grading on convolution tensor product} that an element $p \otimes \mathbf{x}'_{1;l} \otimes q \in \cP \otimes_{\cc} \cQ$ has grading
\[
|p \otimes \mathbf{x}'_{1;l} \otimes q| = |p| + \maltese'_{1;l} + |q|.
\]
Since the bimodule structure maps $\mu^{r, s}$ for $\cP \otimes_{\cc} \cQ$ with $r>0, s>0$ are all zero, 
we find that the first term on the right hand side of \eqref{hochschild chain differential} becomes
\begin{equation}
\begin{split}
& (-1)^{\maltese_{1;k}} \mu_{\cP \otimes_{\cc} \cQ}^{0, 0}(p \otimes \mathbf{x}'_{1;l} \otimes q) \otimes \mathbf{x}_{1;k} \\
+ & \sum_{r>0} (-1)^{\maltese_{1;r} \cdot (|p| + \maltese'_{1, l} + |q| + \maltese_{r+1;k}) + \maltese_{r+1;k}} \mu_{\cP \otimes_{\cc} \cQ}^{r, 0}(\mathbf{x}_{1;r}, p \otimes \mathbf{x}'_{1;l} \otimes q) \otimes \mathbf{x}_{r+1;k} \\
+ & \sum_{s>0} (-1)^{\maltese_{1; k-s}} \mu_{\cP \otimes_{\cc} \cQ}^{0, s}(p \otimes \mathbf{x}'_{1;l} \otimes q, \mathbf{x}_{k-s+1;k}) \otimes \mathbf{x}_{1; k-s}.
\end{split}
\end{equation}
Now using \eqref{structure maps for convolution of bimodules 0}, \eqref{structure maps for convolution of bimodules 1}, \eqref{structure maps for convolution of bimodules 2}, 
we may further write these terms as
\begin{equation}\label{mu 00 on convolution}
\begin{split}
& (-1)^{\maltese_{1;k}} \mu_{\cP \otimes_{\cc} \cQ}^{0, 0}(p \otimes \mathbf{x}'_{1;l} \otimes q) \otimes \mathbf{x}_{1;k} \\
= & \sum (-1)^{\maltese_{1;k} + \maltese'_{j+1;l} + |q| - 1} \mu_{\cP}^{0, j}(p, \mathbf{x}'_{1;j}) \otimes \mathbf{x}'_{j+1; l} \otimes q \\
+ & \sum (-1)^{\maltese_{1;k}}  p \otimes \mathbf{x}'_{1; l-i} \otimes \mu_{\cQ}^{i, 0}(\mathbf{x}'_{l-i+1;l}, q) \\
+ & \sum (-1)^{\maltese_{1;k} + \maltese'_{i+j+1;l}} p \otimes \mathbf{x}'_{1;i} \otimes \mu_{\dd}^{j}(\mathbf{x}'_{i+1;i+j}) \otimes \mathbf{x}'_{i+j+1;m} \otimes q,
\end{split}
\end{equation}
\begin{equation}\label{mu r0 on convolution}
\begin{split}
& \sum_{r>0} (-1)^{\maltese_{1;r} \cdot (|p|  + \maltese'_{1, l} + |q| + \maltese_{r+1;k}) + \maltese_{r+1;k}} \mu_{\cP \otimes_{\cc} \cQ}^{r, 0}(\mathbf{x}_{1;r}, p \otimes \mathbf{x}'_{1;l} \otimes q) \otimes \mathbf{x}_{r+1;k} \\
= & \sum_{r>0, j} (-1)^{\maltese_{1;r} \cdot (|p| + \maltese'_{1, l} + |q| + \maltese_{r+1;k}) + \maltese_{r+1;k} + \maltese'_{j+1;l} + |q| - 1} \mu_{\cP}^{r, j}(\mathbf{x}_{1;r}, p, \mathbf{x}'_{1; j}) \otimes \mathbf{x}'_{j+1;l} \otimes q
\end{split}
\end{equation}
and
\begin{equation}\label{mu 0s on convolution}
\begin{split}
 & \sum_{s>0} (-1)^{\maltese_{1; k-s}} \mu_{\cP \otimes_{\cc} \cQ}^{0, s}(p \otimes \mathbf{x}'_{1;l} \otimes q, \mathbf{x}_{k-s+1;k}) \otimes \mathbf{x}_{1; k-s} \\
 = & \sum_{s>0, i} (-1)^{\maltese_{1; k-s}} p \otimes \mathbf{x}'_{1;l-i} \otimes \mu_{\cQ}^{i, s}(\mathbf{x}'_{l-i+1;l}, q, \mathbf{x}_{k-s+1;k}) \otimes \mathbf{x}_{k-s; 1}.
 \end{split}
 \end{equation}
 
 The second term on the right hand side of \eqref{hochschild chain differential} applied to the bimodule $\cP \otimes_{\cc} \cQ$ is 
 \begin{equation}\label{mu others on convolution}
  \sum (-1)^{\maltese_{1}^{i}} p \otimes \mathbf{x}'_{1;l} \otimes q \otimes \mathbf{x}_{i+j+1; k} \otimes \mu^{j}_{\cc}(\mathbf{x}_{i+1; i+j}) \otimes \mathbf{x}_{1; i}.
 \end{equation}
 Now add the terms \eqref{mu 00 on convolution}, \eqref{mu r0 on convolution}, \eqref{mu 0s on convolution} and \eqref{mu others on convolution}, 
 and we get exactly \eqref{bimodule tensor product differential}.
\end{proof}

\begin{cor}
As chain complexes, $\r{CC}_{*}(\cc, \cP \otimes_{\cc} \cQ) = \cP \otimes_{\cc-\cc} \cQ$,
with differentials carrying the same signs.
\end{cor}

\begin{defn}\label{def: length zero tensors}
We say that a tensor $p \otimes \mathbf{x}'_{1;l} \otimes q \otimes \mathbf{x}_{1;k}$ in the two-sided bar complex $\cP \otimes_{\cc-\cc} \cQ$ has length zero in the $\mathbf{x}$-entries and the $\mathbf{x}'$-entries, if $k = l = 0$.
The following statement is helpful in clarifying some important terms in the formula \eqref{bimodule tensor product differential}.
\end{defn}

\begin{cor}\label{cor: length zero tensors}
The terms in the output of the differential $d_{\r{CC}_{*}(\cc, \cP \otimes_{\cc} \cQ)}$ that have length zero in the $\mathbf{x}$-entries and the $\mathbf{x}'$-entries are
\begin{equation}\label{length zero tensors}
 (-1)^{\maltese_{1;k} \cdot (|p| + \maltese'_{1;l} + |q|) + |q| - 1} \mu_{\cP}^{k, l}(\mathbf{x}_{1;k}, p, \mathbf{x}'_{1;l}) \otimes q 
+  p \otimes \mu_{\cQ}^{l, k}(\mathbf{x}'_{1;l}, q, \mathbf{x}_{1;k}).
\end{equation}
\end{cor}

We can also define Hochschild chain and cochain complexes with coefficients in multi-modules.
Fix an $\ainf$-category $\cc$.
Let 
\begin{align}
\overrightarrow{\cc}_{1} &= (\cc_{1, 1}, \ldots, \cc_{1, r}), \\
\overrightarrow{\cc}_{2} &= (\cc_{2, 1}, \ldots, \cc_{2, s})
\end{align}
be ordered tuples of $\ainf$-categories, 
and $\cP$ a $\overrightarrow{\cc}_{1} - \overrightarrow{\cc}_{2}$ multi-module.
Suppose for some $1 \le i \le r, 1 \le j \le s$ we have $\cc_{1, i} = \cc_{2, j} = \cc$.
Denote by
\begin{align}
\overrightarrow{\cc}_{1}^{(i)} &= (\cc_{1, 1}, \ldots, \widehat{\cc_{1, i}}, \ldots, \cc_{1, r}), \\
\overrightarrow{\cc}_{2}^{(j)} &= (\cc_{2, 1}, \ldots, \widehat{\cc_{2, j}}, \ldots, \cc_{2, s}),
\end{align}
the ordered tuples of $\ainf$-categories with the particular entries removed.

\begin{defn}
The Hochschild chain complex of $\cc$ with cofficients in $\cP$
\begin{equation}
\r{CC}_{*}(\cc, \cP)
\end{equation}
is defined to be the $\overrightarrow{\cc}_{1}^{(i)} - \overrightarrow{\cc}_{2}^{(j)}$ multi-module whose underlying chain complex is
\begin{equation}
\r{CC}_{*}(\cc, \cP)(\vec{X}_{1}^{i, -}; \vec{X}_{2}^{j, -}) = \bigoplus_{\substack{k \ge 0\\ X_{0}, \ldots, X_{k} \in \ob \cc}} \cP(\vec{X}_{1}^{i, +, 0}; \vec{X}_{2}^{j, +, k}) \otimes \cc(X_{k-1}, X_{k}) \otimes \cdots \otimes \cc(X_{0}, X_{1}),
\end{equation}
where  
\begin{align}
&\vec{X}_{1}^{i, -} = (X_{1, 1; 0}, \ldots, X_{1, i-1; 0}, X_{1, i+1; 0}, \ldots, X_{1, r; 0}) \in \prod_{t \neq i} \ob \cc_{1, t}, \\
&\vec{X}_{2}^{j, -} = (X_{2, 1; 0}, \ldots, X_{2, j-1; 0}, X_{1, j+1; 0}, \ldots, X_{2, s; 0}) \in \prod_{t \neq j} \ob \cc_{2, t},
\end{align}
and
\begin{align}
&\vec{X}_{1}^{i, +, 0} = (X_{1, 1; 0}, \ldots, X_{1, i-1; 0}, X_{1, i; 0} = X_{0}, X_{1, i+1; 0}, \ldots, X_{1, r; 0}) \in \prod_{t=1}^{r} \ob \cc_{1, t}, \\
&\vec{X}_{2}^{j, +, k} = (X_{2, 1; 0}, \ldots, X_{2, j-1; 0}, X_{2, j; 0} = X_{k}, X_{1, j+1; 0}, \ldots, X_{2, s; 0}) \in \prod_{t=1}^{s} \ob \cc_{2, t}.
\end{align}

The structure maps are defined as follows.
For 
\begin{align}
& \vec{k}_{i} = (k_{1}, \ldots, k_{i-1}, k_{i+1}, \ldots, k_{r}), \\
& \vec{l}_{j}=(l_{1}, \ldots, l_{j-1}, l_{j+1}, \ldots, l_{s}),
\end{align}
we put 
\begin{align}
&\vec{k}_{i, h} = (k_{1}, \ldots, k_{i-1}, h, k_{i+1}, \ldots, k_{r}), \\
& \vec{l}_{j, h}=(l_{1}, \ldots, l_{j-1}, h, l_{j+1}, \ldots, l_{s}).
\end{align}
The $(\vec{0}; \vec{0})$-th term of the structure maps is the Hochschild differential defined by
\begin{equation}
\mu_{\r{CC}_{*}(\cc, \cP)}^{\vec{0}; \vec{0}}: \r{CC}_{*}(\cc, \cP)(\vec{X}_{1}^{(i)}, \vec{X}_{2}^{(j)}) \to \r{CC}_{*}(\cc, \cP)(\vec{X}_{1}^{(i)}, \vec{X}_{2}^{(j)})[-1]
\end{equation}
by
\begin{equation}
\begin{split}
\mu_{\r{CC}_{*}(\cc, \cP)}^{\vec{0}; \vec{0}}(p \otimes \mathbf{x}_{1; k}) & = \sum (-1)^{\dagger} \mu_{\cP}^{\vec{0}_{i, h}; \vec{0}_{j, l}} (\mathbf{x}_{1; h}; p; \mathbf{x}_{k-l+1; k}) \otimes \mathbf{x}_{1; k-l}\\
& + \sum (-1)^{\maltese_{m}} p \otimes \mathbf{x}_{m+d+1; k} \otimes \mu^{d}(\mathbf{x}_{m+1; m+d}) \otimes \mathbf{x}_{1; m}.
\end{split}
\end{equation}
Here the symbols are
\begin{equation}
\vec{0}_{i, h} = 0, \ldots, 0, h, 0, \ldots, 0, \text{ where } h \text{ appears at the } i\text{-th entry},
\end{equation}
and similarly for $\vec{0}_{j, l}$.
The higher order multi-module structure maps are induced from the original multi-module structure maps for $\cP$,
\begin{equation}
\mu_{\r{CC}_{*}(\cc, \cP)}^{\vec{k}_{i}; \vec{l}_{j}}(\vec{\mathbf{x}}_{1}; p \otimes \mathbf{x}_{1;k}; \vec{\mathbf{x}}_{2}) = \sum (-1)^{\dagger} \mu^{\vec{k}_{i, h}; \vec{l}_{j, l}}(\vec{\mathbf{x}}_{1}^{1,h,+}; p; \vec{\mathbf{x}}_{2}^{k-l+1,k,+}) \otimes \mathbf{x}_{1;k-l}.
\end{equation}
Here the symbol $\vec{\mathbf{x}}_{1}^{1, h, +}$ is obtained by inserting the tensor $\mathbf{x}_{1;h} \in \cc_{1, i}(\ldots)$ into the tensor $\vec{\mathbf{x}}_{1} \in \cc_{1,1}(\ldots) \otimes \cdots \cc_{1,i-1}(\ldots) \otimes \cc_{1,i+1}(\ldots) \otimes \cdots \otimes \cc_{1, r}(\ldots)$ between the $\cc_{1,i-1}$ and $\cc_{1,i+1}$ entries.
\end{defn}

\begin{defn}
The Hochschild cochain complex of $\cc$ with coefficients in $\cP$ 
\begin{equation}
\r{CC}^{*}(\cc, \cP)
\end{equation}
is defined to be the $\overrightarrow{\cc}_{1}^{(i)} - \overrightarrow{\cc}_{2}^{(j)}$ multi-module whose underlying chain complex is
\begin{equation}
(\vec{X}_{1}^{i, -}; \vec{X}_{2}^{j, -}) = \prod_{\substack{k \ge 0\\X_{0}, \ldots, X_{k} \in \ob \cc}} \hom_{\K}(\cc(X_{k-1}, X_{k}) \otimes \cdots \otimes \cc(X_{0}, X_{1}), \cP(\vec{X}_{1}^{i, +, k}; \vec{X}_{2}^{j, +, 0})).
\end{equation}
The structure maps are defined in a way similar to those on the Hochschild chain complex.
\end{defn}

\subsection{Operations on Hochschild complexes}\label{section: products on Hochschild}

The Hochschild cochain complex $\r{CC}^{*}(\cc, \cc) = \r{CC}^{*}(\cc, \cc_{\D})$ with diagonal bimodule coefficient carries has carries a cup product
\begin{equation}\label{cup product cc}
\cup: \r{CC}^{*}(\cc, \cc_{\D}) \otimes \r{CC}^{*}(\cc, \cc_{\D}) \to \r{CC}^{*}(\cc, \cc_{\D}),
\end{equation}
defined by
\begin{equation}
(\phi_{2} \cup \phi_{1})^{k}(\mathbf{x}_{1;k}) = \sum (-1)^{\diamond} \mu^{k-r-s+2}(\mathbf{x}_{j+r+1; k}, \phi_{2}^{r}(\mathbf{x}_{j+1; j+r}), \mathbf{x}_{i+s+1; j}, \phi_{1}^{s}(\mathbf{x}_{i+1; i+s}), \mathbf{x}_{1;i}),
\end{equation}
where
\begin{equation}
\diamond = (|\phi_{2}|-1) \cdot (\maltese_{i+s+1;j} + |\phi_{1}| + \maltese_{1;i}) + (|\phi_{1}|-1) \cdot \maltese_{1;i}.
\end{equation}

The cup product can be regarded as as consequence of the following family of binary operations on Hochschild cochain complexes with arbitrary bimodules
\begin{equation}\label{composition product}
\sqcup: \r{CC}^{*}(\cc, \cP) \otimes \r{CC}^{*}(\cc, \cQ) \to \r{CC}^{*}(\cc, \cP \otimes_{\cc} \cQ).
\end{equation}
Now if $\cP = \cQ$ and there is a morphism of bimodules $\cP \otimes_{\cc} \cP \to \cP$, 
we can define a product on $\r{CC}^{*}(\cc, \cP)$ by composing \eqref{composition product} with the induced map on Hochschild cochain complexes by the bimodule morphism.
The discovery of the map \eqref{composition product} actually dates back to early work of Eilenberg and MacLane \cite{eilenberg-maclane} (in the case of group algebras) and Gerstenhaber \cite{gerstenhaber} (in the case of associative algebras).
For $\phi \in \r{CC}^{*}(\cc, \cP), \psi \in \r{CC}^{*}(\cc, \cQ)$, their product
\begin{equation}
\phi \sqcup \psi
\end{equation}
is defined to be the element of $\r{CC}^{*}(\cc, \cP \otimes_{\cc} \cQ)$ such that for every $\mathbf{x}_{1; m}$,
\begin{equation}\label{composition product formula}
(\phi \sqcup \psi)^{k} (\mathbf{x}_{1; k}) = \sum_{i \le j} \phi^{k-j+1}(\mathbf{x}_{j; k}) \otimes \mathbf{x}_{i+1; j} \otimes \psi^{i}(\mathbf{x}_{1; i}).
\end{equation}

\begin{prop}\label{prop: composition product chain map}
The map $\sqcup$ \eqref{composition product} defined by \eqref{composition product formula} is a chain map.
\end{prop}
\begin{proof}
This is a straightforward computation based on the definitions of the Hochschild cochain differential \eqref{hochschild cochain differential} and the structure maps on the convolution tensor product.
\end{proof}

In the case where $\cP = \cQ = \cc_{\D}$, the quasi-isomorphism of bimodules $\mu_{\D, \cc_{\D}}: \cc_{\D} \otimes_{\cc} \cc_{\D} \to \cc_{\D}$ given by \eqref{bimodule collapse map} induces a map on Hochschild cochain complexes
\[
\mu_{\D, \cc_{\D}, *}: \r{CC}^{*}(\cc, \cc_{\D} \otimes_{\cc} \cc_{\D}) \to \r{CC}^{*}(\cc, \cc_{\D}),
\]
such that the overall composition
\[
\mu_{\D, \cc_{\D}, *} \circ \sqcup: \r{CC}^{*}(\cc, \cc_{\D}) \otimes \r{CC}^{*}(\cc, \cc_{\D}) \to \r{CC}^{*}(\cc, \cc_{\D})
\]
is the usual cup product on Hochschild cohomology with diagonal coefficients.

There is also a cap product between a Hochschild cochain complex with diagonal bimodule coefficient and a Hochschild chain complex with arbitrary bimodule coefficient:
\begin{equation}\label{cap product}
\cap: \r{CC}^{*}(\cc, \cc_{\D}) \otimes \r{CC}_{*}(\cc, \cQ) \to \r{CC}_{*}(\cc, \cQ),
\end{equation}
defined by
\begin{equation}
\phi \cap (q \otimes \mathbf{x}) = \sum (-1)^{\$} \mu_{\cP}^{j, k-j-r+1}(\mathbf{x}_{1; j}, q, \mathbf{x}_{k-i+1; k}, \phi^{r}(\mathbf{x}_{k-i-r+1; k-i}), \mathbf{x}_{s+1; k-i-r}) \otimes \mathbf{x}_{j+1; s}
\end{equation}
where 
\begin{equation}
\$ = (|\phi|-1)\maltese_{1; k-i-r} + \maltese_{1;j} \cdot (|q| + |\phi| + \maltese_{j+1; k}) + \maltese_{j+1;s}.
\end{equation}

Th cap product can also be realized by a more general kind of operations:
\begin{equation}\label{outer cap product}
\sqcap: \r{CC}^{*}(\cc, \cP) \otimes \r{CC}_{*}(\cc, \cQ) \to \r{CC}_{*}(\cc, \cQ \otimes_{\cc} \cP)
\end{equation}
defined by
\begin{equation}\label{outer cap product formula}
\phi \sqcap (q \otimes \mathbf{x}_{1;k}) = \sum (-1)^{\dagger} q \otimes \mathbf{x}_{ i+r+1; k} \otimes \phi^{r}(\mathbf{x}_{i+1; i+r}) \otimes \mathbf{x}_{1;i},
\end{equation}
where the sign
\begin{equation}\label{sign for cap}
\dagger  = (|\phi|-1) \maltese_{1, i}
\end{equation}
comes from standard Koszul sign twist similar to the one in \eqref{hochschild cochain differential}.
In the case where $\cP = \cc_{\D}$, the composition of \eqref{outer cap product} with the map $\r{CC}_{*}(\cc, \cQ \otimes_{\cc} \cc_{\D}) \to \r{CC}_{*}(\cc, \cQ)$ induced by the collapse map $\nu_{\D, \cQ}: \cQ \otimes_{\cc} \cc_{\D} \stackrel{\sim}\to \cQ$ \eqref{bimodule collapse map2} gives the cap product $\cap$ \eqref{cap product}.
From this point of view, there is another cap product
\begin{equation}
\cap: \r{CC}^{*}(\cc, \cP) \otimes \r{CC}_{*}(\cc, \cc_{\D}) \stackrel{\sqcup}\to \r{CC}_{*}(\cc, \cc_{\D} \otimes_{\cc} \cP) \stackrel{\mu_{\D, \cP}}\to \r{CC}_{*}(\cc, \cP),
\end{equation}
but there is no module structure to speak of since $\r{CC}^{*}(\cc, \cP)$ is not generally a differential graded algebra.

Similar to the classical case of associative algebras, 
the cup/cap products \eqref{cup product cc}, \eqref{cap product} makes $\r{HH}_{*}(\cc, \cP)$ a module over the algebra $\r{HH}^{*}(\cc, \cc_{\D})$ (\cite[Proposition 2.4]{ganatra}).
For the operations $\sqcup$ \eqref{composition product} and $\sqcap$ \eqref{outer cap product}, 
they also satisfy a similar associativity relation, 
which leads to the following statement. 

\begin{prop}\label{prop: outer cap product chain map}
The map $\sqcap$ \eqref{outer cap product} defined by \eqref{outer cap product formula} is a chain map.
Moreover, we have
\begin{equation}\label{cup cap compatible}
 (\phi \sqcup \psi) \sqcap \gamma = \phi \sqcap (\psi \sqcap \gamma),
\end{equation}
in $\r{CC}_{*}(\cc, \cP \otimes_{\cc} \cQ \otimes_{\cc} \cR)$,
where $\phi \in \r{CC}^{*}(\cc, \cP), \psi \in \r{CC}^{*}(\cc, \cQ), \gamma \in \r{CC}_{*}(\cc, \cR)$.
\end{prop}
\begin{proof}
The proof is a standard computation using the formulas \eqref{composition product formula} \eqref{outer cap product formula},
in which no application of $\mu$'s are involved.
\end{proof}

\subsection{Calabi-Yau structures}\label{section: cy structures}

In this subsection, we review the definition of Calabi-Yau structures on $\ainf$-categories following \cite{kontsevich-soibelman}. 
See also \cite{brav-dyckerhoff} for definitions on relative Calabi-Yau structures.
We write $\r{CC}_{*}(\cc)_{hS^{1}}$ and resp. $\r{CC}_{*}(\cc)^{hS^{1}}$ for the homotopy orbit complex and resp. homotopy fixed point complex for the (non-unital) Hochschild chain complex for $\cc$,
which compute the positive cyclic homology $\r{HC}_{*}^{+}(\cc)$ and resp. negative cyclic homology $\r{HC}_{*}^{-}(\cc)$.
In this paper, we are not going to prove results concerning strong Calabi-Yau structures,
so we will not specify these chain models for $S^{1}$-complexes;
see \cite{Bourgeois-Oancea, dotsenko-shadrin-vallette, ganatra_cyclic, zhao} for various models.

\begin{defn}\label{def: cy structures}
Let $\cc$ be an $\ainf$-category.
A weak proper Calabi-Yau structure of dimension $n$ is a chain map $\phi: \r{CC}_{*}(\cc) \to \K[-n]$ of degree $-n$,
such that the composition
\begin{equation}
\cc^{*}(X, Y) \otimes \cc^{n-*}(Y, X) \stackrel{\mu_{\cc}^{2}}\to \cc^{n}(Y, Y) \stackrel{i}\to \r{CC}_{n}(\cinf) \stackrel{\phi}\to \K
\end{equation}
induced a nondegenerate pairing on cohomology:
\begin{equation}
H^{*}(\cc(X, Y)) \otimes H^{n-*}(\cc(Y, X)) \stackrel{[\mu_{\cc}^{2}]}\to H^{n}(\cc(Y, Y)) \stackrel{[i]}\to \r{HH}_{n}(\cc) \stackrel{[\phi]}\to \K.
\end{equation}

A strong proper Calabi-Yau structure is a chain map $\tilde{\phi}: \r{CC}_{*}(\cc)_{hS^{1}} \to \K[-n]$ such that $\tilde{\phi} \circ pr$ is weak proper Calabi-Yau structure,
where $\r{CC}_{*}(\cc)_{hS^{1}}$ is the homotopy orbit complex computing positive cyclic homology, 
and $pr: \r{CC}_{*}(\cc) \to \r{CC}_{*}(\cc)_{hS^{1}}$ is the projection to homotopy orbits.

Suppose $\cc$ is smooth. A weak smooth Calabi-Yau structure of dimension $n$ on $\cc$ is a chain map $\sigma: \K[n] \to \r{CC}_{*}(\cc)$, or equivalent a cycle $\sigma \in \r{CC}_{-n}(\cc)$, 
such that the induced map of bimodules
\begin{equation}
\cc^{!} \to \cc_{\D}[-n]
\end{equation}
is a quasi-isomorphism.

A strong smooth Calabi-Yau structure of dimension $n$ on $\cc$ is a chain map $\tilde{\sigma}: \K[n] \to \r{CC}_{*}(\cc)^{hS_{1}}$, 
or equivalently a cycle $\tilde{\sigma} \in \r{CC}_{*}(\cc)^{hS_{1}}[-n]$
such that $\sigma = i(\tilde{\sigma})$ is a weak Calabi-Yau structure of dimension $n$, 
where $i: \r{CC}_{*}(\cc)^{hS_{1}}$ is the inclusion of homotopy fixed points.

\end{defn}

Let $\cP$ be any $\cc$-bimodule. 
Given any Hochschild chain $\sigma \in \r{CC}_{-n}(\cc) = \r{CC}_{-n}(\cc, \cc_{\D})$, we can use the operation $\sqcap$ \eqref{outer cap product} to define a chain map
\begin{equation}\label{capping with cy}
-\cap \sigma: \r{CC}^{*}(\cc, \cP) \stackrel{- \sqcap \sigma}\to \r{CC}_{*-n}(\cc, \cc_{\D} \otimes_{\cc} \cP) \stackrel{\mu_{\D, \cP, *}}\to \r{CC}_{*-n}(\cc, \cP),
\end{equation}
where $\mu_{\D, \cP}$ is the collapse map \eqref{bimodule collapse map}, which is a quasi-isomorphism of bimodules.
The condition for $\sigma$ being a weak smooth Calabi-Yau structure implies:

\begin{lem}\label{lem: capping with cy is iso}
Capping with a weak smooth Calabi-Yau structure $\sigma \in \r{CC}_{-n}(\cc)$ induces a natural quasi-isomorphism
\begin{equation}\label{capping with cy iso}
- \cap \sigma: \r{CC}^{*}(\cc, \cP) \stackrel{\sim}\to \r{CC}_{*-n}(\cc, \cP).
\end{equation}
Here naturality means it commutes with chain maps induced by maps bimodules $\cP \to \cQ$.
It is a chain homotopy equivalence, i.e., has a chain homotopy inverse.
\end{lem}
\begin{proof}
By Definition \ref{def: cy structures}, $\sigma$ induces a quasi-isomorphism of $\ainf$-bimodules $\cc^{!} \stackrel{\sim}\to \cc_{\D}[-n]$.
Then apply the quasi-isomorphisms \eqref{2pt to 1pt} and \eqref{1pt to 2pt}.

Since quasi-isomorphisms of $\ainf$-bimodules are invertible (when $\K$ is a field, \cite{seidel_book}),
we also have a quasi-isomorphism of $\ainf$-bimodules $\cc_{\D}[-n] \to \cc^{!}$, inducing a chain homotopy inverse of \eqref{capping with cy iso}.
\end{proof}

An $\ainf$-functor $F: \cc \to \dd$ induces a chain map
\begin{equation}
F_{*}: \r{CC}_{*}(\cc, \cc) \to \r{CC}_{*}(\dd, \dd),
\end{equation}
which is moreover a morphism of $S^{1}$-complexes.
Thus by \cite[Corollary 3]{ganatra_cyclic}, it induces chains maps
\begin{align}
F_{*, hS^{1}}: \r{CC}_{*}(\cc, \cc)_{hS^{1}} & \to \r{CC}_{*}(\dd, \dd)_{hS^{1}}, \\
F_{*}^{hS^{1}}: \r{CC}_{*}(\cc, \cc)^{hS^{1}} & \to \r{CC}_{*}(\dd, \dd)^{hS^{1}}.
\end{align}

\begin{defn}\label{def: weak smooth cy functor}
Let $(\cc, \sigma), (\dd, \tau)$ be $\ainf$-categories equipped with weak smooth Calabi-Yau structures of the same dimension $n$.
We say that $F: \cc \to \dd$ is a weak smooth Calabi-Yau functor, 
if $[F_{*}(\sigma)] = [\tau] \in \r{HH}_{*}(\cc, \cc)$.

Suppose $\phi, \psi$ admit lifts to strong proper Calabi-Yau structures $\tilde{\phi}, \tilde{\psi}$. 
We say that $F: \cc \to \dd$ is a strong proper Calabi-Yau functor,
if $[F_{*}^{hS^{1}}(\tilde{\psi})] = [\tilde{\phi}] \in \r{HC}_{*}^{+}(\cc)$.

A weak/strong smooth Calabi-Yau functor $F: \cc \to \dd$ is called a weak/strong smooth Calabi-Yau equivalence,
if there exists a weak/strong smooth Calabi-Yau functor $G: \cc \to \dd$ such that $G \circ F$ is homotopic to $\id_{\cc}$ and $F \circ G$ is homotopic to $\id_{\dd}$.
\end{defn}

\begin{defn}\label{def: weak proper cy functor}
Let $(\cc, \phi), (\dd, \psi)$ be $\ainf$-categories equipped with weak proper Calabi-Yau structures of the same dimension $n$.
We say that $F: \cc \to \dd$ is a weak proper Calabi-Yau functor, 
if $[\psi \circ F_{*}] = [\phi] \in \r{HH}_{*}(\cc, \cc)^{\vee}[-n]$.

Suppose $\phi, \psi$ admit lifts to strong proper Calabi-Yau structures $\tilde{\phi}, \tilde{\psi}$. 
We say that $F: \cc \to \dd$ is a strong proper Calabi-Yau functor,
if $[\tilde{\psi} \circ F_{*}^{hS^{1}}] = [\tilde{\phi}] \in \r{HC}_{*}^{+}(\cc)^{\vee}$. 

A weak/strong proper Calabi-Yau functor $F: \cc \to \dd$ is called a weak/strong proper Calabi-Yau equivalence,
if there exists a weak/strong smooth Calabi-Yau functor $G: \cc \to \dd$ such that $G \circ F$ is homotopic to $\id_{\cc}$ and $F \circ G$ is homotopic to $\id_{\dd}$.
\end{defn}

\subsection{A tautological pairing}\label{section: taut pairing}

To better understand a proper Calabi-Yau structure,
we take a closer look at the first piece of information it provides, i.e. a non-degenerate pairing between morphisms spaces of $\cc$.

We first introduce some terminologies. 

\begin{defn}\label{def: bilinear pairings}
Let $A, B, C$ be graded chain complexes over $\K$.
A bilinear map $\pi: A \times B \to C$ is said to be compatible with gradings, 
if there exists $m \in \Z$ such that for every $i, j$ and $a \in A^{i}, b \in B^{j}$, the output $f(a, b)$ has degree $i+j+m$.
When such an $m$ exists, we call $m$ the degree of the bilinear map $\pi$, $\deg(\pi) = m$.

A bilinear map $\pi: A \times B \to C$ is said to be compatible with differentials, if
\begin{equation}\label{graded leibniz}
d_{C}(f(a, b)) = f( (-1)^{|b| - 1} d_{A} a, b) + f(a, d_{B} b)
\end{equation}
for all $a \in A, b \in B$.

When $C = \K$ the ground field, a bilinear map $\pi: A \times B \to \K$ that is compatible with gradings and differentials is called a bilinear pairing between $A$ and $B$. 
If furthermore $A=B$, a bilinear pairing between $A$ and $B$ is simply called a bilinear pairing on $A$.
\end{defn}

The most basic yet important example is the following:

\begin{example}
Let $A^{\vee} = \hom_{\K}(A^{-*}, \K)$ be the linear dual chain complex of $A$, 
with its grading as indicated, i.e., $(A^{\vee})^{i} = \hom_{\K}(A^{-i}, \K)$.
The evaluation pairing 
\begin{equation}\label{evaluation on a}
ev: A^{\vee} \times A \to \K
\end{equation}
defined by 
\begin{equation}
ev(f, a) = f(a)
\end{equation}
is a bilinear pairing between $A^{\vee}$ and $A$ of degree $0$.
It is clear by definition that the pairing has degree $0$,
so it remains to check the condition \eqref{graded leibniz}.
We following the sign convention for $\ainf$-categories when defining the differential on chain complexes,
in which case we have $\mu^{1}_{A} = -d_{A}$, and
\begin{equation}
 (\d_{A^{\vee}}f)(a) = (-1)^{|a| - 1}f(\mu^{1}_{A}(a)) = (-1)^{|a|} f(d_{A} a),
\end{equation}
following the pattern for the structure maps \eqref{structure maps for linear dual} for the linear dual bimodule $\cc^{\vee}$.
So
\begin{equation}\label{skew symmetry for ev}
ev((-1)^{|a|-1} \d_{\A^{\vee}}f, a) + ev(f, d_{A}a) = (-1)^{|a|-1+|a|} f(d_{A}a) + f(d_{A}a) = 0.
\end{equation}
which is \eqref{graded leibniz} for the bilinear map $ev$.

The other pairing
\begin{equation}\label{evaluation on a opp}
\tilde{ev}: A \times A^{\vee} \to \K
\end{equation}
is defined by 
\begin{equation}
\tilde{ev}(a, f) = (-1)^{(|a|-1)(|f|-1)} f(a) = (-1)^{|a||f| + |a| + |f|} f(a).
\end{equation}
\end{example}

Using the above map $ev$,
we can define a tautological pairing as follows.
Let $C$ be a graded chain complex over $\K$,
which is either finite or infinite dimensional.
Consider the graded chain complex
\begin{equation}
V = C^{\vee}[1-n] \oplus C.
\end{equation}
We define a bilinear pairing on $V$ of degree $1-n$
\begin{equation}\label{taut pairing}
\langle \cdot, \cdot \rangle_{taut}: V \otimes V \to \K[1-n],
\end{equation}
by
\begin{equation}\label{taut pairing formula}
 \langle (f_{1}, x_{1}), (f_{2}, x_{2}) \rangle_{taut} =   f_{1}(x_{2}) + (-1)^{1+(|x_{1}|-1)(|x_{2}|-1)} f_{2}(x_{1}) =  f_{1}(x_{2}) + (-1)^{|x_{1}||x_{2} + |x_{1}| + |x_{2}|} f_{2}(x_{1}).
\end{equation}
Noting that the grading on the direct sum complex $V$ is determine by the $C$-factor, we see that the pairing is graded symmetric
\begin{equation}
 \langle (f_{1}, x_{1}), (f_{2}, x_{2}) \rangle_{taut} = (-1)^{1+(|x_{1}|-1)(|x_{2}|-1)}  \langle (f_{2}, x_{2}), (f_{1}, x_{1}) \rangle_{taut}.
\end{equation}
In fact, the following is almost straightforward from the definition:

\begin{lem}\label{lem: taut nondegenerate}
The pairing $\langle \cdot, \cdot \rangle_{taut}$ \eqref{taut pairing} is nondegenerate.
\end{lem}
\begin{proof}
Suppose 
\begin{equation}\label{pairing zero}
f_{1}(x_{2}) + (-1)^{|x_{1}||x_{2} + |x_{1}| + |x_{2}|}  f_{2}(x_{1}) = 0
\end{equation}
for all $(f_{1}, x_{1})$. 
By fixing $x_{1} = 0$ and considering all $f_{1} \in V^{\vee}[1-n]$, we get from \eqref{pairing zero} $f_{1}(x_{2}) = 0, \forall f_{1}$ so $x_{2} = 0$. 
By fixing $f_{1} = 0$ and considering all $x_{1} \in V$, we get $f_{2}(x_{1}) = 0, \forall x_{1}$ so $f_{2} = 0$.
\end{proof}

More generally, we can consider for a given $\ainf$-category $\cc$ the following bimodule
\begin{equation}
\B = \cc^{\vee}[1-n] \oplus \cc_{\D}.
\end{equation}
$\B$ always has a tautological pairing, 
\begin{equation}\label{bimodule taut pairing}
\langle \cdot, \cdot \rangle_{taut}: \B(X, Y) \times \B(Y, X) \to \K.
\end{equation}
for all objects $X, Y \in \ob \cc$, defined in the same manner as \eqref{taut pairing}.
More generally, we can allow a `twist' on the bimodule structure on $\B$ by viewing $\B$ as a mapping cone of a map $\cc^{\vee}[-n] \to \cc_{\D}$
which comes from a categorical copairing
\begin{equation}
\K \to \cc(X, Y) \otimes \cc(Y, X)[n].
\end{equation}
When $\cc$ is proper, a cyclic $\ainf$-category structure on $\B$ with respect to this pairing is equivalent to the data of a pre-Calabi-Yau structure on $\cc$ (\cite{kontsevich-takeda-vlassopoulos}).
However, since we will be mainly interested in the non-proper case, 
hoping to find a strictly cyclic $\ainf$-structure on $\B$ is requiring too much symmetry for those categories $\cc$ appearing naturally in Floer theory, for which current geometric constructions are insufficient.
Instead, we shall proceed to \S\ref{section: neighborhood of infinity},
 to find $\ainf$-structures on a homotopy replacement of $\B$,
which possesses a weak proper Calabi-Yau structure.

\subsection{Canonical pairing systems}\label{section: bilinear pairings}

In this subsection, we discuss a generalization of the pairings \eqref{evaluation on a},
\eqref{taut pairing} and \eqref{bimodule taut pairing},
and its relation to Hochschild invariants.
 
 \begin{defn}\label{def: canonical pairing system}
 Let $\cP, \cQ$ be $\cc-\cc$-bimodules.
 A canonical pairing system $\pi$ between $\cP$ and $\cQ$ is a family of pairings of degree $m$
 \begin{equation}\label{bilinear pairing}
\pi_{X_{0}, X_{1}}: \cP(X_{0}, X_{1}) \times \cQ(X_{1}, X_{0}) \to \K[m],
\end{equation}
 one for every pair of objects $X_{0}, X_{1}$,
 such that for all $\mathbf{x}_{1; k}, \mathbf{x}'_{1; l}, p, q$ we have
 \begin{equation}\label{condition for pairing system}
(-1)^{\maltese_{1;k} \cdot(|p| + \maltese'_{1; l} + |q|) + |q| - 1} \pi(\mu_{\cP}^{k, l}(\mathbf{x}_{1; k}, p, \mathbf{x}'_{1; l}), q) + \pi(p, \mu^{l, k}_{\cQ}(\mathbf{x}'_{1; l}, q, \mathbf{x}_{1; k})) = 0.
 \end{equation}
 \end{defn}
 
 \begin{rem}\label{rem: sign for pairing system}
 The first set of equations in the condition \eqref{condition for pairing system} is 
 \begin{equation}\label{condition for pairing system 0}
 \pi((-1)^{|q| - 1} \mu_{\cP}^{0, 0}(p), q) + \pi(p, \mu_{\cQ}^{0, 0}(q)) = 0.
 \end{equation}
 In view that the $(0, 0)$-th order map on an $\ainf$-bimodule is a differential,
 we can thus understand the condition \eqref{condition for pairing system} as a generalization of the usual skew-symmetry relation for pairings between chain complexes.
 \end{rem}
 
 Suppose $\pi$ is a canonical pairing system between $\cP$ and $\cQ$ of degree $m$.
We define an induced map
\begin{equation}\label{pairing induced map}
\pi_{*}: \cP \otimes_{\cc-\cc} \cQ = \r{CC}_{*}(\cc, \cP \otimes_{\cc} \cQ) \to \K[m]
\end{equation}
as follows.
\begin{equation}\label{pairing induced map formula}
\pi_{*}(p \otimes \mathbf{x}'_{1;l} \otimes q \otimes \mathbf{x}_{1; k})=
\begin{cases}
\pi(p, q), & \text{ if } k=l=0, \\
0, & \text{ if } k>0 \text{ or } l > 0.
\end{cases}
\end{equation}

\begin{prop}\label{prop: pairing induced map}
The map $\pi_{*}: \r{CC}_{*}(\cc, \cP \otimes_{\cc} \cQ) \to \K$ \eqref{pairing induced map} defined by \eqref{pairing induced map formula} is a chain map of degree $m$.
\end{prop}
\begin{proof}
Consider an element $p \otimes \mathbf{x}'_{1;l} \otimes q \otimes \mathbf{x}_{1; k}$.
We must show that 
\[
\pi_{*}(d_{\r{CC}_{*}(\cc, \cP \otimes_{\cc} \cQ)} (p \otimes \mathbf{x}'_{1;l} \otimes q \otimes \mathbf{x}_{1; k}))=0,
\]
where the Hochschild chain differential agrees with the differential $d_{\cP \otimes_{\cc-\cc} \cQ}$ by Lemma \ref{lem: cc is bimodule tensor product}.
Since the map $\pi_{*}$ defined by \eqref{pairing induced map formula} sends all tensors that have non zero length in the $\mathbf{x}$-entries and the $\mathbf{x}'$-entries,
it suffices to consider the terms in the output of $d_{\r{CC}_{*}(\cc, \cP \otimes_{\cc} \cQ)} (p \otimes \mathbf{x}'_{1;l} \otimes q \otimes \mathbf{x}_{1; k})$ that have zero length in the $\mathbf{x}'$-entries and the $\mathbf{x}$-entries in the sense of Definition \ref{def: length zero tensors}.
By Corollary \ref{cor: length zero tensors}, 
these terms are \eqref{length zero tensors}.
Apply $\pi_{*}$ to the sum of these terms, in view of its definition \eqref{pairing induced map formula}, we obtain
\begin{equation}
 (-1)^{\maltese_{1;k} \cdot(|p| + \maltese'_{1; l} + |q|) + |q| - 1}  \pi(\mu_{\cP}^{k, l}(\mathbf{x}_{1; k}, p, \mathbf{x}'_{1; l}), q) +\pi(p, \mu^{l, k}_{\cQ}(\mathbf{x}'_{1; l}, q, \mathbf{x}_{1; k}))
 \end{equation}
which is exactly equal to zero because of the condition \eqref{condition for pairing system}.
\end{proof}

The first main example of a canonical pairing system is the following:

\begin{lem}\label{lem: evaluation pairing canonical}
The evaluation pairings
\begin{equation}\label{evaluation pairing c}
ev:  \cc^{\vee}(X_{0}, X_{1}) \times  \cc_{\D}(X_{1}, X_{0}) =  \cc(X_{0}, X_{1})^{\vee} \times \cc(X_{0}, X_{1}) \to \K,
\end{equation}
defined by 
\begin{equation}
ev(f, w) = f(w)
\end{equation}
 form a canonical pairing system of degree $0$.

The other pairings
\begin{equation}\label{opp evaluation pairing c}
\tilde{ev}: \cc_{\D}(X_{1}, X_{0}) \times \cc^{\vee}(X_{0}, X_{1}) = \cc(X_{0}, X_{1}) \times \cc(X_{0}, X_{1})^{\vee} \to \K
\end{equation}
defined by 
\begin{equation}
\tilde{ev}(w, f)  = (-1)^{|w||f| + |w| + |f|} f(x)
\end{equation}
also form a canonical pairing system of degree $0$.
\end{lem}
\begin{proof}
Recall that the bimodule structure maps for $\cc^{\vee}$ are defined by \eqref{structure maps for linear dual}, so we have
\begin{equation}
\begin{split}
&    (-1)^{\maltese_{1;k} \cdot(|f| + \maltese'_{1; l} + |w|) + |w| - 1} ev(\mu_{\cc^{\vee}}^{k, l}(\mathbf{x}_{1; k}, f, \mathbf{x}'_{1; l}), w) \\
= & (-1)^{\maltese_{1;k} \cdot(|f| + \maltese'_{1; l} + |w|) + |w| - 1}  (-1)^{\maltese_{1;k} \cdot(|f| + \maltese'_{1; l} + |w|-1) + |w| - 1} f(\mu^{k+l+1}_{\cc}(\mathbf{x}'_{1;l}, w, \mathbf{x}_{1;k})) \\
= & (-1)^{\maltese_{1;k}} f(\mu^{k+l+1}_{\cc}(\mathbf{x}'_{1;l}, w, \mathbf{x}_{1;k})).
\end{split}
\end{equation}
The structure maps for $\cc_{\D}$ are defined by \eqref{structure maps for diagonal}, so we have
\begin{equation}
ev(f, \mu_{\cc_{\D}}^{l, k}(\mathbf{x}'_{1;l}, w, \mathbf{x}_{1;k})) = (-1)^{\maltese_{1;k} +1} f(\mu_{\cc}^{k+l+1}(\mathbf{x}_{1;k}, w, \mathbf{x}'_{1;l})).
\end{equation}
It follows that 
\[
(-1)^{\maltese_{1;k} \cdot(|f| + \maltese'_{1; l} + |w|) + |w| - 1}  ev(\mu_{\cc^{\vee}}^{k, l}(\mathbf{x}_{1; k}, f, \mathbf{x}'_{1; l}), w)
+ ev(f, \mu_{\cc_{\D}}^{l, k}(\mathbf{x}'_{1;l}, w, \mathbf{x}_{1;k})) = 0,
\]
which is \eqref{condition for pairing system}.

For the other pairing $\tilde{ev}$ \eqref{opp evaluation pairing c}, we first compute
\begin{equation}
\mu_{\cc_{\D}}^{k, l}(\mathbf{x}_{1;k}, w, \mathbf{x}'_{1;l}) = (-1)^{\maltese'_{1;l} + 1} \mu_{\cc}^{k+l+1}(\mathbf{x}_{1;k}, w, \mathbf{x}'_{1;l}),
\end{equation}
and
\begin{equation}
\mu_{\cc^{\vee}}^{l, k}(\mathbf{x}'_{1;l}, f, \mathbf{x}_{1;k})(w)  = (-1)^{\maltese'_{1;l}(|f| + \maltese_{1;k} + |w| - 1) + |w| - 1} f(\mu_{\cc}^{l+k+1}(\mathbf{x}'_{1;l}, w, \mathbf{x}_{1;k})).
\end{equation}
Note that the degree satisfy
\begin{align}
& |\mu_{\cc_{\D}}^{k, l}(\mathbf{x}_{1;k}, w, \mathbf{x}'_{1;l})| = |w| + \maltese_{1;k} + \maltese'_{1;l} + 1, \\
& \mu_{\cc^{\vee}}^{l, k}(\mathbf{x}'_{1;l}, f, \mathbf{x}_{1;k}) = |f| + \maltese'_{1;l} + \maltese_{1;k} + 1.
\end{align}
We then compute
\begin{equation}\label{opp ev left}
\begin{split}
& (-1)^{\maltese_{1;k}(|w| + \maltese'_{1;l} + |f| - 1) + |f| -  1} \tilde{ev}(\mu_{\cc_{\D}}^{k, l}(\mathbf{x}_{1;k}, w, \mathbf{x}'_{1;l}), f) \\
= & (-1)^{\maltese_{1;k}(|w| + \maltese'_{1;l} + |f| - 1) + |f| - 1 + \maltese'_{1;l} + 1 + (|w| + \maltese_{1;k} + \maltese'_{1;l} + 1)|f| + |w| + \maltese_{1;k} + \maltese'_{1;l} + 1 + |f|} f(\mu_{\cc}^{k+l+1}(\mathbf{x}_{1;k}, w, \mathbf{x}'_{1;l})) \\
= & (-1)^{\maltese_{1;k}|w| + \maltese_{1;k} \maltese'_{1;l} + \maltese'_{1;l}|f| + |w||f| + |w| + |f| + \maltese_{1;k} + 1} f(\mu_{\cc}^{k+l+1}(\mathbf{x}_{1;k}, w, \mathbf{x}'_{1;l})),
\end{split}
\end{equation}
and
\begin{equation}\label{opp ev right}
\begin{split}
& \tilde{ev}(w, \mu_{\cc^{\vee}}^{l, k}(\mathbf{x}'_{1;l}, f, \mathbf{x}_{1;k})) \\
= & (-1)^{\maltese'_{1;l}(|f| + \maltese_{1;k} + |w| - 1) + |w| - 1 + |w|(|f| + \maltese'_{1;l} + \maltese_{1;k} + 1) + |w| +  |f| + \maltese'_{1;l} + \maltese_{1;k} + 1} f(\mu_{\cc}^{k + l + 1}(\mathbf{x}_{1;k}, w, \mathbf{x}'_{1;l}))\\
= & (-1)^{\maltese'_{1,l}|f| + \maltese_{1;k} \maltese'_{1;l} + |w| + |w||f|  + \maltese_{1;k}|w| + |f| + \maltese_{1;k}}
\end{split}
\end{equation}
So \eqref{opp ev left} and \eqref{opp ev right} add up to zero.
\end{proof}

It is illustrate to see the case where $k=0$, in which case the equation is
\[
(-1)^{|w| - 1} ev(\mu^{0, 0}_{\cc^{\vee}}(f), w) + ev(f, \mu^{0, 0}_{\cc_{\D}}) = (-1)^{|w| - 1} (-1)^{|w| - 1} f(\mu^{1}_{\cc}(w)) + (-1)^{-1} f(\mu^{1}_{\cc}(w) = 0,
\]
which is \eqref{skew symmetry for ev}.

By Proposition \ref{prop: pairing induced map}, the canonical pairing systems $ev$ \eqref{evaluation pairing c} and $\tilde{ev}$ \eqref{opp evaluation pairing c} induce chain maps of degree zero:
\begin{equation}
ev_{*}: \cc^{\vee} \otimes_{\cc-\cc} \cc_{\D} = \r{CC}_{*}(\cc, \cc^{\vee} \otimes_{\cc} \cc_{\D}) \to \K,
\end{equation}
and
\begin{equation}
\tilde{ev}_{*}: \cc_{\D} \otimes_{\cc-\cc} \cc^{\vee} = \r{CC}_{*}(\cc, \cc_{\D} \otimes_{\cc} \cc^{\vee}) \to \K.
\end{equation}

A similar example is the following:

\begin{cor}\label{cor: taut pairing canonical}
The tautological pairings $\langle \cdot, \cdot \rangle_{taut}$ \eqref{bimodule taut pairing} form a canonical pairing system,
which induces a map of chain complexes 
\begin{equation}
\B \otimes_{\cc-\cc} \B = \r{CC}_{*}(\cc, \B \otimes_{\cc} \B) \to \K[1-n]
\end{equation}
of degree $1-n$, where $\B = \cc^{\vee}[1-n] \oplus \cc_{\D}$.
\end{cor}
\begin{proof}
Since the pairing is only between $\cc^{\vee}[1-n](X_{0}, X_{1})$ from the first factor of the left hand side of \eqref{pre pairing on Z} and $\cc_{\D}(X_{1}, X_{0})$ from the second factor,
as well as between $\cc_{\D}(X_{0}, X_{1})$ from the first factor and $\cc^{\vee}[1-n](X_{1}, X_{1})$ from the second factor,
the result follows from Lemma \ref{lem: evaluation pairing canonical}.
\end{proof}

\begin{example}
If $(\cc, \mu^{k}, \langle \cdot, \cdot \rangle_{cyc})$ is a cyclic $\ainf$-category,
then the pairing defines a canonical pairing system between $(\cc)_{\D}$ and itself.
The sign in the equation \eqref{condition for pairing system} comes from the graded symmetry of the cyclic pairing $\langle \cdot, \cdot \rangle_{cyc}$.

However, one should be careful about the difference between this example and Corollary \ref{cor: taut pairing canonical}:
it does not say that $\B = \cc^{\vee}[1-n] \oplus \cc_{\D}$ itself carries a cyclic $\ainf$-structure, 
that pairing is between $\cc-\cc$-bimodules, and we do not even have an $\ainf$-category structure on $\B$
\end{example}

More importantly, we consider the $\cc^{op}-\cc^{op}$-bimodules $\Y^{l}_{X_{1}} \otimes_{\K} (\Y^{l}_{X_{0}})^{\vee}$ \eqref{finite hom bimodule} and $\hom_{\K}(\Y^{l}_{X_{1}}, \Y^{l}_{X_{0}})$ \eqref{hom bimodule}.

\begin{defn}\label{def: trace pairing}
For each pair of objects $Y_{1}, Y_{0} \in \ob \cc^{op}$, we define a pairing
\begin{equation}\label{hom pairing}
\pi^{ev}_{Y_{1}, Y_{0}}: (\Y^{l}_{X_{1}} \otimes_{\K} (\Y^{l}_{X_{0}})^{\vee})(Y_{1}, Y_{0}) \times (\hom_{\K}(\Y^{l}_{X_{1}}, \Y^{l}_{X_{0}}))(Y_{0}, Y_{1}) \to \K
\end{equation}
by the formula
\begin{equation}\label{hom pairing formula}
\pi^{ev}_{Y_{1}, Y_{0}}( z \otimes f, \psi) = (-1)^{|z|} f(\psi(z)).
\end{equation}
In the standard terms, we call this the composition trace pairing,
which comes from the supertrace of the composition of a finite-rank linear map (in this case rank-one) and a general linear map with opposite degrees.
It is clearly nondegenerate.
\end{defn}

\begin{lem}\label{lem: hom pairing canonical}
The pairings $\pi^{ev}_{Y_{1}, Y_{0}}$ form a canonical pairing system of degree $0$ between the $\cc^{op}-\cc^{op}$-bimodules $\Y^{l}_{X_{1}} \otimes_{\K} (\Y^{l}_{X_{0}})^{\vee}$ and $\hom_{\K}(\Y^{l}_{X_{1}}, \Y^{l}_{X_{0}})$.
\end{lem}
\begin{proof}
We must verify the condition \eqref{condition for pairing system} for the pairings $\pi^{ev}_{Y_{1}, Y_{0}}$.
Since the formula \eqref{hom pairing formula} is not sensitive to the objects $Y_{1}, Y_{0}$,
we shall omit the subscripts $Y_{1}, Y_{0}$ throughout the proof for clearness of notation.
By definitions of the bimodule structure maps \eqref{bimodule maps for finite linear hom} and \eqref{bimodule maps for linear hom},
the $\mu^{k, l}$ maps all vanish for $k, l$ both positive.
So the only conditions to check are
\begin{equation}\label{pairing condition ev 0}
(-1)^{|\psi| - 1} \pi^{ev}( \mu_{\Y^{l}_{X_{1}} \otimes_{\K} (\Y^{l}_{X_{0}})^{\vee}}^{0, 0}(z \otimes f), \psi) 
+ \pi^{ev}(z \otimes f, \mu_{\hom_{\K}(\Y^{l}_{X_{1}}, \Y^{l}_{X_{0}})}^{0, 0}(\psi)) = 0,
\end{equation}
\begin{equation}\label{pairing condition ev 1}
(-1)^{\maltese_{1;k} \cdot (|f|+|z|+|\psi|) + |\psi| - 1} \pi^{ev}(\mu_{\Y^{l}_{X_{1}} \otimes_{\K} (\Y^{l}_{X_{0}})^{\vee}}^{k, 0}(\mathbf{x}_{1;k}, z \otimes f), \psi) 
+ \pi^{ev}(z \otimes f, \mu_{\hom_{\K}(\Y^{l}_{X_{1}}, \Y^{l}_{X_{0}})}^{0, k}(\psi, \mathbf{x}_{1;k})) = 0,
\end{equation}
\begin{equation}\label{pairing condition ev 2}
\pi^{ev}(\mu_{\Y^{l}_{X_{1}} \otimes_{\K} (\Y^{l}_{X_{0}})^{\vee}}^{0, l}(z \otimes f, \mathbf{x}'_{1;l}), \psi) 
+ \pi^{ev}(z \otimes f, \mu_{\hom_{\K}(\Y^{l}_{X_{1}}, \Y^{l}_{X_{0}})}^{l, 0}(\mathbf{x}'_{1;l}, \psi)) = 0.
\end{equation}

First we check \eqref{pairing condition ev 0},
in which the two terms are possibly nonzero only if the degrees satisfy
\begin{equation}\label{pairing condition ev 0 degree}
|z| + |f| + |\psi| + 1 = 0.
\end{equation}
By definition we have
\begin{align}
& \mu_{\Y^{l}_{X_{1}} \otimes_{\K} (\Y^{l}_{X_{0}})^{\vee}}^{0, 0}(z \otimes f)(w) =   (-1)^{|f| - 1} \mu^{1}_{\cc^{op}}(z) \otimes f(w) + (-1)^{|w| - 1} z \otimes f(\mu^{1}_{\cc^{op}}(w)) , \\
& \mu_{\hom_{\K}(\Y^{l}_{X_{1}}, \Y^{l}_{X_{0}})}^{0, 0}(\psi)(w) = (-1)^{|w|-1} \phi(\mu^{1}_{\cc^{op}}(w)) + (-1)^{ |w| - 1} \mu^{1}_{\cc^{op}}(\phi(w)).
\end{align}
Then we compute
\begin{equation}\label{pairing condition ev 0 left}
\begin{split}
& (-1)^{|\psi|-1} \pi^{ev}(\mu_{\Y^{l}_{X_{1}} \otimes_{\K} (\Y^{l}_{X_{0}})^{\vee}}^{0, 0}(z \otimes f), \psi) \\
= & (-1)^{|\psi|-1} \pi^{ev}((-1)^{|f| - 1} \mu^{1}_{\cc^{op}}(z) \otimes f +  z \otimes (w \mapsto (-1)^{|w| - 1}f(\mu^{1}_{\cc^{op}}(w))), \psi)\\
= & (-1)^{|\psi| + |f| + |\mu^{1}_{\cc^{op}}(z)|}  f(\psi(\mu^{1}_{\cc^{op}}(z))) + (-1)^{|\psi| + |z| + |\psi(z)| } f(\mu^{1}_{\cc^{op}}(\psi(z))) ,\\
= & f(\psi(\mu^{1}_{\cc^{op}}(z))) + f(\mu^{1}_{\cc^{op}}(\psi(z))),
\end{split}
\end{equation}
where the last equality is because of \eqref{pairing condition ev 0 degree}, $|\mu^{1}_{\cc^{op}}(z)| = |z| + 1$ and $|\psi(z)| = |\psi| + |z|$, and 
\begin{equation}\label{pairing condition ev 0 right}
\begin{split}
& \pi^{ev}(z \otimes f, \mu_{\hom_{\K}(\Y^{l}_{X_{1}}, \Y^{l}_{X_{0}})}^{0, 0}(\psi))\\
= & \pi^{ev}(z \otimes f, (w \mapsto (-1)^{|w|-1} \phi(\mu^{1}_{\cc^{op}}(w)) + (-1)^{ |w| - 1} \mu^{1}_{\cc^{op}}(\phi(w)))) \\
= & (-1)^{|z| + |z| - 1} f(\psi(\mu^{1}_{\cc^{op}}(z))) + (-1)^{|z| + |z| - 1} f(\mu^{1}_{\cc^{op}}(\psi(z))) \\
= & (-1)^{1} f(\psi(\mu^{1}_{\cc^{op}}(z)))  + (-1)^{1} f(\mu^{1}_{\cc^{op}}(\psi(z))).
\end{split}
\end{equation}
So \eqref{pairing condition ev 0 left} and \eqref{pairing condition ev 0 right} add up to zero.

Let us check \eqref{pairing condition ev 1}; the other is completely analogous. 
These terms in \eqref{pairing condition ev 1} are possibly nonzero only if the degrees satisfy
\begin{equation}\label{trace pairing degree}
\maltese_{1;k} + |z| + |f| + 1 + |\psi| = 0.
\end{equation}
By definition we have
\begin{align}
& \mu_{\Y^{l}_{X_{1}} \otimes_{\K} (\Y^{l}_{X_{0}})^{\vee}}^{k, 0}(\mathbf{x}_{1;k}, z \otimes f)(w) = (-1)^{|f| - 1} \mu^{k+1}_{\cc^{op}}(\mathbf{x}_{1;k}, z) \otimes f(w), \\
& \mu_{\hom_{\K}(\Y^{l}_{X_{1}}, \Y^{l}_{X_{0}})}^{0, k}(\psi, \mathbf{x}_{1;k})(w) = (-1)^{|w|-1} \psi(\mu^{k+1}_{\cc^{op}}(\mathbf{x}_{1;k}, w)).
\end{align}
Note that the degree of $\mu^{k+1}_{\cc^{op}}(\mathbf{x}_{1;k}, z)$ is
\begin{equation}
|\mu^{k+1}_{\cc^{op}}(\mathbf{x}_{1;k}, z)| = |z| + \maltese_{1;k} + 1.
\end{equation}
Now apply the pairing \eqref{hom pairing} to obtain
\begin{equation}\label{trace 1}
\begin{split}
& (-1)^{\maltese_{1;k} \cdot (|z|+|f|+|\psi|) + |\psi| - 1} \pi^{ev}(\mu_{\Y^{l}_{X_{1}} \otimes_{\K} (\Y^{l}_{X_{0}})^{\vee}}^{k, 0}(\mathbf{x}_{1;k}, z \otimes f), \psi) \\
= & (-1)^{\maltese_{1;k} \cdot (|z|+|f|+|\psi|) + |\psi| - 1} \pi^{ev}((-1)^{|f| - 1} \mu^{k+1}_{\cc^{op}}(\mathbf{x}_{1;k}, z) \otimes f(\cdot), \psi) \\
= & (-1)^{\maltese_{1;k}(|z|+|f|+|\psi|)) + |\psi| + |f| + |z| + \maltese_{1;k} + 1 + |f|} f(\mu_{\cc^{op}}^{k+1}(\psi(z), \mathbf{x}_{1;k}),\\
= & (-1)^{\maltese_{1;k}(-\maltese_{1;k} - 1)} f(\mu_{\cc^{op}}^{k+1}(\psi(z), \mathbf{x}_{1;k}) \\
=  & f(\mu_{\cc^{op}}^{k+1}(\psi(z), \mathbf{x}_{1;k}).
\end{split}
\end{equation}
where the fourth equality is because of \eqref{trace pairing degree}.
Similarly, we have
\begin{equation}\label{trace 2}
\begin{split}
& \pi^{ev}(z \otimes f, \mu_{\hom_{\K}(\Y^{l}_{X_{1}}, \Y^{l}_{X_{0}})}^{0, k}(\psi, \mathbf{x}_{1;k})) \\
= &  \pi^{ev}(z \otimes f, w \mapsto (-1)^{|w|-1} \psi(\mu^{k+1}_{\cc^{op}}(\mathbf{x}_{1;k}, w))  \\
= & (-1)^{|z| } (-1)^{|z|-1} f(\mu_{\cc^{op}}^{k+1}(\phi(z), \mathbf{x}_{1;k})) \\
= & (-1)^{ 1} f(\mu_{\cc^{op}}^{k+1}(\phi(z), \mathbf{x}_{1;k})).
 \end{split}
\end{equation}
So \eqref{trace 1} and \eqref{trace 2} add up to zero.

For the other equation \eqref{pairing condition ev 2},
by definition we have
\begin{align}
& \mu_{\Y^{l}_{X_{1}} \otimes_{\K} (\Y^{l}_{X_{0}})^{\vee}}^{0, l}(z \otimes f, \mathbf{x}'_{1;l})(w) = (-1)^{|w|-1} z \otimes f(\mu_{\cc^{op}}^{l+1}(\mathbf{x}'_{1;l}, w)), \\
& \mu_{\hom_{\K}(\Y^{l}_{X_{1}}, \Y^{l}_{X_{0}})}^{l, 0}(\mathbf{x}'_{1;l}, \psi)(w) = (-1)^{ |w| -1} \mu^{l+1}_{\cc^{op}}(\mathbf{x}'_{1;l}, \psi(w)). 
\end{align}
Then we compute
\begin{equation}
\begin{split}
& (-1)^{|\psi|-1} \pi^{ev}(\mu_{\Y^{l}_{X_{1}} \otimes_{\K} (\Y^{l}_{X_{0}})^{\vee}}^{0, l}(z \otimes f, \mathbf{x}'_{1;l}), \psi) \\
= & (-1)^{|\psi|-1} \pi^{ev}(z \otimes (w \mapsto (-1)^{|w|-1} f(\mu_{\cc^{op}}^{l+1}(\mathbf{x}'_{1;l}, w))), \psi) \\
= & (-1)^{|\psi| + |z| + |\psi(z)|} f(\mu_{\cc^{op}}^{l+1}(\mathbf{x}'_{1;l}, \psi(z)))\\
= & f(\mu_{\cc^{op}}^{l+1}(\mathbf{x}'_{1;l}, \psi(z))),
\end{split}
\end{equation}
and
\begin{equation}
\begin{split}
& \pi^{ev}(z \otimes f, \mu_{\hom_{\K}(\Y^{l}_{X_{1}}, \Y^{l}_{X_{0}})}^{l, 0}(\mathbf{x}'_{1;l}, \psi) \\
= & \pi^{ev}(z \otimes f, (w \mapsto (-1)^{ |w| - 1} \mu^{l+1}_{\cc^{op}}(\mathbf{x}'_{1;l}, \psi(w)))) \\
= & (-1)^{|z| + |z| -1} f(\mu_{\cc^{op}}^{l+1}(\mathbf{x}'_{1;l}, \psi(z))) \\
= & (-1)^{1} f(\mu_{\cc^{op}}^{l+1}(\mathbf{x}'_{1;l}, \psi(z))).
\end{split}
\end{equation}
\end{proof}

Similar to the tautological pairing $\langle \cdot, \cdot \rangle_{taut}$ \eqref{bimodule taut pairing},
we can use the pairing the other way round with a sign twist to define a pairing of degree $1$:
\begin{equation}\label{pre pairing on Z}
\begin{split}
\pi^{ev, +}_{Y_{0}, Y_{1}}: & (\Y^{l}_{X_{1}} \otimes_{\K} (\Y^{l}_{X_{0}})^{\vee}[1] \bigoplus \hom_{\K}(\Y^{l}_{X_{0}}, \Y^{l}_{X_{1}}))(Y_{1}, Y_{0}) \\
& \times (\Y^{l}_{X_{0}} \otimes_{\K} (\Y^{l}_{X_{1}})^{\vee}[1] \bigoplus \hom_{\K}(\Y^{l}_{X_{1}}, \Y^{l}_{X_{0}}))(Y_{1}, Y_{0}) \to \K[1],
\end{split}
\end{equation}
\begin{equation}\label{pre pairing on Z formula}
\pi^{ev, +}_{Y_{0}, Y_{1}}( (z \otimes f, \phi), (w \otimes g, \psi)) =  (-1)^{|z|} f(\psi(z)) + (-1)^{|w| + |\phi|\psi| + |\phi| + |\psi|} g(\phi(w)).
\end{equation}
where the second sign follows the pattern of the sign in \eqref{taut pairing formula}.

Immediately from the definition \eqref{pre pairing on Z formula}, we see that

\begin{lem}\label{lem: Z pairing symmetric}
The pairing $\pi^{ev, +}_{Y_{0}, Y_{1}}$ \eqref{pre pairing on Z} is graded symmetric with sign $(-1)^{1+(|\phi|-1)(|\psi|-1)}$.
\end{lem} \qed

\begin{cor}\label{cor: Z pairing canonical}
The pairings $\pi^{ev, +}$ \eqref{pre pairing on Z} form a canonical pairing system of degree $1$ between the $\cc-\cc$-bimodules 
\[
\Y^{l}_{X_{1}} \otimes_{\K} (\Y^{l}_{X_{0}})^{\vee}[1] \bigoplus \hom_{\K}(\Y^{l}_{X_{0}}, \Y^{l}_{X_{1}})
\]
 and 
 \[
\Y^{l}_{X_{0}} \otimes_{\K} (\Y^{l}_{X_{1}})^{\vee}[1] \bigoplus \hom_{\K}(\Y^{l}_{X_{1}}, \Y^{l}_{X_{0}}).
 \]
\end{cor}
\begin{proof}
Since the pairing is only between $\Y^{l}_{X_{1}} \otimes_{\K} (\Y^{l}_{X_{0}})^{\vee}[1](Y_{1}, Y_{0})$ from the first factor of the left hand side of \eqref{pre pairing on Z} and $\hom_{\K}(\Y^{l}_{X_{1}}, \Y^{l}_{X_{0}})(Y_{0}, Y_{1})$ from the second factor,
as well as between $\hom_{\K}(\Y^{l}_{X_{0}}, \Y^{l}_{X_{1}})(Y_{1}, Y_{0})$ from the first factor and $\Y^{l}_{X_{0}} \otimes_{\K} (\Y^{l}_{X_{1}})^{\vee}[1](Y_{0}, Y_{1})$ from the second factor,
the result follows from Lemma \ref{lem: hom pairing canonical}.
\end{proof}

\section{The categorical punctured neighborhood of infinity}\label{section: neighborhood of infinity}

In this section, we review the definition of the {\it (algebraizable) categorical formal punctured neighborhood of infinity} of a given $\ainf$-catetegory $\cc$,
whose definition was first due to \cite{efimov2} in the case of dg-categories,
and has a straightforward generalization to $\ainf$-categories in \cite{GGV}.
We develop further studies of it which lead to several important Hochschild invariants related to Calabi-Yau structures.

\subsection{The definition}

The notion of the Calkin algebra was originally introduced in \cite{calkin} to study bounded operators in Hilbert spaces,
and has the following algebraic analogue (\cites{drinfeld_inf, efimov2}).
Define the category of {\it Calkin modules} over any $\ainf$-category $\cc$ to be 
\begin{equation}
\calk_{\cc} = \rmod \cc / \perf \cc.
\end{equation}
When $\cc = \K$ is the $\ainf$-category with one object whose endomorphism is the ground ring $\K$ on which all $\ainf$-operations vanish except $\mu^{1}$ and $\mu^{2}$;
we call the resulting category $\calk_{\K}$ the category of Calkin complexes over $\K$.
To give a concrete description of $\mu^{1}, \mu^{2}$, we first compute the morphism space in $\calk_{\K}$ to be:

\begin{lem}\label{lem: hom in calkin}
Let $M, N \in \ob \calk_{\K} = \ob \ch_{\K}$, i.e., chain complexes over $\K$ thought of as objects in the category $\calk_{\K}$.
The morphism space in $\calk_{\K}$ is given by
\begin{equation}\label{hom in calkin}
\calk_{\K}(M, N) := \hom_{\calk_{\K}}(M, N) = \cone(N \otimes M^{\vee} \to \hom_{\K}(M, N)) = N \otimes M^{\vee}[1] \bigoplus \hom_{\K}(M, N).
\end{equation}
\end{lem}
\begin{proof}
Since the derived tensor product over $\K$ is the ordinary tensor product, 
the bar complex for computing the morphism spaces in the quotient category $\ch_{\K}/ \perf \K$ reduces to
$\hom_{\K}(\K, N) \otimes \hom_{\K}(M, \K)  \to \hom_{\K}(M, N)$,
which gives the desired complex \eqref{hom in calkin}.
\end{proof}

The differential of $a = (z \otimes f, \phi) \in \calk_{\K}(M, N)$ is defined to be
\begin{equation}
\begin{split}
d_{\calk_{\K}}  (z \otimes f, \phi) = & ((-1)^{|f|} d_{N}z \otimes f + z \otimes \d f, \d \phi) \\
= & ((-1)^{|f|} d_{N} z \otimes f - (-1)^{|f|} z \otimes f \circ d_{M}, d_{N} \circ \phi - (-1)^{|\phi|} \phi \circ d_{M}).
\end{split}
\end{equation}
For $a_{i} = (z_{i} \otimes f_{i}, \phi_{i}) \in \calk_{\K}(M_{i-1}, M_{i}), i = 1, 2$, 
their {\it Calkin composition}, induced from the dg category structure on $\ch_{\K}/\perf \K$, is defined to be
\begin{equation}\label{calkin composition}
(f_{2} \otimes z_{2}, \phi_{2}) \circ_{\calk_{\K}} (f_{1} \otimes z_{1}, \phi_{1}) = (\phi_{1}^{*}(f_{2}) \otimes z_{2} + f_{1} \otimes \phi_{2}(z_{1}), \phi_{2} \circ \phi_{1}).
\end{equation}
Passing from the dg category structure to the $\ainf$-category structure introduces a sign twist:
\begin{align}
\mu^{1}_{\calk_{\K}}(a) & = - d_{\calk_{\K}}(a), \\
\mu^{2}_{\calk_{\K}}(a_{2}, a_{1}) & = (-1)^{|a_{1}|} a_{2} \circ_{\calk_{\K}} a_{1}. \label{mu2 calkin}
\end{align}
 
Composed with the natural quotient functor $\ch_{\K} \to \calk_{\K}$, the Yoneda functor for $\cc$ induces
\begin{equation}\label{calkinyoneda}
\bar{y}: \cc \to \rmod \cc \to \fun(\cc^{op}, \calk_{\K}).
\end{equation}
We define the categorical formal punctured neighborhood of infinity $\cinf$ of $\cc$ to be the essential image of \eqref{calkinyoneda}, i.e.
\begin{equation}\label{cinf definition}
\cinf = ess-im(\bar{y}: \cc \to \fun(\cc^{op}, \calk_{\K}).
\end{equation}

\begin{lem}\label{lem: cinf is dg}
The $\ainf$-category $\cinf$ is a dg category, i.e., has $\mu^{k}_{\cinf} = 0$ for all $k \ge 3$.
\end{lem}
\begin{proof}
Since $\fun(\cc^{op}, \calk_{\K})$ has $\mu^{k} = 0$ for all $k \ge 3$,
it follows from the definition \eqref{cinf definition} that $\cinf$ has $\mu^{k}_{\cinf} = 0$ for all $k \ge 3$ as well.
\end{proof}

By construction, it comes with a canonical $\ainf$-functor, the induced Yoneda functor,
\begin{equation}\label{c to cinf}
\bar{y}: \cc \to \cinf,
\end{equation}
which sends (right) proper objects to zero,
where we say an object $X \in \ob \cc$ is (right) proper if $\cc(Y, X) \in \perf \K$ for all $Y$;
in other words, the right Yoneda module of $X$ is a proper $\cc$-module.
Denote by $\cc_{prop}$ the subcategory of (right) proper objects.
Thus the functor \eqref{c to cinf} induces a functor
\begin{equation}\label{cmodprop to cinf1}
\tilde{\bar{y}}: \cc/\cc_{prop} \to \cinf.
\end{equation}

It follows immediately by definition that

\begin{lem}\label{lem: cinf=0cproper}
$\cinf = 0$ if and only if $\cc$ is proper.
\end{lem}
\begin{proof}
If $\cc$ is proper, then every right Yoneda module $Y^{r}_{X}, X \in \ob \cc$ is proper.
Thus $\bar{y}(L) = 0$, where $\bar{y}$ is \eqref{calkinyoneda}.

Conversely, if $\cinf=0$, then $\bar{y}(X) = pr(Y^{r}_{X}) = 0$, where $pr: \ch_{\K} \to \ch_{\K}/\perf \K$ is the quotient functor. 
Thus $\cc(Y, X) \in \perf \K$ for all $X, Y$, so $\cc$ is proper. 
\end{proof}

Let $\r{Prop} \cc \subset \rmod \cc$ denote the full subcategory of proper right $\cc$-modules,
also called the {\it pseudo-perfect} modules.
If $\cc$ is smooth,
then by \cite{toen-vaquie} $\r{Prop} \cc \xhookrightarrow{} \perf \cc$,
and the functor \eqref{cmodprop to cinf1} can be rewritten as
\begin{equation}\label{cmodprop to cinf2}
\tilde{\bar{y}}: \cc/\r{Prop} \cc \to \cinf.
\end{equation}

Since the morphism spaces in the functor category are given by Hochschild cochain complexes, 
we see that

\begin{lem}[{\cite[\S 2.4]{GGV}}]\label{lem: cinf hom as cone}
For any $X_{0}, X_{1} \in \ob \cc = \ob \cinf$, we have strict equalities of chain complexes
\begin{align}
\cinf(X_{0}, X_{1}) & = \r{CC}^{*}(\cc^{op}, \calk_{\K}(\Y^{l}_{X_{0}}, \Y^{l}_{X_{1}})) \\
& = \r{CC}^{*}(\cc^{op}, \cone(\Y^{l}_{X_{1}} \otimes_{\K} (\Y^{l}_{X_{0}})^{\vee} \stackrel{i}\to \hom_{\K}(\Y^{l}_{X_{0}}, \Y^{l}_{X_{1}}))) \label{cinf hom as cone1} \\ 
& = \cone(\r{CC}^{*}(\cc^{op}, \Y^{l}_{X_{1}} \otimes_{\K} (\Y^{l}_{X_{0}})^{\vee}) \stackrel{i_{*}}\to \r{CC}^{*}(\cc^{op}, \hom_{\K}(\Y^{l}_{X_{0}}, \Y^{l}_{X_{1}}))), \label{cinf hom as cone2}
\end{align}
where the bimodule map $i$ is \eqref{bimodule i map},
and the differential on the Hochschild cochain complex is defined in \eqref{hochschild cochain differential}.

Moreover, the cone \eqref{cinf hom as cone1} carries a natural product structure giving an explicit formula for the product structure on $\cinf$
\begin{equation}\label{mu2 on cinf}
\begin{split}
\mu_{\cinf}^{2}(\psi_{2}, \psi_{1})^{k}(x_{1}, \ldots, x_{k}) = & \sum_{l} \mu^{2}_{\calk_{\K}}( \psi_{2}(x_{1}, \ldots, x_{l}), \psi_{1}(x_{l+1}, \ldots, x_{k})) \\
=& \sum_{l} (-1)^{|\psi_{1}(x_{l+1}, \ldots, x_{k})|} \psi_{2}(x_{1}, \ldots, x_{l}) \circ_{\calk_{\K}} \psi_{1}(x_{l+1}, \ldots, x_{k})) \\
 =& \sum_{l} (-1)^{|\psi_{1}| + \maltese_{l+1; k}} \psi_{2}(x_{1}, \ldots, x_{l}) \circ_{\calk_{\K}} \psi_{1}(x_{l+1}, \ldots, x_{k}).
\end{split}
\end{equation}
\end{lem} 
\begin{proof}
This follows from the definition of $\mu^{2}$ in the functor category $\fun(\cc^{op}, \calk_{\K})$ \eqref{muk in fun}, Lemma \ref{lem: cc computes hom in fun} and Lemma \ref{lem: hom in calkin}.
\end{proof}

For simplicity of notation, for each $X_{0}, X_{1} \in \ob \cinf$, we denote the $\cc^{op} - \cc^{op}$-bimodule in the coefficient of \eqref{cinf hom as cone1} by
\begin{equation}\label{Z symbol for cone bimodules}
\cZ_{X_{0}}^{X_{1}} := \cone(\Y^{l}_{X_{1}} \otimes_{\K} (\Y^{l}_{X_{0}})^{\vee} \stackrel{i}\to \hom_{\K}(\Y^{l}_{X_{0}}, \Y^{l}_{X_{1}})) = \Y^{l}_{X_{1}} \otimes_{\K} (\Y^{l}_{X_{0}})^{\vee}[1] \bigoplus \hom_{\K}(\Y^{l}_{X_{0}}, \Y^{l}_{X_{1}})),
\end{equation}
so that $\cinf(X_{0}, X_{1}) = \r{CC}^{*}(\cc^{op}, \cZ_{X_{0}}^{X_{1}})$.

\begin{lem}\label{lem: typical chain complex for Z}
As chain complexes, we have
\begin{equation}\label{Z as cone}
\begin{split}
\cZ_{X_{0}}^{X_{1}}(Y_{1}, Y_{0}) = & \cone(\Y^{l}_{X_{1}} \otimes_{\K} (\Y^{l}_{X_{0}})^{\vee}(Y_{1}, Y_{0}) \to \hom_{\K}(\Y^{l}_{X_{0}}(Y_{0}), \Y^{l}_{X_{1}}(Y_{1}))) \\
= & \cone(\cc(Y_{1}, X_{1}) \otimes_{\K} \cc(Y_{0}, X_{0})^{\vee} \to \hom_{\K}(\cc(Y_{0}, X_{0}), \cc(Y_{1}, X_{1}))) 
\end{split}
\end{equation}
\end{lem}
\begin{proof}
This is a straightforward calculation by the definition of the bimodules $\cZ$ \eqref{Z symbol for cone bimodules}.
Also see Definition \ref{example: linear tensor bimodule} and Definition \ref{example: linear hom bimodule} for the relevant definitions of the linear tensor product bimodule and the linear hom bimodule.
\end{proof}

The bimodule structure maps on this cone \eqref{Z symbol for cone bimodules} are induced from the bimodule structure maps \eqref{bimodule maps for finite linear hom} for $\Y^{l}_{X_{1}} \otimes_{\K} (\Y^{l}_{X_{0}})^{\vee}$ and \eqref{bimodule maps for linear hom} for $\hom_{\K}(\Y^{l}_{X_{0}}, \Y^{l}_{X_{1}})$, 
as well as the bimodule map $i$ \eqref{bimodule i map} (which has vanishing higher order terms),
\begin{equation}\label{structure maps for Z}
\begin{split}
 & \mu^{k, l}_{\cZ_{X_{0}}^{X_{1}}}(\mathbf{x}_{1;k}, (z \otimes f, \phi), \mathbf{x}'_{1;l}) =\\
& \begin{cases}
 & (\mu^{k, l}_{\Y^{l}_{X_{1}} \otimes_{\K} (\Y^{l}_{X_{0}})^{\vee}}(\mathbf{x}_{1;k}, z \otimes f, \mathbf{x}'_{1;l}), \mu^{k, l}_{\hom_{\K}(\Y^{l}_{X_{0}}, \Y^{l}_{X_{1}})}(\mathbf{x}_{1;k}, \phi, \mathbf{x}'_{1;l})),  \text{ if } k>0 \text{ or } l>0, \\
& (\mu^{0, 0}_{\Y^{l}_{X_{1}} \otimes_{\K} (\Y^{l}_{X_{0}})^{\vee}}(z \otimes f), \mu^{0, 0}_{\hom_{\K}(\Y^{l}_{X_{0}}, \Y^{l}_{X_{1}})}(\phi) + i(z \otimes f)).
 \end{cases}
 \end{split}
\end{equation}

Phrasing Lemma \ref{lem: cinf hom as cone} more functorially, 
the composition in the Calkin category $\calk_{\K} = \ch_{\K}/\perf \K$ induces a morphism of $\cc^{op}-\cc^{op}$-bimodules 
\begin{equation}\label{bimodule calkin composition}
c_{\calk_{\K}}: \cZ_{X_{1}}^{X_{2}} \otimes_{\cc^{op}} \cZ_{X_{0}}^{X_{1}} \to \cZ_{X_{0}}^{X_{2}}.
\end{equation}
Now the second half of Lemma \ref{lem: cinf hom as cone} can be restated as:

\begin{cor}\label{cor: mu2 on cinf is composition}
The $\mu^{2}_{\cinf}$-product on $\cinf$ is given by the composition
\begin{equation}\label{mu2 on cinf is composition}
\cinf(X_{1}, X_{2}) \otimes \cinf(X_{0}, X_{1}) \stackrel{\sqcup}\to \r{CC}^{*}(\cc^{op}, \cZ_{X_{1}}^{X_{2}} \otimes_{\cc^{op}} \cZ_{X_{0}}^{X_{1}}) \stackrel{c_{\calk_{\K}, *}}\to \r{CC}^{*}(\cc^{op}, \cZ_{X_{0}}^{X_{2}}) = \cinf(X_{0}, X_{2})
\end{equation}
up to the sign twist $(-1)^{|c_{1}|}$,
where $\sqcup$ is \eqref{composition product}, 
and the map $c_{\calk_{\K}, *}$ is map on Hochschild cochain complexes induced by the morphism of bimodules $c_{\calk_{\K}}$ \eqref{bimodule calkin composition}.
That is
\begin{equation}\label{mu2 on cinf composition sign}
\mu^{2}_{\cinf}(c_{2}, c_{1}) = (-1)^{|c_{1}|} c_{\calk_{\K}, *}(c_{2} \sqcup c_{1}).
\end{equation}
\end{cor}
\begin{proof}
This follows immediately from Lemma \ref{lem: cinf hom as cone} and the formula \eqref{composition product formula} for the product $\sqcup$.
\end{proof}

From this, it is also clear that $\mu^{2}_{\cinf}$ is associative on the chain level with the standard $\ainf$-sign twist,
because $\sqcup$ and the composition in the Calkin category $\calk_{\K}$ are both associative.

\begin{cor}\label{cor: cinf as cone}
As dg categories, or $\ainf$-categories with $\mu^{k} = 0$ for $k \ge 3$, there are quasi-equivalence
\begin{align}\label{cinf as cone}
\cinf \cong & \cone(\r{CC}^{*}(\cc^{op}, \calk_{\K}(\cc^{op}, \cc^{op}))) \\
\cong & \cone(\r{CC}^{*}(\cc^{op}, \cc^{op}_{\D} \otimes (\cc^{op})^{\vee}) \to \r{CC}^{*}(\cc^{op}, \hom_{\K}(\cc^{op}_{\D}, \cc^{op}_{\D})))
\end{align}
\end{cor} \qed

In the case where $\cc$ is smooth, we also have

\begin{lem}[{\cite[Proposition 2.11]{GGV}}]\label{lem: cinf smooth cone}
Suppose $\cc$ is a smooth $\ainf$-category. 
Then we have a distinguished triangle of $\cc^{op}$-bimodules
\begin{equation}\label{triangle of cinf}
(\cc^{op})^{!} \otimes_{\cc^{op}} (\cc^{op})^{\vee} \to \cc^{op}_{\D} \to \cinf.
\end{equation}
\end{lem} \qed

Sometimes it will be convenient to switch from $\cc^{op}$-bimodules to $\cc$-bimodules.
Noting that the convention for the diagonal bimodule is $(\cinf)_{\D}(X, Y) = \cinf(Y, X)$,
one obtains the following corollary by interchanging the left and right actions to go from a $\cc^{op}$-bimodule to a $\cc$-bimodule.

\begin{cor}\label{cor: cinf smooth cone}
If $\cc$ is a smooth $\ainf$-category, 
we have a distinguished triangle of $\cc$-bimodules
\begin{equation}
\cc^{!} \otimes_{\cc} \cc^{\vee} \to \cc_{\D} \to (\bar{y}, \bar{y})^{*}(\cinf)_{\D}
\end{equation}
where $\bar{y}: \cc \to \cinf$ is \eqref{c to cinf}.
\end{cor}
\begin{proof}
Since $\cc^{op}$-bimodule is naturally a $\cc$-bimodule by interchanging the left and right actions,
the result follows by applying the interchange to \eqref{triangle of cinf}.
\end{proof}

\subsection{An extended product}\label{section: cup product}

Viewing $\cinf$ as a bimodule over $\cc$ via the canonical functor $\bar{y}: \cc \to \cinf$ \eqref{calkinyoneda}, 
we can consider the Hochschild cohomology $\r{CC}^{*}(\cc, \cinf)$ with coefficients in $\cinf$, i.e.,
\begin{equation}
\r{CC}^{*}(\cc, \cinf) := \r{CC}^{*}(\cc, (\bar{y}, \bar{y})^{*}(\cinf)_{\D}).
\end{equation}
The $\ainf$ category structure on $\cinf$ induces a morphism of $\cc$-bimodules
\begin{equation}\label{wedge product}
\wedge_{\infty}: \cinf \otimes_{\cc} \cinf \to \cinf
\end{equation}
by means of pullback by $\bar{y}$;
the formula is similar to \eqref{bimodule collapse map formula} for the map \eqref{bimodule collapse map},
where $\mu^{k+s+1, l}_{(\cinf)_{\D}} = (-1)^{\maltese'_{l}+1} \mu^{k+s+1+l+1}_{\cinf}$, 
since $(\cinf)_{\D}$ is a diagonal bimodule itself whose structure maps are defined in \eqref{structure maps for diagonal}
(in fact, most of the maps vanish since $\cinf$ is a dg category by Lemma \ref{lem: cinf is dg}).
In other words, we define $\wedge_{\infty}$ to be the pullback of the bimodule collapse map $\mu_{\D, (\cinf)_{\D}}: (\cinf)_{\D} \otimes_{\cinf} (\cinf)_{\D} \to (\cinf)_{\D}$ for $\cinf$ by $(\bar{y}, \bar{y})^{*}$.

Thus we can define an extended cup product on $\r{CC}^{*}(\cc, \cinf)$ by
\begin{equation}\label{extended cup product}
\cup_{\infty}: \r{CC}^{*}(\cc, \cinf) \otimes \r{CC}^{*}(\cc, \cinf) \stackrel{\sqcup} \to \r{CC}^{*}(\cc, \cinf \otimes_{\cc} \cinf) \stackrel{\wedge_{\infty}}\to \r{CC}^{*}(\cc, \cinf),
\end{equation}
where $\sqcup: \r{CC}^{*}(\cc, \cinf) \otimes \r{CC}^{*}(\cc, \cinf) \to \r{CC}^{*}(\cc, \cinf \otimes_{\cc} \cinf)$ is the product \eqref{composition product} defined in \S\ref{section: products on Hochschild}.
Compare to a similar given in \cite{rivera-takeda-wang} using a method of graphical calculus.

Now we relate this cup product to the usual cup product \eqref{cup product cc} on $\r{CC}^{*}(\cinf, \cinf)$,
defined on the chain level as the composition
\begin{equation}\label{cup product on cinf}
\widehat{\cup}_{\infty}: \r{CC}^{*}(\cinf, \cinf) \otimes \r{CC}^{*}(\cinf, \cinf) \stackrel{\sqcup}\to \r{CC}^{*}(\cinf, \cinf \otimes_{\cinf} \cinf) \stackrel{\mu_{\D, (\cinf)_{\D}}}\to \r{CC}^{*}(\cinf, \cinf),
\end{equation}
where $\mu_{\cinf}$ is the map induced by the canonical collapse map of $\cinf$-bimodules
\begin{equation}
\mu_{\D, \cinf}: (\cinf)_{\D} \otimes_{\cinf} (\cinf)_{\D} \stackrel{\sim}\to (\cinf)_{\D}
\end{equation} 
defined as in \eqref{bimodule collapse map},
which is a quasi-isomorphism.

\begin{prop}
The pullback 
\begin{equation}
\r{HH}^{*}(\cinf, \cinf) \to \r{HH}^{*}(\cc, \cinf)
\end{equation}
is a map of unital rings.
\end{prop}
\begin{proof}
Under the pullback by $\bar{y}: \cc \to \cinf$, 
the bimodule $(\cinf)_{\D} \otimes_{\cinf} (\cinf)_{\D}$ becomes $\cinf \otimes_{\cc} \cinf: = (\bar{y}, \bar{y})^{*}(\cinf)_{\D} \otimes_{\cc} (\bar{y}, \bar{y})^{*}(\cinf)_{\D}$,
such that the pullback of $\mu_{\D, (\cinf)_{\D}}$ agrees with $\wedge_{\infty}$ (this is how $\wedge_{\infty}$ is defined).
The statement then follows by the definition of \eqref{extended cup product} and the definition of the usual cup product on $\r{CC}^{*}(\cinf, \cinf)$ by \eqref{cup product on cinf}.
\end{proof}

\subsection{The residue}\label{section:res}

In this subsection, we shall present our main construction of a chain map of degree $1-n$
\begin{equation}\label{res cy}
\r{res}: \r{CC}_{n-1}(\cinf) \to \K,
\end{equation}
which gives the weak proper Calabi-Yau structure on $\cinf$ under the hypothesis of Theorem \ref{thm: cinf cy}.

Throughout the remaining subsections of \S\ref{section: neighborhood of infinity},
we assume $\cc$ has a weak smooth Calabi-Yau structure $\sigma$ of dimension $n$,
so that Lemma \ref{lem: capping with cy is iso} provides a canonical quasi-isomorphism
\[
-\cap \sigma: \r{CC}^{*}(\cc, \cP) \stackrel{\sim}\to \r{CC}_{*-n}(\cc, \cP)
\]
for any $\cc$-bimodule $\cP$.
We will define a map 
\begin{equation}\label{res with tensor product coefficients}
\r{res}_{pre}: \r{CC}_{*}(\cinf, \cinf \otimes_{\cinf} \cinf) \to \K[1-n]
\end{equation}
of degree $1-n$.
By Proposition \ref{prop: pairing induced map}, it suffices to construct a canonical pairing system (in the sense of Definition \ref{def: canonical pairing system}) of degree $1-n$
\begin{equation}\label{pairing on cinf}
\pi_{\infty, X_{0}, X_{1}}: \cinf(X_{0}, X_{1}) \times \cinf(X_{1}, X_{0}) \to \K[1-n]
\end{equation}
for all $X_{0}, X_{1} \in \ob \cinf$.

Recall from \eqref{cinf hom as cone1} in Lemma \ref{lem: cinf hom as cone} that the morphism spaces in $\cinf$ can be computed as Hochschild cochain complexes of $\cc^{op}$ with coefficients in certain cones of $\cc^{op}$-bimodules.
Writing the Hochschild chain complex in the shifted form \eqref{cc shifted form}, we have
\begin{align}
\r{CC}_{*}(\cinf) = &\bigoplus (\cinf)_{\D}(X_{0}, X_{k}) \otimes \cinf[1](X_{k-1}, X_{k}) \otimes \cdots \otimes \cinf[1](X_{0}, X_{1}) \\
= &\bigoplus \cinf(X_{k}, X_{0}) \otimes \cinf[1](X_{k-1}, X_{k}) \otimes \cdots \otimes \cinf[1](X_{0}, X_{1}) \label{computing cc of cinf 1}\\
&= \bigoplus \r{CC}^{*}(\cc^{op}, \cZ_{X_{k}}^{X_{0}}) \otimes \r{CC}^{*}(\cc^{op}, \cZ_{X_{k-1}}^{X_{k}})[1] \otimes \cdots \otimes \r{CC}^{*}(\cc^{op}, \cZ_{X_{0}}^{X_{1}})[1]. \label{computing cc of cinf 2}
\end{align}
Here $\cZ_{X}^{Y}$ are defined in \eqref{Z symbol for cone bimodules}.

\begin{lem}\label{lem: hochschild differential on cinf decreases length only by one}
The Hochschild differential on $\r{CC}_{*}(\cinf)$,
with respect to the direct sum decomposition \eqref{computing cc of cinf 1},
has the following formula
\begin{equation}\label{hochschild chain differential on cinf}
\begin{split}
& d_{\r{CC}_{*}}(a_{0} \otimes \mathbf{x}_{1;k}) \\
=  & (-1)^{\star_{0}^{0}} \mu_{\cinf}^{1}(a_{0}) \otimes \mathbf{x}_{1;k} + (-1)^{\star_{1}^{0}} \mu_{\cinf}^{2}(x_{1}, a_{0}) \otimes \mathbf{x}_{2; k}
 + (-1)^{\star_{0}^{1} + |x_{k}|-1} \mu_{\cinf}^{2}(a_{0}, x_{k}) \otimes \mathbf{x}_{1;k-1} \\
& + \sum (-1)^{\maltese_{1;i}} a_{0} \otimes \mathbf{x}_{i+2; k} \otimes \mu_{\cinf}^{1}(x_{i+1}) \otimes \mathbf{x}_{1;i} 
+ \sum (-1)^{\maltese_{1;i}} a_{0} \otimes \mathbf{x}_{i+3; k} \otimes \mu_{\cinf}^{2}(x_{i+2}, x_{i+1}) \otimes \mathbf{x}_{1;i}.
\end{split}
\end{equation}
\end{lem}
\begin{proof}
By Lemma \ref{lem: cinf is dg}, $\cinf$ has $\mu^{k}_{\cinf} = 0$ for all $k \ge 3$.
In particular, the formula \eqref{structure maps for diagonal} implies that the only non-trivial structure maps for the diagonal bimodule $(\cinf)_{\D}$ (where $(\cinf)_{\D}(X, Y) = \cinf(Y, X)$) are
\begin{equation}\label{cinf diagonal 00}
\mu^{0,0}_{(\cinf)_{\D}}(a) = (-1)^{1} \mu^{1}_{\cinf}(a) = -\mu^{1}_{\cinf}(a),
\end{equation}
\begin{equation}\label{cinf diagonal 10}
\mu^{1,0}_{(\cinf)_{\D}}(x, a) = (-1)^{1} \mu^{2}_{\cinf}(x, a) = -\mu^{2}_{\cinf}(x, a),
\end{equation}
and 
\begin{equation}\label{cinf diagonal 01}
\mu^{0,1}_{(\cinf)_{\D}}(a, x') = (-1)^{|x'|-1+1} \mu^{2}_{\cinf}(a, x') = (-1)^{|x'|} \mu^{2}_{\cinf}(a, x').
\end{equation}
where we use the letter $a$ to denote an element of the diagonal bimodule, and $x, x'$ from morphism spaces in $\cinf$.

The formula \eqref{hochschild chain differential on cinf} follows from the general formula \eqref{hochschild chain differential} for Hochschild chain differentials:
\begin{enumerate}[label=(\roman*)]

\item The first sum on the right hand side of \eqref{hochschild chain differential} has only terms in three cases $i=0, j=0$ or $i=1,j=0$ or $i=0,j=1$, because the diagonal bimodule $(\cinf)_{\D}$ has only the above non-trivial structure maps \eqref{cinf diagonal 00}, \eqref{cinf diagonal 10} and \eqref{cinf diagonal 01}

\item The second sum has only terms in two cases: $j=1$ or $j=2$, because the category $\cinf$ has only non-trivial structure maps $\mu_{\cinf}^{1}$ and $\mu_{\cinf}^{2}$.

\end{enumerate}
\end{proof}

\begin{cor}\label{cor: hochschild differential on cinf decreases length only by one}
The Hochschild differential on $\r{CC}_{*}(\cinf)$,
with respect to the direct sum decomposition \eqref{computing cc of cinf 2},
is the sum of the tensor product differential of the Hochschild differentials on the various Hochschild cochain complexes $\r{CC}^{*}(\cc^{op}, \cZ_{X_{i}}^{X_{i+1}})$ (where $X_{k+1}=X_{0}$), 
as well as maps
\begin{equation}\label{contracting differential}
\begin{split}
& \r{CC}^{*}(\cc^{op}, \cZ_{X_{k}}^{X_{0}}) \otimes \r{CC}^{*}(\cc^{op}, \cZ_{X_{k-1}}^{X_{k}})[1] \otimes \cdots \otimes \r{CC}^{*}(\cc^{op}, \cZ_{X_{0}}^{X_{1}})[1]\\
\to \bigoplus &\r{CC}^{*}(\cc^{op}, \cZ_{X_{k}}^{X_{0}}) \otimes \r{CC}^{*}(\cc^{op}, \cZ_{X_{k-1}}^{X_{k}})[1] \otimes \cdots \r{CC}^{*}(\cc^{op}, \cZ_{X_{i-1}}^{X_{i+1}})[1] \otimes \cdots \otimes \r{CC}^{*}(\cc^{op}, \cZ_{X_{0}}^{X_{1}})[1],
\end{split}
\end{equation}
induced by maps
\begin{equation}
\mu^{2}_{\cinf}: \r{CC}^{*}(\cc^{op}, \cZ_{X_{i}}^{X_{i+1}}) \otimes \r{CC}^{*}(\cc^{op}, \cZ_{X_{i-1}}^{X_{i}}) \to \r{CC}^{*}(\cc^{op}, \cZ_{X_{i-1}}^{X_{i+1}}).
\end{equation}
\end{cor}
\begin{proof}
This follows from Lemma \ref{lem: hochschild differential on cinf decreases length only by one} and Lemma \ref{lem: cinf hom as cone}.
\end{proof}

Now we return to the construction of the canonical pairing system \eqref{pairing on cinf}.
First take the product \eqref{composition product}
\begin{equation}\label{cup product for cinf}
\cinf(X_{0}, X_{1}) \otimes \cinf(X_{1}, X_{0}) = \r{CC}^{*}(\cc^{op}, \cZ_{X_{0}}^{X_{1}}) \otimes \r{CC}^{*}(\cc^{op}, \cZ_{X_{1}}^{X_{0}}) \stackrel{\sqcup}\to \r{CC}^{*}(\cc^{op}, \cZ_{X_{0}}^{X_{1}} \otimes_{\cc^{op}} \cZ_{X_{1}}^{X_{0}}).
\end{equation}
Since $\cc$ has a weak smooth Calabi-Yau structure, so does $\cc^{op}$ (via $\r{CC}_{*}(\cc) \cong \r{CC}_{*}(\cc^{op})$ on the chain level), 
and by Lemma we have \ref{lem: capping with cy is iso} a canonical quasi-isomorphism of chain complexes
\begin{equation}\label{cy for cc of Z}
-\cap \sigma: \r{CC}^{*}(\cc^{op}, \cZ_{X_{0}}^{X_{1}} \otimes_{\cc^{op}} \cZ_{X_{1}}^{X_{0}}) \stackrel{\sim}\to \r{CC}_{*-n}(\cc^{op}, \cZ_{X_{0}}^{X_{1}} \otimes_{\cc^{op}} \cZ_{X_{1}}^{X_{0}}).
\end{equation}
Now it suffices to define a chain map of degree $1$:
\begin{equation}\label{trace on cc Z}
\r{CC}_{*}(\cc^{op}, \cZ_{X_{0}}^{X_{1}} \otimes_{\cc^{op}} \cZ_{X_{1}}^{X_{0}}) \to \K[1].
\end{equation}
This will be constructed using Proposition \ref{prop: pairing induced map}.

Now we define the map \eqref{trace on cc Z} as follows.
By Lemma \ref{lem: typical chain complex for Z}, a typical element of $\cZ_{X_{0}}^{X_{1}}(Y_{1}, Y_{0})$ is of the following form
\begin{equation}
(z \otimes f, \phi) \in \cone(\cc(Y_{1}, X_{1}) \otimes_{\K} \cc(Y_{0}, X_{0})^{\vee}  \to \hom_{\K}(\cc(Y_{0}, X_{0}), \cc(Y_{1}, X_{1}))),
\end{equation}
and a typical element of $\cZ_{X_{1}}^{X_{0}}(Y_{0}, Y_{1})$ is of the form
\begin{equation}
(w \otimes g, \psi) \in ( \cc(Y_{0}, X_{0}) \otimes_{\K} \cc(Y_{1}, X_{1})^{\vee}  \to \hom_{\K}(\cc(Y_{1}, X_{1}), \cc(Y_{0}, X_{0}))).
\end{equation}
Because of the grading convention on the mapping cone complex, the degrees should satisfy
\begin{align}
|f| + |z| & = |\phi| +1, \label{grading relation on cone1}\\
|g| + |w| & = |\psi| +1. \label{grading relation on cone2}
\end{align}
We define a bilinear pairing:
\begin{equation}\label{pairing on Z}
\pi^{\cZ}_{Y_{0}, Y_{1}}: \cZ_{X_{0}}^{X_{1}}(Y_{1}, Y_{0}) \times \cZ_{X_{1}}^{X_{0}}(Y_{0}, Y_{1}) \to \K[1]
\end{equation}
by the formula
\begin{equation}\label{pairing on Z formula}
\pi^{\cZ}_{Y_{0}, Y_{1}}((z \otimes f, \phi), (w \otimes g, \psi) ) = (-1)^{|z| + |f|} f(\psi(z)) + (-1)^{|w| + |g| + |\phi||\psi| + |\phi| + |\psi|} g(\phi(w)).
\end{equation}
That is, this is the pairing from \eqref{pre pairing on Z} with formula \eqref{pre pairing on Z formula}.

\begin{lem}\label{lem: pairing on Z canonical}
The pairings $\pi^{\cZ}_{Y_{0}, Y_{1}}$ \eqref{pairing on Z} form a canonical pairing system of degree $1$ for the $\cc^{op} - \cc^{op}$-bimodules $\cZ_{X_{0}}^{X_{1}}$ and $\cZ_{X_{1}}^{X_{0}}$.
\end{lem}
\begin{proof}
Since the $\cc^{op}-\cc^{op}$-bimodule structure on $\cZ_{X_{0}}^{X_{1}}$ is just formal consequence of the fact that a $\cc-\cc$-bimodule is automatically a $\cc^{op}-\cc^{op}$-bimodule when switching the left and right actions, 
and since the condition \eqref{condition for pairing system} is completely symmetric in the entries from the category $\cc$ or $\cc^{op}$, 
Lemma \ref{lem: pairing on Z canonical} is equivalent to the statement that the pairings $\pi^{\cZ}_{Y_{0}, Y_{1}}$ defined by the same formula \eqref{pairing on Z} form a canonical pairing system of degree $1$ for the $\cc-\cc$-bimodules.

Now we appeal to Lemma \ref{lem: hom pairing canonical} and Corollary \ref{cor: Z pairing canonical} to draw the conclusions.
The only difference is that the $\mu^{0, 0}$-terms for $\cZ_{X_{0}}^{X_{1}}$ and $\cZ_{X_{0}}^{X_{1}}$ \eqref{structure maps for Z} have an extra term coming from the bimodule map $i$ \eqref{bimodule i map},
which has only a $i^{0, 0}$-th term. 
When applying $\mu^{0, 0}_{\cZ_{X_{0}}^{X_{1}}}$ to $(z \otimes f, \phi)$ and pairing it with $(w \otimes g, \psi)$ we get an extra term of the form $(-1)^{*} f(w)g(z)$.
Similarly, the pairing $(z \otimes f, \phi)$ with $\mu^{0, 0}_{\cZ_{X_{1}}^{X_{0}}}(w \otimes g, \psi)$ also yields the same term,
where the signs \eqref{bimodule maps for finite linear hom}, \eqref{bimodule maps for linear hom}, \eqref{pairing on Z} formula ensure that these terms add up to zero.
\end{proof}

Now apply Proposition \ref{prop: pairing induced map} to the canonical pairing system $\pi^{\cZ}_{Y_{0}, Y_{1}}$ \eqref{pairing on Z} to obtain:

\begin{cor}\label{cor: induced map on cc Z}
The induced map
\begin{equation}\label{induced map on cc Z}
\pi^{\cZ}_{*}: \r{CC}_{*}(\cc^{op}, \cZ_{X_{0}}^{X_{1}} \otimes_{\cc^{op}} \cZ_{X_{1}}^{X_{0}}) = \cZ_{X_{0}}^{X_{1}} \otimes_{\cc^{op}-\cc^{op}} \cZ_{X_{1}}^{X_{0}} \to \K[1]
\end{equation}
defined by 
\begin{equation}\label{induced map on cc Z formula}
\pi^{\cZ}_{*}((z \otimes f, \phi) \otimes \mathbf{x}'_{1;l} \otimes (w \otimes g, \psi) \otimes \mathbf{x}_{1;k}) = 
\begin{cases}
\pi^{\cZ}_{Y_{0}, Y_{1}}((z \otimes f, \phi), (w \otimes g, \psi)), & \text{ if } k=l=0, \\
0, & \text{ if } k > 0 \text{ or } l > 0,
\end{cases}
\end{equation}
is a chain map of degree $1$.
\end{cor} \qed

There is one additional cyclically associativity property or the pairing $\pi^{\cZ}$ in the following sense.

\begin{lem}\label{lem: Z pairing associative}
The following diagram commutes
\begin{equation}
\begin{tikzcd}
& \cZ_{X_{2}}^{X_{0}}(Y_{2}, Y_{0}) \times \cZ_{X_{1}}^{X_{2}}(Y_{1}, Y_{2}) \times \cZ_{X_{0}}^{X_{1}}(Y_{0}, Y_{1}) \arrow[r, " c_{\calk_{\K}} \times \id"] \arrow[d, "\id \times c_{\calk_{\K}}"] & \cZ_{X_{1}}^{X_{0}}(Y_{1}, Y_{0}) \times \cZ_{X_{0}}^{X_{1}}(Y_{0}, Y_{1}) \arrow[d, "\pi^{\cZ}"] \\
& \cZ_{X_{2}}^{X_{0}}(Y_{2}, Y_{0}) \times \cZ_{X_{0}}^{X_{1}}(Y_{0}, Y_{1}) \arrow[r, "\pi^{\cZ}"] & \K
\end{tikzcd}
\end{equation}
\end{lem}
\begin{proof}
This is a straightforward computation following the definition of the Calkin composition \eqref{calkin composition} and the formula of the pairing \eqref{pairing on Z formula}.
\end{proof}

Now we compose the chain map $\pi^{\cZ}_{*}$ \eqref{induced map on cc Z}
with the quasi-isomorphism \eqref{cy for cc of Z} and the product \eqref{cup product for cinf} to get a chain map of degree $1-n$:
\begin{equation}\label{defining pairing on cinf}
\pi_{\infty, X_{0}, X_{1}}: \cinf(X_{0}, X_{1}) \otimes \cinf(X_{1}, X_{0}) \stackrel{\sqcup}\to  \r{CC}^{*}(\cc^{op}, \cZ_{X_{0}}^{X_{1}} \otimes_{\cc^{op}} \cZ_{X_{1}}^{X_{0}}) \stackrel{-\cap \sigma}\to \r{CC}_{*}(\cc^{op}, \cZ_{X_{0}}^{X_{1}} \otimes_{\cc^{op}} \cZ_{X_{1}}^{X_{0}})[-n] \stackrel{\pi^{\cZ}_{*}}\to \K[1-n].
\end{equation}
which gives the pairing $\pi_{\infty, X_{0}, X_{1}}$ \eqref{pairing on cinf}.

\begin{lem}\label{lem: pairing commutative}
The pairing \eqref{defining pairing on cinf} is graded symmetric with sign $(-1)^{1+(|c_{0}|-1)(|c_{1}|-1)}$.
\end{lem}
\begin{proof}
Since the pairing $\pi^{\cZ}_{Y_{0}, Y_{1}}$ \eqref{pairing on Z} which was defined by \eqref{pre pairing on Z} is graded symmetric by Lemma \ref{lem: Z pairing symmetric},
it follows by the definition \eqref{induced map on cc Z formula} that $\pi^{\cZ}_{*}$,
which only takes nonzero values on the length zero tensors (Definition \ref{def: length zero tensors}) in the subspace $\cZ_{X_{0}}^{X_{1}}(Y_{0}, Y_{1}) \otimes \cZ_{X_{1}}^{X_{0}}(Y_{1}, Y_{0})$ of the two-sided bar complex $\cZ_{X_{0}}^{X_{1}} \otimes_{\cc^{op}-\cc^{op}} \cZ_{X_{1}}^{X_{0}} = \r{CC}_{*}(\cc^{op}, \cZ_{X_{0}}^{X_{1}} \otimes_{\cc^{op}} \cZ_{X_{1}}^{X_{0}})$,
is also graded symmetric in the sense that it takes the same value as the other map
\begin{equation}
\pi^{\cZ}_{*}: \r{CC}_{*}(\cc^{op}, \cZ_{X_{1}}^{X_{0}} \otimes_{\cc^{op}} \cZ_{X_{0}}^{X_{1}}) \to \K,
\end{equation}
defined in the same manner, up to the sign twist specified by Lemma \ref{lem: Z pairing symmetric}. 
\end{proof}

To get the desired map \eqref{res cy}, we need to appeal to Proposition \ref{prop: pairing induced map} again,
which requires that we prove the following:

\begin{lem}\label{lem: cinf canonical pairing system}
The pairings $\pi_{\infty, X_{0}, X_{1}}$ \eqref{pairing on cinf} defined by \eqref{defining pairing on cinf} form a canonical pairing system between the diagonal bimodule $(\cinf)_{\D}$ and itself.
\end{lem}
\begin{proof}
We need to check the condition \eqref{condition for pairing system} for $\pi_{\infty, X_{0}, X_{1}}$.
By construction, the map $\pi_{\infty, X_{0}, X_{1}}$ is a chain map from the tensor product $\cinf(X_{0}, X_{1}) \otimes \cinf(X_{1}, X_{0})$ to $\K$,
so the first set of equations in \eqref{condition for pairing system} with $k=l=0$ is satisfied.
Since $\cinf$ has $\mu^{k}_{\cinf} = 0$ for all $k \ge 3$,
the only remaining equations to check are
\begin{equation}\label{cinf pairing symmetric}
(-1)^{|c_{1}|-1} \pi_{\infty}(\mu^{0, 1}_{(\cinf)_{\D}}(c_{0}, x'), c_{1}) + \pi_{\infty}(c_{0}, \mu^{1, 0}_{(\cinf)_{\D}}(x', c_{1})) = 0,
\end{equation}
and
\begin{equation}\label{cinf pairing symmetric 2}
(-1)^{(|x|-1)(|c_{0}| + |c_{1}|) + |c_{1}| - 1}\pi_{\infty}(\mu^{1, 0}_{(\cinf)_{\D}}(x, c_{0}), c_{1}) + \pi_{\infty}(c_{0}, \mu^{0, 1}_{(\cinf)_{\D}}(c_{1}, x)) = 0.
\end{equation}

Let us check \eqref{cinf pairing symmetric};
 the other is analogous, which is essentially the same as \eqref{cinf pairing symmetric} using Lemma \ref{lem: pairing commutative}.
By \eqref{cinf diagonal 10} and \eqref{cinf diagonal 01},
the equation \eqref{cinf pairing symmetric} is equivalent to
\begin{equation}\label{cinf pairing symmetric 1}
 \pi_{\infty}((-1)^{|x'| + |c_{1}|} \mu^{2}_{\cinf}(c_{0}, x'), c_{1}) =  \pi_{\infty}(c_{0}, \mu^{2}_{\cinf}(x', c_{1})),
\end{equation}
or 
\begin{equation}\label{cinf pairing symmetric 11}
 \pi_{\infty}((-1)^{|x'|} \mu^{2}_{\cinf}(c_{0}, x'), c_{1}) =  \pi_{\infty}(c_{0}, (-1)^{|c_{1}|}\mu^{2}_{\cinf}(x', c_{1})),
\end{equation}

Recall from \eqref{capping with cy} that the map $-\cap \sigma$ is defined as the following composition
\begin{equation}
\begin{split}
& \r{CC}^{*}(\cc^{op}, \cZ_{X_{0}}^{X_{1}} \otimes_{\cc^{op}} \cZ_{X_{1}}^{X_{0}}) \stackrel{- \sqcap \sigma}\to \r{CC}_{*}(\cc^{op},   (\cc^{op})_{\D} \otimes_{\cc^{op}} \cZ_{X_{0}}^{X_{1}} \otimes_{\cc^{op}} \cZ_{X_{1}}^{X_{0}})[-n]\\
& \stackrel{\mu_{\D, *}}\to \r{CC}_{*}(\cc^{op}, \cZ_{X_{0}}^{X_{1}} \otimes_{\cc^{op}} \cZ_{X_{1}}^{X_{0}})[-n],
\end{split}
\end{equation}
where $\mu_{\D} = \mu_{\D, \cZ_{X_{0}}^{X_{1}} \otimes_{\cc^{op}} \cZ_{X_{1}}^{X_{0}}}$ is the collapse map \eqref{bimodule collapse map} for the bimodule $\cZ_{X_{0}}^{X_{1}} \otimes_{\cc^{op}} \cZ_{X_{1}}^{X_{0}}$.
Thus $\pi_{\infty}$ is given as the following composition
\begin{equation}
\begin{split}
& \r{CC}^{*}(\cc^{op}, \cZ_{X_{0}}^{X_{1}}) \otimes \r{CC}^{*}(\cc^{op}, \cZ_{X_{1}}^{X_{0}}) \stackrel{\sqcup}\to \r{CC}^{*}(\cc^{op}, \cZ_{X_{0}}^{X_{1}} \otimes_{\cc^{op}} \cZ_{X_{1}}^{X_{0}}) \\
 \stackrel{- \sqcap \sigma}\to & \r{CC}_{*}(\cc^{op}, (\cc^{op})_{\D} \otimes_{\cc^{op}} \cZ_{X_{0}}^{X_{1}} \otimes_{\cc^{op}} \cZ_{X_{1}}^{X_{0}})[-n]
 \stackrel{\mu_{\D, *}}\to \r{CC}_{*}(\cc^{op}, \cZ_{X_{0}}^{X_{1}} \otimes_{\cc^{op}} \cZ_{X_{1}}^{X_{0}})[-n] \stackrel{\pi^{\cZ}_{*}}\to \K[1-n],
\end{split}
\end{equation}
i.e, 
\begin{equation}\label{pairing on cinf as composition}
\pi_{\infty}(c_{0}, c_{1}) = \pi^{\cZ}_{*}( \mu_{\D, *}( (c_{0} \sqcup c_{1}) \sqcap \sigma )).
\end{equation}

Then \eqref{cinf pairing symmetric 11} follows from the product relation \eqref{cup cap compatible} in Proposition \ref{prop: outer cap product chain map},
the sign twist in defining $\mu^{2}_{\cinf}$ \eqref{mu2 on cinf composition sign},
as well as Lemma \ref{lem: Z pairing associative}.
\end{proof}

By Lemma \ref{lem: cinf canonical pairing system} and Proposition \ref{prop: pairing induced map}, we get the desired chain map of degree $1-n$
\[
\r{res}_{pre}: \r{CC}_{*}(\cinf, (\cinf)_{\D} \otimes_{\cinf} (\cinf)_{\D}) \to \K[1-n]
\]
as proposed in \eqref{res with tensor product coefficients}.
Since quasi-isomorphisms of $\ainf$-bimodules are invertible, 
we choose a homotopy inverse 
\begin{equation}\label{inverse of collapse}
\gamma: (\cinf)_{\D} \to (\cinf)_{\D} \otimes_{\cinf} (\cinf)_{\D}
\end{equation}
 of the collapse map $\mu_{\D, (\cinf)_{\D}}: (\cinf)_{\D} \otimes_{\cinf} (\cinf)_{\D} \to (\cinf)_{\D}$ \eqref{bimodule collapse map},
which induces a quasi-isomorphism 
\begin{equation}\label{inverse of collapse on cc}
\gamma_{*}: \r{CC}_{*}(\cinf, (\cinf)_{\D}) \stackrel{\sim}\to \r{CC}_{*}(\cinf, (\cinf)_{\D} \otimes_{\cinf} (\cinf)_{\D}).
\end{equation}
The composition of \eqref{res with tensor product coefficients} with this map \eqref{inverse of collapse on cc} yields the desire residue map \eqref{res cy}:
\begin{equation}\label{res cy chain map}
\r{res} = \r{res}_{pre} \circ \gamma_{*}: \r{CC}_{*}(\cinf) = \r{CC}_{*}(\cinf, (\cinf)_{\D}) \to \K[1-n].
\end{equation}

Summarizing the above discussion throughout this subsection, we have thus proved:

\begin{prop}\label{prop: res cy chain map}
The map $\r{res} = \r{res}_{pre} \circ \gamma_{*}$ \eqref{res cy chain map} is a chain map of degree $1-n$.
\end{prop} \qed

\subsection{Nondegeneracy of the induced pairing}\label{section: pairing nondegenerate}

Given the chain map of degree $1-n$ 
\[
\r{res}: \r{CC}_{*}(\cinf) \to \K[1-n]
\]
defined by \eqref{res cy chain map}, 
we define a pairing on $\cinf$ as the following composition
\begin{equation}\label{cinf res pairing}
\langle \cdot, \cdot \rangle_{\r{res}}: \cinf^{*}(X, Y) \otimes \cinf^{n-1-*}(Y, X) \stackrel{\mu^{2}_{\cinf}}\to \cinf^{n-1}(Y, Y) \stackrel{i_{Y}}\to \r{CC}_{n-1}(\cinf) \stackrel{\r{res}}\to \K,
\end{equation}
where $i_{Y}: \cinf(Y, Y) \to \r{CC}_{*}(\cinf)$ is the canonical inclusion.

\begin{prop}\label{prop: res pairing nondegenerate}
The pairing $\langle \cdot, \cdot \rangle_{\r{res}}$ \eqref{cinf res pairing} is nondegenerate on the level of cohomology groups.
\end{prop}
\begin{proof}
Observe that the map $i_{Y}: \cinf(Y, Y) \to \r{CC}_{*}(\cinf)$ is the canonical inclusion of chain complexes, i.e., 
the map
\begin{equation}
\cinf^{*}(Y, Y) \to \bigoplus_{\substack{k \ge 0\\X_{0}, \ldots, X_{k} \in \ob \cinf}} \cinf^{*}(X_{k}, X_{0}) \otimes \cinf^{*}(X_{k-1}, X_{k})[1] \otimes \cdots \otimes \cinf^{*}(X_{0}, X_{1})[1],
\end{equation}
including into the $k=0$ piece.

We consider the following diagram
\begin{equation}\label{cd for pairing}
\begin{tikzcd}
& \cinf^{*}(X, Y) \otimes \cinf^{n-1-*}(Y, X) \arrow[r, "\sqcup"] \arrow[d, "\mu_{\cinf}^{2}"] \arrow[rd]  & \r{CC}^{n-1}(\cc^{op}, \cZ_{X}^{Y} \otimes_{\cc^{op}} \cZ_{Y}^{X}) \arrow[r, "- \cap \sigma"] \arrow[ld, "\circ_{\calk_{\K}}"] & \r{CC}_{-1}(\cc^{op}, \cZ_{X}^{Y} \otimes_{\cc^{op}} \cZ_{Y}^{X}) \arrow[d, "\pi^{\cZ}_{*}"] \\
& \cinf^{n-1}(Y, Y) \arrow[d, "i_{Y}"] & \r{CC}_{n-1}(\cinf, (\cinf)_{\D} \otimes_{\cinf} (\cinf)_{\D}) \arrow[r, "\r{res}_{pre}"] \arrow[ld, "\mu_{{\D}, (\cinf)_{\D}, *}"] &\K \\
& \r{CC}_{n-1}(\cinf)
 \end{tikzcd}
\end{equation}
The arrow 
\begin{equation}\label{inclusion of length zero tensors}
\iota: \cinf^{*}(X, Y) \otimes \cinf^{n-1-*}(Y, X) \to \r{CC}_{n-1}(\cinf, (\cinf)_{\D} \otimes_{\cinf} (\cinf)_{\D})
\end{equation}
 is the inclusion of ``length zero'' tensors in the two-sided bar complex $(\cinf)_{\D} \otimes_{\cinf} (\cinf)_{\D}$ for computing $\r{CC}_{n-1}(\cinf, (\cinf)_{\D} \otimes_{\cinf - \cinf} (\cinf)_{\D})$.
By the formula \eqref{bimodule collapse map formula} for the collapse map \eqref{bimodule collapse map},
the left triangle strictly commutes. 
Thus, by choosing the homotopy inverse $\gamma: (\cinf)_{\D} \to (\cinf)_{\D} \otimes_{\cinf} (\cinf)_{\D}$ \eqref{inverse of collapse} whose induced map on Hochschild cochain complexes $\gamma_{*}$ \eqref{inverse of collapse on cc},
the composition
\begin{equation}\label{mu2 to cc diagonal}
\cinf^{*}(X, Y) \otimes \cinf^{n-1-*}(Y, X) \stackrel{\mu_{\cinf}^{2}}\to \cinf^{n-1}(Y, Y) \stackrel{i_{Y}}\to \r{CC}_{n-1}(\cinf) \stackrel{\gamma_{*}}\to \r{CC}_{n-1}(\cinf, (\cinf)_{\D} \otimes_{\cinf} (\cinf)_{\D})
\end{equation}
is chain homotopic to the inclusion of ``length zero'' tensors, $\iota$ \eqref{inclusion of length zero tensors}.

Recall from Corollary \eqref{cor: induced map on cc Z} that the formula \eqref{induced map on cc Z formula} for the chain map $\pi^{\cZ}_{*}: \r{CC}_{*}(\cc^{op}, \cZ_{X_{0}}^{X_{1}} \otimes_{\cc^{op}} \cZ_{X_{1}}^{X_{0}}) \to \K[1]$ \eqref{induced map on cc Z}
only takes nonzero values only for tensors that have zero length in $\mathbf{x}$-entries and $\mathbf{x}'$-entries in the two-sided bar complex for computing $\r{CC}_{*}(\cc^{op}, \cZ_{X_{0}}^{X_{1}} \otimes_{\cc^{op}} \cZ_{X_{1}}^{X_{0}})$.
Going along the top horizontal row then followed by $\pi^{\cZ}_{*}$ gives the pairing $\pi_{\infty, X, Y}$ \eqref{pairing on cinf},
which induces the map $\r{res}_{pre}$ in the middle row, 
whose restriction to ``length zero'' tensors via the map $\iota$ \eqref{inclusion of length zero tensors} is exactly the nondegenerate pairing $\pi_{\infty, X, Y}$ \eqref{pairing on cinf},
\begin{equation}
\pi_{\infty, X, Y} = \r{res}_{pre} \circ \iota.
\end{equation}
On the other hand, the composition of \eqref{mu2 to cc diagonal} with $\r{res}_{pre}$ is by definition the residue pairing $\langle \cdot, \cdot, \rangle_{\r{res}}$, i.e.,
\begin{equation}
\langle \cdot, \cdot, \rangle_{\r{res}} = \r{res}_{pre} \circ \gamma_{*} \circ i_{Y} \circ \mu_{\cinf}^{2}.
\end{equation}
Since $\gamma_{*} \circ i_{Y} \circ \mu_{\cinf}^{2}$ is chain homotopic to $\iota$,
it follows that the cohomology-level pairing induced by $\langle \cdot, \cdot, \rangle_{\r{res}}$ agrees with the induced pairing by $\pi_{\infty, X, Y}$, 
which is nondegenerate.
\end{proof}

\begin{proof}[Proof of Theorem \ref{thm: cinf cy}]
By Proposition \ref{prop: res cy chain map}, the map $\r{res}: \cc_{*}(\cinf) \to \K[1-n]$ \eqref{res cy chain map}  is a chain map.
Its induced pairing is nondegenerate on cohomology groups by Proposition \ref{prop: res pairing nondegenerate}.
Thus the proof of Theorem \ref{thm: cinf cy} is complete.
\end{proof}

\section{Orlov's singularity category on a proper singular scheme}\label{section: singularity category}

\subsection{Some basic definitions}

Let $f: X \to \r{Spec}(\K)$ be a separated scheme of finite type over a field $\K$.
Let $QCoh(X)$ be the dg enhancement of unbounded derived category of quasi-coherent sheaves.
$QCoh(X)$ is compactly generated, 
and the compact objects are exactly the perfect complexes $\perf(X)$.
Let $Coh(X) \subset QCoh(X)$ be the full dg subcategory of objects with bounded coherent cohomology sheaves.
For a smooth $X$, we have $\perf(X) = Coh(X)$;
and for a singular $X$, the inclusions $\perf(X) \subset Coh(X) \subset QCoh(X)$ are strictly proper.
However, by \cite[Theorem 6.3]{lunts}, the dg category $Coh(X)$ is always smooth regardless whether $X$ is smooth or not.

The dg category $\r{Ind}Coh(X)$ ind-coherent sheaves is by definition the ind-completion of $Coh(X)$,
in which the compact objects are $Coh(X)$.
Let $D_{X} = f^{!}\K \in \r{Ind} Coh(X)$ be the dualizing complex,
where 
\[
f^{!}: \r{Mod}_{\K} \to \r{Ind} Coh(X)
\]
is the right adjoint to pushforward 
\[
f_{*}: \r{Ind} Coh(X) \to \r{Mod}_{\K},
\]
which is the unique colimit-preserving extension of $f_{*}: Coh(X) \to \r{Mod}_{\K}$.
We say that $X$ is {\it Cohen-Macaulay} of dimension $n$ if $\omega_{X} = D_{X}[-n]$ is a coherent sheaf.
We say that $X$ is {\it Gorenstein} of dimension $n$ if $\omega_{X}$ is tensor-invertible, i.e., a line bundle.

Suppose now $f: X \to \r{Spec}(\K)$ is proper.
Then $f_{*}: \r{Ind}Coh(X) \to \r{Mod}_{\K}$ sends compact objects to compact objects,
whose right adjoint $f^{!}$ is therefore continuous and sends compact objects to compact objects.
In this case, the dualizing complex $D_{X} = f^{!}\K$ lives in $Coh(X)$.
We say an object $\ee \in Coh(X)$ is $!$-perfect,
if the object $R\mathcal{H}om_{\o_{X}}(\ee, D_{X})$ is perfect.
Denote by $!\!-\! \perf(X)$ the full dg subcategory of $!$-perfect objects.
Clearly, by Grothendieck-Verdier duality,
we have a quasi-equivalence
\[
!\!-\! \perf(X) \simeq \perf(X)^{op}.
\]
Thus we obtain a quasi-equivalence
\begin{equation}
\dd^{b}_{sg}(X)^{op} \simeq Coh(X)/!\!-\! \perf(X).
\end{equation}
In addition, it is observed:

\begin{prop}[{\cite[Proposition B.1]{efimov2}}]
Let $\K$ be a perfect field,
and $X$ a separated scheme of finite type over $\K$.
Then the dg category $\r{Prop} Coh(X)$ is quasi-equivalent to the dg category $!\!-\! \perf(X)_{prop}$ of $!$-perfect complexes with proper support.
\end{prop}

By composing this with the functor from \eqref{cmodprop to cinf2},
 we get a composite functor
\begin{equation}\label{singoptocohinf}
\dd^{b}_{sg}(X)^{op} = \dd^{b}Coh(X)/\perf(X)^{op} \simeq \dd^{b}Coh(X)/!\!-\!\perf(X) \simeq \dd^{b}Coh(X)/\r{Prop}D^{b}Coh(X) \to \widehat{\dd^{b}Coh(X)}_{\infty}.
\end{equation}
Using this, Efimov proved:

\begin{thm}[{\cite[Theorem 9.2]{efimov2}}]\label{thm: singopiscohinf}
The functor $\dd^{b}_{sg}(X)^{op} \to \widehat{\dd^{b}Coh(X)}_{\infty}$ is a quasi-equivalence.
\end{thm}

\subsection{Calabi-Yau structures on the singularity category}

For $X$ a proper scheme of finite type over $\K$,
the dg category $\dd^{b}Coh(X)$ is smooth but not proper,
(which will be proper if $X$ is smooth).
It has the desirable property to support a smooth Calabi-Yau structure when $X$ is Gorenstein:

\begin{prop}[{\cite{brav-dyckerhoff}[Proposition 5.12]}]\label{prop: gorenstein cy}
Let $X$ be a Gorenstein scheme of dimension $n$.
Then giving a strong smooth Calabi-Yau structure on $\dd^{b}Coh(X)$ is equivalent to giving a trivialization of $\omega_{X} \simeq \o_{X}$.
\end{prop}

We can now prove Theorem \ref{thm: sing cy}:

\begin{proof}
By Proposition \ref{prop: gorenstein cy},
a trivialization of $\omega_{X} \simeq \o_{X}$ gives rise to a strong smooth Calabi-Yau structure on $\dd^{b}Coh(X)$ of dimension $n$,
which in particular induces a weak smooth Calabi-Yau structure on $\dd^{b}Coh(X)$ of dimension $n$,
Then Theorem \ref{thm: cinf cy} implies that $\widehat{\dd^{b}Coh(X)}_{\infty}$ has a weak proper Calabi-Yau structure of dimension $n-1$.
By the equivalence \eqref{singoptocohinf} from Theorem \ref{thm: singopiscohinf},
we deduce that $\dd^{b}_{sg}(X)^{op}$ has a weak proper Calabi-Yau structure of dimension $n-1$,
so does $\dd^{b}_{sg}(X)$.
\end{proof}

\section{The Rabinowitz Fukaya category}\label{section: rabinowitz fukaya category}

In this section, we provide a shortcut to the construction of the Rabinowitz Fukaya category.
Still, we set up Floer theory using quadratic Hamiltonians following \cites{abouzaid_gc, GGV}, 
as we find it particular simple for our algebraic argument.

\subsection{Geometric preliminaries}

A {\it Liouville domain} is a compact symplectic manifold-with-boundary with an exact symplectic form $\omega = d\lambda$,
where $\lambda$ is called a {\it Liouville one-form},
such that its dual vector field $Z$, called a {\it Liouville vector field} defined by $\lambda = i_{Z}\omega$,
 is outward pointing along the boundary. 
An exact symplectic manifold$(X, \omega)$ is called a {\it Liouville manifold}, 
if there exists a Liouville domain $X_{0} \subset X$ such that the positive flow of $\p X_{0}$ under $Z$ is defined for all time and the map 
\begin{equation}\label{cylindrical end}
X_{0} \cup_{\p X_{0}} \p X_{0} \times [1, +\infty) \to X
\end{equation}
 is a diffeomorphism.
In other words, $X$ has a positive cylindrical end,
modeled on the contact manifold $\p X_{0}$ equipped with the contact form $\alpha|_{\p X_{0}}$.
Different choices of $\p X_{0}$ are contactomorphic, 
so there is a well-defined contact manifold $(\p_{\infty} X, \xi)$,
which we call the {\it boundary at infinity} of $X$.
We shall assume $c_{1}(X)=0$ for the rest of the paper.

A Lagrangian submanifold $L \subset X$ is {\it exact}, 
if $\lambda|_{L} = df$ for some $f: L \to \R$ called a {\it primitive} for $L$.
$L$ is said to be {\it cylindrical}, if outside of a compact set, $L$ is invariant under the flow of $Z$.
These conditions imply that an exact cylindrical submanifold $L$ also has a positive cylindrical end
\begin{equation}
L = L_{0} \cup_{\p L_{0}} \p L_{0} \times [1, +\infty),
\end{equation}
modeled on the Legendrian submanifold $l = \p L_{0} \subset \p X_{0}$.
Moreover, the primitive $f_{L}: L \to \R$ is locally constant in the cylindrical end.
The Legendrian isotopy class is independent of choices of cylindrical ends,
so there is a well-defined Legendrian submanifold $\p_{\infty} L \subset \p_{\infty} X$,
which we call the {\it boundary at infinity} of $L$.
Suppose $L$ comes with a grading in the sense of \cite{seidel_graded},
and has a Spin structure.

We will fix, once for all, an at most countable collection 
\begin{equation}\label{collection of objects}
\L = \{L_{1}, L_{2}, \ldots | L_{i} \text{ exact, cylindrical, equipped with a grading and a Spin structure}\}.
\end{equation}
 of exact cylindrical Lagrangian submanifolds, 
 for which there exists a generic choice of Liouville form such that
\begin{equation}\label{non-degenerate Reeb dynamics}
\begin{split}
& \text{ all Reeb orbits of } \alpha \text{ are non-degenerate, and } \\
 &\text{all Reeb chords from the Legendrian } \p_{\infty}L_{i} \text{ to } \p_{\infty}L_{j} \text{ are non-degenerate}, \forall L_{i}, L_{j} \in \L.
\end{split}
\end{equation}

A Hamiltonian $H: X \to \R$ is said to be {\it admissible} or {\it quadratic at infinity},
if outside of a compact set, $H$ only depends on the radial coordinate on $\p X_{0} \times [1, +\infty)$ via the identification \eqref{cylindrical end},
and satisfies 
\begin{equation}
H(y, r) = r^{2}.
\end{equation}
The space of admissible Hamiltonians is denoted by 
\begin{equation}
\mathcal{H}(X).
\end{equation}

An $\omega$-compatible almost complex structure $J$ is said to be {\it $h$-rescaled contact type} on the cylindrical end,
if it satisfies 
\begin{equation}
\label{h rescaled contact type}
\frac{h}{r} \lambda \circ J = dr
\end{equation}
for some $h>0$ on the cylindrical end.
The space of all $h$-rescaled contact type almost complex structures is denoted by
\begin{equation}
\mathcal{J}_{h}(X).
\end{equation}

\subsection{Popsicles}\label{section: popsicles}

In this subsection we briefly review the definition of popsicles, 
originally introduced in \cite{abouzaid-seidel},
and adapted in \cite{GGV} to construct the $\ainf$-structure on the Rabinowitz Fukaya category.

Let $k \ge 1$ be a positive integer.
Let $\mathbf{p}: F \to \{1, \ldots, d\}$ a map from a finite set,
and put $p_{f} = \mathbf{p}(f)$ for $f \in F$.
Let $S$ be a connected bordered Riemann surface of genus zero with $k+1$ boundary punctures $z_{0}, z_{1}, \ldots z_{k}$ ordered cyclically along the boundary.

\begin{defn}\label{def: popsicles}
A popsicle structure $\eta$ on $S$ of flavor $\mathbf{p}$ consists of a collection of choices of preferred points $\{\eta_{f}\}_{f \in F}$,
with each $\eta_{f}$ on the unique hyperbolic geodesic $C_{p_{f}}$ connecting $z_{0}$ to $z_{p}$.

We call the pair $(S, \eta)$ a popsicle of flavor $\mathbf{p}$.
\end{defn}

A popsicle $(S, \eta)$ of flavor $\mathbf{p}$ is {\it stable} if $k+|F| \ge 2$.
The stable popsicles form a moduli space 
\begin{equation}
\cR^{k+1, \mathbf{p}, \mathbf{w}},
\end{equation}
which can be compactified by adding broken popsicles.
Consider a ribbon tree $T$ with $k$ leaves and one root,
such that the valency $|v|$ of each vertex $v \in V(T)$ is at least two.
Let $\vec{F} = \{F_{v}\}_{v \in V(T)}$ be a decomposition of $F$ indexed vertices of $T$,
and $\mathbf{p}_{v}: F_{v} \to \{1, \ldots, |v|-1\}$ be the map induced by $\mathbf{p}$.
A broken popsicle of type $T$ of flavor $\vec{\mathbf{p}} = \{\mathbf{p}_{v}\}_{v \in V(T)}$ is a collection of popsicles $\{(S_{v}, \eta_{v})\}_{v \in V(T)}$,
each of flavor $F_{v}$ for each vertex $v$.
A broken popsicle is stable if every component $(S_{v}, \eta_{v})$ is stable, which means $|v|+|F_{v}| \ge 3$.
The stable ones form a moduli space
\begin{equation}
\cR^{T, \vec{\mathbf{p}}} = \prod_{v \in V(T)} \cR^{|v|, \mathbf{p}_{v}}.
\end{equation}
Define the moduli space of broken popsicles with $k$ leaves and one root to be
\begin{equation}
\bar{\cR}^{k+1, \mathbf{p}} = \coprod_{\substack{T, \vec{F}\\ T \text{ has } k \text{ leaves and one root }}} \cR^{T, \vec{\mathbf{p}}}.
\end{equation}
This is a compact manifold with corners (\cite{abouzaid-seidel}).
There is a natural action of the group of automorphisms $\r{Aut}(\mathbf{p}) \subset \r{Sym}(F)$ of $F$ preserving $\mathbf{p}$ on the moduli space,
which is a trivial action if $\mathbf{p}$ is injective.

Now we assign {\it weights} to popsicles. 
Consider non-negative integers $w_{i} \in \Z_{\le 0}, i = 0, 1,\ldots, k$ satisfying
\begin{equation}\label{weight condition 1}
w_{0} = w_{1} + \cdots + w_{k} + |F|
\end{equation}
and
\begin{equation}\label{weight condition 2}
|\mathbf{p}^{-1}(i) \le -w_{i}, \forall i=1,\ldots, k.
\end{equation}
Assign the weight $w_{i}$ to the puncture $z_{i}$ of a popsicle $(S, \eta)$ of flavor $\mathbf{p}$.
Put $\mathbf{w} = (w_{0}, \ldots, w_{k})$. 
Denote the resulting moduli space with this decoration by $\cR^{k+1, \mathbf{p}, \mathbf{w}}$.

The weights are automatically inherited by broken popsicles,
such that both \eqref{weight condition 1} and \eqref{weight condition 2} are satisfied on each smooth component $S_{v}$ labeled by some vertex $v \in V(T)$.
Denote by $\mathbf{w}_{v}$ the collection of weights on that component.
The compactified moduli space of broken popsicles is
\begin{equation}\label{compactified moduli space of popsicles}
\bar{\cR}^{k+1, \mathbf{p}, \mathbf{w}} = \coprod_{\substack{T, \vec{F}\\ T \text{ has } k \text{ leaves and one root }}} \cR^{T, \vec{\mathbf{p}}, \vec{\mathbf{w}}}
= \coprod_{\substack{T, \vec{F}\\ T \text{ has } k \text{ leaves and one root }}} \prod_{v \in V(T)} \mathcal{R}^{|v|, \mathbf{p}_{v}, \mathbf{w}_{v}}.
\end{equation}

We shall be only interested in the case where $w_{i} \in \{-1, 0\}$.
Although this condition is not always preserved when passing to boundary strata of the compactified moduli space $\bar{\cR}^{k+1, \mathbf{p}, \mathbf{w}}$,
but we will observe:

\begin{lem}[{\cite[Lemma 4.2]{seidel6}}]
Let $\mathbf{w}$ be a collection of weights satisfying $w_{i} \in \{-1, 0\}$.
Consider a broken popsicle in the codimension-one boundary stratum of $\bar{\cR}^{k+1, \mathbf{p}, \mathbf{w}}$ modeled on a tree $T$ with two vertices $v_{2}, v_{1}$ with an edge going from $v_{2}$ to $v_{1}$.
Suppose there is an inherited weight that does not belong to $-1, 0$, i.e., for some $i$ we have
\begin{equation}
w_{v_{1}, i} = w_{v_{2}, 0} < -1.
\end{equation}
Then exactly one of the following is true:
\begin{enumerate}[label=(\roman*)]

\item $|\mathbf{p}_{v_{1}}^{-1}(i)| \ge 2$, or

\item $w_{v_{1}, 0 } = w_{0} = -1, w_{v_{1}, i} = w_{v_{2}, 0} = -2$, 
and $\mathbf{p}_{v_{1}}^{-1}(i) = \{f\}$ is a singleton set.
In this case, there are exactly two possible values for $k = \mathbf{p}(j_{v_{1}}(f))$ satisfying $w_{k} = w_{v_{2}, k-i+1} = -1$ and $\mathbf{p}_{v_{2}}^{-1}(k-i+1) = \varnothing$.

\end{enumerate}
\end{lem}

\subsection{Pseudoholomorphic maps from popsicles}

Fix a quadratic Hamiltonian $H \in \mathcal{H}(X)$ and a time-dependent almost complex structure $J_{t}: [0, 1] \to \mathcal{J}_{1}(X)$.
To define pseudoholomorphic maps from popsicles to the target Liouville manifold,
we shall introduce {\it Floer data} on popsicles.
The reader is referred to \cites{abouzaid_gc, ganatra, GGV} for more detailed discussions regarding Floer data (in which the Hamiltonians are quadratic),
but for completeness and easy references in later sections, 
we shall include the basic definitions here.
The weight $w_{i}$ determines a rule for assigning either a positive or a negative strip-like end near the puncture $z_{i}$. 
To describe the rule, we introduce a function
\begin{equation}\label{sign function}
\d_{i} = 
\begin{cases}
0, &\text{ if } i =  0, w_{i} < 0 \text{ or } i = 1, w_{i} = 0, \\
-1, &\text{ if } i = 0, w_{i} = 0 \text{ or } i = 1, w_{i} < 0.
\end{cases}
\end{equation}
Also define the associated sign symbols
\begin{equation}\label{sign symbol}
\sk_{i}= 
\begin{cases}
+, &\text{ if } i =  0, w_{i} < 0 \text{ or } i = 1, w_{i} = 0, \\
-, &\text{ if } i = 0, w_{i} = 0 \text{ or } i = 1, w_{i} < 0.
\end{cases}
\end{equation}
Denote by
\begin{align}
Z^{+} & = [0, +\infty) \times [0, 1], \\
Z^{-} & = (-\infty, 0] \times [0, 1].
\end{align}
the positive and the negative infinite half-strips.

\begin{defn}\label{def: Floer data}
A Floer datum $\mathbf{D}_{S, \eta}$ for a $\mathbf{w}$-weighted popsicle $(S, \eta)$ of flavor $\mathbf{p}$ consists of:
\begin{enumerate}[label=(\roman*)]

\item A collection of real numbers $\nu_{i} \ge 1, i=0, \ldots, k$ satisfying
\begin{equation}
\sum_{i=0}^{k} (-1)^{\d_{i}} \nu_{i} \le 0.
\end{equation}

\item A collection of strip-like ends 
\begin{equation} 
\e_{i}: Z^{\sk_{i}}: \to S, i=0, \ldots, k.
\end{equation}
These depend on the weights $\mathbf{w}$ following the rule by \eqref{sign symbol}.

\item A rescaling function $\rho_{S}: S \to [1, +\infty)$ that is constant and equal to $\nu_{i}$ over the strip-like end $\e_{i}$, for every $i=0, \ldots, k$.

\item A one-form $\alpha_{S}$ on $S$ such that $\alpha_{S}|_{\p S} = 0, d \alpha_{S} \le 0$ and compatible with strip-like ends in the sense that
\begin{equation}
\e_{i}^{*} \alpha_{S} = \nu_{i} dt.
\end{equation}

\item A family of Hamiltonians $H_{S}: S \to \mathcal{H}(X)$ compatible with strip-like ends in the sense that
\begin{equation}
\e_{i}^{*}H_{S} = \frac{H \circ \psi^{\nu_{i}}}{\nu_{i}^{2}}, i = 0, \ldots, k.
\end{equation}

\item A domain-dependent family of almost complex structures $J_{S}$ such that at every point $z \in S$, $J_{S, z} \in \mathcal{J}_{\rho_{S}(z)}(X)$, and compatible with strip-like ends in the sense that
\begin{equation}
\e_{i}^{*}J_{S} = (\psi^{\nu_{i}})^{*} J_{t}, i = 0, \ldots, k.
\end{equation}

\end{enumerate}
In addition, the data $\rho_{S}, \alpha_{S}, H_{S}, J_{S}$ must be invariant under the action of $\r{Aut}(\mathbf{p})$,
which is naturally lifted to the space of functions on $S$ valued in $[1, +\infty)$, the space of one-forms on $S$,
the space of smooth maps from $S$ to $\mathcal{H}(X)$, 
and the space of smooth maps from $S$ to $\mathcal{J}(X)$.
We call such a Floer datum $\r{Aut}(\mathbf{p})$-invariant.
\end{defn}

We can also extend the definition to an unstable domain $Z$ by requiring 
\begin{equation}
\rho_{Z} \equiv 1, \alpha_{Z} \equiv dt, H_{Z} \equiv H \text{ and } J_{Z} \equiv J_{t}.
\end{equation}

\begin{defn}\label{def: conformally equivalent Floer data}
Two Floer data $\mathbf{D}_{S, \eta}, \mathbf{D}'_{S, \eta}$ are said to be conformally equivalent, 
if there exist $K>0, C>0$ such that
\begin{align}
\nu'_{i} & = K \nu_{i}, \\
\rho'_{S} & = K \rho_{S}, \\
\alpha'_{S} & = K \alpha_{S}, \\
H'_{S} &= \frac{H_{S} \circ \psi^{K}}{K^{2}} + C, \\
J'_{S} &= (\psi^{K})^{*} J_{S}.
\end{align}
\end{defn}

To ensure that the moduli spaces of pseudoholomorphic maps defined with respect to choices of Floer data as above have good compactifications, we introduce the following notion, following \cite{abouzaid_gc} (see also \cite[\S 3.5]{GGV}):

\begin{defn}\label{def: universal and conformally consistent choice of Floer data}
A universal and conformally consistent choice of Floer data $\mathbf{D}_{\cR}$ for all disks with popsicle structures is a choice of Floer data,
one for each representative $(S, \eta)$ of element in $\bar{\cR}^{k+1, \mathbf{p}, \mathbf{w}}$, 
being $\r{Aut}(\mathbf{p})$-invariant and smoothly varying over $\cR^{k+1, \mathbf{p}, \mathbf{w}}$, 
which at boundary strata agree with a product Floer data chosen for lower dimensional moduli spaces up to conformal equivalence.
\end{defn}

All the choices as part of a Floer datum $D_{S, \eta}$ as in Definition \ref{def: Floer data} form a contractible space;
so by a standard induction argument, a universal and conformally consistent choice of Floer data $\mathbf{D}_{\cR}$ for all disks with popsicle structures exists; see \cites{abouzaid_gc, seidel_book, ganatra}, etc.

For boundary conditions of the pseudoholomorphic maps, we introduce:

\begin{defn}
A Lagrangian label for $(S, \eta)$ is a collection of exact cylindrical Lagrangian submanifolds $L_{i} \in \L, i = 0, \ldots, k$ where $\L$ is defined in \eqref{collection of objects},
with each $L_{i}$ assigned to the component of $\p S$ between the punctures $z_{i}$ and $z_{i+1}$, $\r{mod}$  $k+1$ cyclically.
\end{defn}

For $i = 1, \ldots, k$, let 
\begin{equation}
x_{i} \in 
\begin{cases}
\chi(L_{i-1}, L_{i}; H), & \text{ if } \d_{i} = 0, \\
\chi(L_{i}, L_{i-1}; H), & \text{ if } \d_{i} = -1,
\end{cases}
\end{equation}
For $i=0$, let 
\begin{equation}
x_{0} \in 
\begin{cases}
\chi(L_{0}, L_{k}; H), & \text{ if } \d_{0} = -1, \\
\chi(L_{k}, L_{0}; H), & \text{ if } \d_{0} = 0.
\end{cases}
\end{equation}
Here the functions $\d_{i}$ are defined in \eqref{sign function}.

A pseudoholomorphic map from a popsicle is a map $u: S \to X$ satisfying {\it homogeneous Cauchy-Riemann equation}
\begin{equation}\label{cr equation for popsicles}
\begin{cases}
& (du-X_{H_{S}} \otimes \alpha_{S})^{0, 1} = \frac{1}{2} [ (du-X_{H_{S}} \otimes \alpha_{S}) + J_{S} \circ (du-X_{H_{S}} \otimes \alpha_{S}) \circ j] = 0, \\
& u(z) \in (\psi^{\rho_{S}(z)})^{*} L_{i}, z \in \p S, i = 0, \ldots, k \text{ exponentially},\\
& \lim\limits_{s \to \sk_{i}\infty} u \circ \e_{i}(s, \cdot) = (\psi^{\nu_{i}})^{*} x_{i} \text{ exponentially}.
\end{cases}
\end{equation}

Put $\mathbf{x} = (x_{0}, x_{1}, \ldots, x_{k})$.
Let 
\begin{equation}\label{moduli space of popsicles}
\cR^{k+1, \mathbf{p}, \mathbf{w}}(\mathbf{x})
\end{equation}
be the moduli space of pseudoholomorphic maps from popsicles, i.e., maps $u: S \to X$ satisfying \eqref{cr equation for popsicles}.

\begin{lem}[{\cite[Lemma 3.5]{GGV}}]
The virtual dimension of the moduli space $\cR^{k+1, \mathbf{p}, \mathbf{w}}(\mathbf{x})$ \eqref{moduli space of popsicles} is
\begin{equation}
v-\dim \cR^{k+1, \mathbf{p}, \mathbf{w}}(\mathbf{x}) = k - 2 + |F| + n(1+\sum_{i=0}^{k} \d_{i}) - \sum_{i=0}^{k} (-1)^{\d_{i}} |x_{i}|.
\end{equation}
\end{lem}

Suppose we have made a universal and conformally consistent choice of Floer data $\mathbf{D}_{\cR}$ for all disks with popsicle structures.
The moduli space $\cR^{k+1, \mathbf{p}, \mathbf{w}}(\mathbf{x})$ \eqref{moduli space of popsicles} has a natural Gromov compactification 
\begin{equation}
\bar{\cR}^{k+1, \mathbf{p}, \mathbf{w}}(\mathbf{x})
\end{equation}
by adding stable maps from broken popsicles modeled on ribbon trees carrying non-positive weights,
as well as inhomogeneous pseudoholomorphic strips which are elements of $\cR^{2, \varnothing, (0, 0)}(x_{0}, x_{1})$ and $\cR^{2, \varnothing, (-1, -1)}(x_{0}, x_{1})$.
In particular, the codimension-one boundary strata are covered by
\begin{align}
\cR^{j+1, \mathbf{p}_{1}, \mathbf{w}_{1}}(\tilde{x}, x_{i+1}, \ldots, x_{i+j}) & \times \cR^{k-j+2, \mathbf{p}_{2}, \mathbf{w}_{2}}(x_{0}, x_{1}, \ldots, x_{i}, \tilde{x}, x_{i+j+1}, \ldots, x_{k}), \\
\cR^{2, \varnothing, (w_{i}, w_{i})}(\tilde{x}_{i}, x_{i}) & \times \cR^{k+1, \mathbf{p}, \mathbf{w}}(x_{0}, \ldots, x_{i-1}, \tilde{x}_{i}, x_{i+1}, \ldots, x_{k}), \\
 \cR^{k+1, \mathbf{p}, \mathbf{w}}(\tilde{x}_{0}, x_{1}, \ldots, x_{k}) & \times \cR^{2, \varnothing, (w_{0}, w_{0})}(x_{0}, \tilde{x}_{0}),
\end{align}
for some `intermediate' chords $\tilde{x}, \tilde{x}_{i}, \tilde{x}_{0}$.
Boundary strata of higher codimensions can be described inductively.

The following statement summarizes the above description of the structure of the compactified moduli spaces,
where the proof for transversality is by now standard using a Sard-Smale argument, following e.g. \cite{seidel_book}.

\begin{prop}[{\cite[Proposition 3.6]{GGV}}]\label{prop: transversality and compactness of popsicles}
There exists a universal and conformally consistent choice of Floer data $\mathbf{D}_{\cR}$ such that the following holds.
\begin{enumerate}[label=(\roman*)]

\item If $k - 2 + |F| + n(1+\sum_{i=0}^{k} \d_{i}) - \sum_{i=0}^{k} (-1)^{\d_{i}} |x_{i}| = 0$, 
the moduli space $\bar{\M}_{\oc}^{k, \mathbf{p}_{oc, k}, \mathbf{w}_{oc, k}}(y; \mathbf{x})$ \eqref{compactified moduli space for oc} is a compact smooth manifold of dimension zero.

\item If $k - 2 + |F| + n(1+\sum_{i=0}^{k} \d_{i}) - \sum_{i=0}^{k} (-1)^{\d_{i}} |x_{i}| = 1$, 
the moduli space $\bar{\M}_{\oc}^{k, \mathbf{p}_{oc, k}, \mathbf{w}_{oc, k}}(y; \mathbf{x})$ \eqref{compactified moduli space for oc} is a compact smooth manifold-with-boundary of dimension one,
with boundary strata given by products of zero-dimensional moduli spaces.

\item If $k - 2 + |F| + n(1+\sum_{i=0}^{k} \d_{i}) - \sum_{i=0}^{k} (-1)^{\d_{i}} |x_{i}|=2$,
the moduli space $\bar{\M}_{\oc}^{k, \mathbf{p}_{oc, k}, \mathbf{w}_{oc, k}}(y; \mathbf{x})$ \eqref{compactified moduli space for oc} is a compact smooth manifold-with-boundary of dimension two,
with codimension-one boundary strata given by products of zero-dimensional and one-dimensional moduli spaces,
and codimension-two boundary strata given by products of zero-dimensional moduli spaces. 

\end{enumerate}
\end{prop}

\subsection{Floer complexes}

Choose an admissible Hamiltonian $H \in \mathcal{H}(X)$.
Let $L_{0}, L_{1}$ be a pair of exact cylindrical Lagrangian submanifolds of $X$.
Consider the set
\begin{equation}
\chi(L_{0}, L_{1}; H)
\end{equation}
of all time-one $H$-chords from $L_{0}$ to $L_{1}$, 
which are maps $x: [0, 1] \to X$ satisfying Hamilton's equation:
\begin{equation}
\begin{cases}
&\frac{dx}{dt} + X_{H}(x(t)) = 0, \\
&x(0) \in L_{0}, x(1) \in L_{1}.
\end{cases}
\end{equation}
Each time-one $H$-chord $x$ a well-defined Maslov index, 
which we will denote by $|x| \in \Z/2$.
When $c_{1}(X) = 0$ and the Lagrangians come with gradings, 
this admits an integral lift $|x| \in \Z$.
To each chord $x \in \mathcal{X}(L_{0}, L_{1}; H)$, denote by $o_{x}$ its orientation line,
and $|o_x|_{\K}$ the associated $\K$-normalized orientation space.  

Define
\begin{equation}\label{Floer cochains}
CW^{*}(L_{0}, L_{1}; H) = \bigoplus_{x \in \chi(L_{0}, L_{1}; H)} |o_x|_{\K}[-\deg(x)],
\end{equation}
where the degree shift means that we are putting $|o_x|_{\K}[-\deg(x)]$ in degree zero, 
or equivalently $|o_x|_{\K}$ in degree $\deg(x)$.
Elements of this graded $\K$-modules are called the {\it wrapped Floer cochains}.
The differential is defined by counting rigid elements in the moduli spaces of Floer trajectories $\cR^{2}(x_{0}, x_{1})$, 
\begin{align}
d:  CW^{*}(L_{0}, L_{1}; H) & \to CW^{*+1}(L_{0}, L_{1}; H),\\
d([x_{1}])  = & \sum_{\substack{x_{0}\\|x_{0}|=|x_{1}|+1}} \sum_{u \in \cR^{2, \varnothing, (0, 0)}(x_{0}, x_{1})} o_{u}([x_{1}]),
\end{align}
where $o_{u}: o_{x_{1}} \stackrel{\cong}\to o_{x_{0}}$ is the induced isomorphism of orientation lines by each rigid element in the moduli space (see \cites{seidel_book, abouzaid_gc}),
and we denote the induced isomorphism on $\K$-normalized orientation space by the same notation.

Now consider the negative Hamiltonian $-H$.
It is not admissible in the usual sense, but we can still define a Floer complex for it as follows.
The negative reparametrization $\bar{x}(t):=x(1-t)$ provides a canonical one-to-one correspondence between time-one $(-H)$-chords $x \in \chi(L_{0}, L_{1}; -H)$ of Maslov index $i$ to time-one $H$-chords $\bar{x} \in \chi(L_{1}, L_{1}; H)$ of Maslov index $n-i$.
Let $o^{-}_{\bar{x}}$ be the negative orientation line for the chord $\bar{x}$,
such that there is a canonical isomorphism $o_{x} \cong o^{-}_{\bar{x}}$.
Define
\begin{equation}\label{Floer chains}
CW^{*}(L_{0}, L_{1}; -H) = \prod_{x \in \chi(L_{0}, L_{1}; -H)} |o_{x}|_{\K}[-\deg(x)] = \prod_{\bar{x} \in \chi(L_{1}, L_{0}; H)} |o^{-}_{\bar{x}}|_{\K}[n-\deg(\bar{x})].
\end{equation}
Elements of this graded $\K$-modules are called the {\it wrapped Floer chains}.

\begin{lem}\label{lem: PD for w-}
There is a canonical isomorphism of chain complexes, called the Poincar\'{e} duality isomorphism,
\begin{equation}\label{PD for w-}
I: CW^{*}(L_{0}, L_{1}; -H) \to CW^{*}(L_{1}, L_{1}; H)^{\vee}[-n] = \hom_{\K}(CW^{n-*}(L_{1}, L_{0}; H), \K).
\end{equation}
\end{lem}
\begin{proof}
On generators, it is exactly given by the identification $x \mapsto \bar{x}$.
By choosing the family of almost complex structures for $-H$ to agree with the ones chosen for $H$ when choosing Floer data,
we also get a natural identification of moduli spaces of Floer trajectories for $-H$ and those for $H$.
Thus $I$ intertwines $d$ and $\p$ as well, i.e., is a chain map.
\end{proof}

An important feature in Floer theory is the existence of {\it a continuation map}
\begin{equation}\label{continuation map}
c: CW^{*}(L_{0}, L_{1}; -H) \to CW^{*}(L_{0}, L_{1}; H)
\end{equation}
associated to a monotone homotopy of Hamiltonians.
This is defined by counting rigid elements in the moduli space $\cR^{2, \{1\}, (0, -1)}(x_{0}, x_{1})$, i.e.,
\begin{equation}\label{continuation map formula}
c([x_{1}^{-}]) = \sum_{\substack{x_{0}\\|x_{0}| = n - |x_{1}|}} \sum_{u \in \cR^{2, \{1\}, (0, -1)(x_{0}, x_{1})}} \cR^{2, \{1\}, (0, -1)}_{u}([x_{1}^{-}]),
\end{equation}
where we take the input $[x_{1}^{-}] \in CW^{*}(L_{0}, L_{1}; -H)$ to be the generator associated with the chord $x_{1}^{-} \in \chi(L_{0}, L_{1}; -H)$ which uniquely corresponds to the chord $x_{1} \in \chi(L_{1}, L_{0}; H)$ by negative reparametrization.
Note that the wrapped Floer chain complex $CW^{*}(L_{0}, L_{1}; -H)$ is a direct product, 
we must therefore show that the map is well-defined.
This can be done by the following lemma.

\begin{lem}\label{lem: finitely many continuation moduli spaces}
Let $x_{0} \in \chi(L_{1}, L_{0}; H)$ and $x_{1} \in \chi(L_{0}, L_{1}; H)$ satisfy $|x_{0}|+|x_{1}|=n$. 
There are only finitely many pairs $(x_{0}, x_{1})$ for which the moduli space $\cR^{2, \{1\}, (0, -1)}(x_{0}, x_{1})$ is non-empty.
\end{lem}
\begin{proof}
If there exists some $u \in \cR^{2, \{1\}, (0, -1)}(x_{0}, x_{1})$,
the domain of $u$ has only two negative punctures (outputs) and no positive punctures (input),
this implies that the sum of the actions of $x_{0}$ and $x_{1}$ is positive.
However, all but finitely many $H$-chords have negative action (\cite[Lemma B.2]{abouzaid_gc}).
\end{proof}

\begin{cor}\label{cor: continuation map well-defined}
The continuation map \eqref{continuation map} is well-defined.
\end{cor}
\begin{proof}
Lemma \ref{lem: finitely many continuation moduli spaces} implies that $c([x_{1}^{-}]) = 0$ for all but finitely many $x_{1}^{-} \in \chi(L_{0}, L_{1}; -H)$,
and that for each such $x_{1}^{-}$, the output $c([x_{1}^{-}])$ is a finite sum as well.
\end{proof}

The Rabinowitz Floer complex is defined to be the mapping cone
\begin{equation}\label{rc}
RC^{*}(L_{0}, L_{1}) = \cone(c: CW^{*}(L_{0}, L_{1}; -H) \to CW^{*}(L_{0}, L_{1}; H)).
\end{equation}

As a by product of the definitions, we immediately see that the Rabinowitz Floer complex defined in \eqref{rc} is a locally linear compact vector space in the sense of Lefschetz \cite{lefschetz}.

\begin{prop}\label{prop: rab tate}
Suppose $\K$ is a field. 
The Rabinowitz Floer complex $RC^{*}(L_{0}, L_{1})$ is a locally linearly compact vector space over $\K$.
\end{prop}
\begin{proof}
Since the wrapped Floer cochain space \eqref{Floer cochains} is discrete,
it follows from Lefschetz-Tate duality (\cite[II. (28.2-29.1)]{lefschetz}, \cite[Theorem 1.25]{rojas}) that the wrapped Floer chain space \eqref{Floer chains} is linear compact because of the duality \eqref{PD for w-}.
By definition, a locally linearly compact vector space is a direct sum of a discrete vector space with a locally linear compact vector space.
\end{proof}

In fact, this is a type of locally linearly compact vector space, or equivalently a Tate vector space in the sense of \cite{Beilinson-Feigin-Mazur}, which is self dual (\cite[Proposition 3.15]{CO_tate}),
which in our context will imply

\begin{cor}
The topological linear dual of the Poincar\'{e} duality isomorphism $I$ \eqref{PD for w-} is still an isomorphism, 
and combined with $I$ itself gives an isomorphism between $RC^{*}(L_{0}, L_{1})$ and its topological linear dual,
which is itself.
\end{cor}

While we expect an interesting category theory enriched in Tate vector spaces,
we do not plan to address this further in this paper as our main results are independent of it.

\subsection{$\ainf$-structures}

Define maps
\begin{equation}\label{ainf maps for rw}
\mu^{k}_{\rw}: RC^{*}(L_{k-1}, L_{k}) \otimes \cdots RC^{*}(L_{0}, L_{1}) \to RC^{*}(L_{0}, L_{k})
\end{equation}
by countign rigid elements in the moduli space of pseudoholomorphic maps from popsicles.
This is defined first on components as maps of the form
\begin{equation}\label{components for rw ainf maps}
\mu^{k, \mathbf{p}, \mathbf{w}}: CW^{*}(L_{k-1}, L_{k}; \sk_{k}H) \otimes \cdots \otimes CW^{*}(L_{0}, L_{1}; \sk_{1}H) \to CW^{*}(L_{0}, L_{k}; \sk_{0}H),
\end{equation}
where the symbols $\sk_{i}$ are defined in \eqref{sign symbol}.
On a basis of elements the map \eqref{components for rw ainf maps} takes the form:
\begin{equation}\label{ainf maps for rw basis}
\mu^{k}_{\rw}([x_{k}^{\sk_{k}}], \ldots, [x_{1}^{\sk_{1}}]) = \prod_{x_{0}} \sum_{u \in \cR^{k+1, \mathbf{p}, \mathbf{w}}(\vec{x})} (-1)^{*_{k, \mathbf{p}, \mathbf{w}} + \diamond_{k, \mathbf{p}, \mathbf{w}}} o_{u}([x_{k}^{\sk_{k}}] \otimes \cdots \otimes [x_{1}^{\sk_{1}}]),
\end{equation}
where the signs are
\begin{align}
*_{k, \mathbf{p}, \mathbf{w}} = & \sum_{i=1}^{k} (i + w_{1} + \ldots + w_{i-1}) |x_{i}^{\sk_{i}}| + \sum_{i=1}^{k} (k-i) w_{i}, \label{popsicle sign 1} \\
 \diamond_{k, \mathbf{p}, \mathbf{w}} = & \sum_{i=1}^{k} |\mathbf{p}^{-1}(\{i+1, \ldots, k\})|(w_{i} + |\mathbf{p}^{-1}(i)|). \label{popsicle sign 2}
\end{align}
Technically, each rigid element $u$ would induce an isomorphism of orientation lines associated with the rescaled chords $(\psi^{\nu_{i}})^{*}x_{i}$, but there are canonical isomorphisms $o_{(\psi^{\nu_{i}})^{*}x_{i}} \cong o_{x_{i}}$ induced by the Liouville rescaling (\cite{abouzaid_gc}); similarly for the negative orientation lines.
The maps \eqref{ainf maps for rw basis} extend to well-defined maps \eqref{ainf maps for rw} by \cite[Lemma 4.8]{GGV}, 
which satisfy the $\ainf$-equations by \cite[Lemma 4.9]{GGV}.

\begin{defn}
The Rabinowitz Fukaya category $\rw = \rw(\L)$ is defined to be the $\ainf$-category having objects being exact cylindrical Lagrangian usbmanifolds in the collection $\L$ \eqref{collection of objects},
morphism spaces $\hom_{\rw}(L_{0}, L_{1}) = RC^{*}(L_{0}, L_{1})$,
and $\ainf$-structure maps $\mu^{k}_{\rw}$ \eqref{ainf maps for rw}.

When no confusion may occur, we shall refer to this as the Rabinowitz Fukaya category of $X$,
and denote it by $\rw = \rw(X)$.
\end{defn}

By construction, the Rabinowitz Fukaya categories $\rw$ comes with a natural strict $\ainf$-functor 
\begin{equation}\label{w to rw}
j: \w \to \rw
\end{equation}
which is the identity on objects,
has $j^{1}: CW(K, L; H) \to RC^{*}(K, L)$ the obvious inclusion map,
and $j^{k} = 0$ for all $k \ge 2$.
This functor sends compact exact Lagrangians to zero,
and more generally all {\it right proper objects} (meaning those whose right Yoneda modules are proper modules) to zero.
In this sense, the non-triviality of $\rw$ measures the failure of $\w$ being proper.

\subsection{Relations to the wrapped Fukaya category} \label{section: rw relation to w}

In addition to the functor $j$ \eqref{w to rw}, we have the following result establishing a deeper connection between the Rabinowitz Fukaya category and the wrapped Fukaya category:

\begin{thm}[{\cite[Theorem 1.1]{GGV}}]\label{thm: rw=winf}
There is a canonical $\ainf$-functor
\begin{equation}\label{rw to winf2}
\Phi: \rw \to \winf
\end{equation}
which is a quasi-equivalence whenever $X$ is a non-degenerate Liouville manifold with $c_{1}(X)=0$.
\end{thm}

Below we shall list some of the key properties that are both relevant to the proof of Theorem \ref{thm: rw=winf} and also useful in other places of this paper, 
while refer the reader to \cite{GGV} for full details of the construction of the functor \eqref{rw to winf2} and the proof that it is a quasi-equivalence.

Recall from Corollary \ref{cor: cinf as cone} that $\winf$ naturally has the structure of a bimodule over $\w^{op}$,
together with a quasi-isomorphism
\begin{equation}\label{winf as cone}
\winf \cong \cone(\r{CC}^{*}(\w^{op}, (\w^{op})^{\vee} \otimes_{\K} \w_{\D}^{op}) \to \r{CC}^{*}(\w^{op}, \hom_{\K}(\w_{\D}, \w_{\D}))).
\end{equation}
Similarly, we can regard $\rw$ as a $\w^{op}$-bimodule via the natural functor $j: \w \to \rw$ \eqref{w to rw},
by restricting all except for one (the main input) of the inputs to $\w$ via the natural functor $j$ (the resulting one would be a bimodule over $\w$, 
but that is equivalent to a bimodule over $\w$ by interchanging the left and right actions.

\begin{rem}
Here we are saying that $\rw$, whose underlying cochain complex is
\begin{equation}
\rw(K, L) = RC^{*}(K, L)
\end{equation}
is a bimodule over $\w^{op}$.
The reader should notice that this is different from the diagonal bimodule $\rw_{\D}$ of $\rw$ itself, 
which has underlying cochain complex 
\begin{equation}
\rw_{\D}(K, L) = \rw(L, K) = RC^{*}(L, K).
\end{equation}
This is one of the main results that we are treating most of the relevant bimodules as over $\w^{op}$ instead of $\w$,
although it indeed brings up some inconvenience at times.
\end{rem}

There is another distinguished $\w^{op}$-bimodule
\begin{equation}
\w_{-},
\end{equation}
whose underlying cochain space is
\begin{equation}
\w_{-}(K, L) = CW^{*}(K, L; -H).
\end{equation}
The bimodule structure maps are defined by counting rigid popsicle maps used to define structure maps for $\rw$ subject to certain constructions on the flavors $\mathbf{p}$ and weights $\mathbf{w}$,
giving an $\ainf$-enhancement of the action of wrapped Floer cohomology on wrapped Floer homology
(see \cite[\S 5.6]{GGV} for details).
This comes with a map of bimodules
\begin{equation}\label{bimodule continuation map}
\ck: \w_{-} \to \w^{op}_{\D}
\end{equation} 
enhancing the continuation map.

\begin{prop}[{\cite[Lemma 5.9]{GGV}}]\label{prop: rw as cone of bimodule}
There is a canonical quasi-isomorphism of $\w^{op}$-bimodules
\begin{equation}
\iota: \cone(\ck: \w_{-} \to \w^{op}_{\D}) \to \rw.
\end{equation}
\end{prop} \qed

In addition to the bimodule continuation map, we also have a bimodule enhancement of the Poncar\'{e} duality isomorphism $I$ \eqref{PD for w-}.

\begin{prop}[{\cite[Proposition 5.11]{GGV}}]\label{prop: bimodule PD}
There is a canonical quasi-isomorphism of $\w^{op}$-bimodules
\begin{equation}\label{bimodule PD}
\mathcal{I}: \w_{-} \stackrel{\sim}\to (\w^{op})^{\vee}[-n],
\end{equation}
whose $(0,0)$-th term on chain complexes satisfies $\mathcal{I}^{0, 0} = I$ where $I$ is the Poncar\'{e} duality isomorphism \eqref{PD for w-}. 
\end{prop} \qed

The construction of the bimodule map $\mathcal{I}$ \eqref{bimodule PD} uses disks with two primary inputs pairing between $\w_{-}$ and $\w^{op}$, 
as well as arbitrarily many auxiliary inputs from $\w^{op}$; see \S 3.6 of \cite{GGV}.

The functor $\Phi$ \eqref{rw to winf} induces a map of $\w^{op}$-bimodules
\begin{equation}\label{rw to winf as bimod}
\Phi_{\w^{op}}: \rw \to \winf,
\end{equation}
which is a quasi-isomorphism when $X$ is non-degenerate with $c_{1}(X) = 0$.
In view of the quasi-isomorphism \eqref{winf as cone}, 
this bimodule map can be written as two components,
\begin{equation}\label{phi+}
\Phi^{+}_{\w^{op}}: \rw \to \r{CC}^{*}(\w^{op}, \hom_{\K}(\w_{\D}, \w_{\D})),
\end{equation}
and
\begin{equation}\label{phi-}
\Phi^{-}_{\w^{op}}: \rw \to \r{CC}^{*}(\w^{op}, (\w^{op})^{\vee} \otimes_{\K} \w_{\D}^{op})[1].
\end{equation}
Here $\Phi^{+}_{\w^{op}}$ \eqref{phi+} is still a bimodule map, 
but $\Phi^{-}_{\w^{op}}$ \eqref{phi-} is not, but rather some kind of homotopy of between bimodule maps.
We consider the following two compositions
\begin{equation}\label{comp phi+}
\w^{op}_{\D} \to \cone(\ck: \w_{-} \to \w^{op}_{\D}) \stackrel{\iota}\to \rw \stackrel{\Phi^{-}_{\w^{op}}}\to \r{CC}^{*}(\w^{op}, \hom_{\K}(\w_{\D}, \w_{\D})),
\end{equation}
\begin{equation}\label{comp phi-}
\w_{-}[1] \to \cone(\ck: \w_{-} \to \w^{op}_{\D}) \stackrel{\iota}\to \rw \stackrel{\Phi^{-}_{\w^{op}}}\to \r{CC}^{*}(\w^{op}, (\w^{op})^{\vee} \otimes_{\K} \w_{\D}^{op})[1].
\end{equation}
The first map $\w_{-}[1] \to \cone(\ck: \w_{-} \to \w^{op}_{\D})$ is not a bimodule map either just as \eqref{phi-} is not, 
but the overall composition \eqref{comp phi-}, with the grading shifted back, defines a bimodule map
\begin{equation}\label{comp phi-2}
\w_{-} \to \r{CC}^{*}(\w^{op}, (\w^{op})^{\vee} \otimes_{\K} \w_{\D}^{op}).
\end{equation}

\begin{lem}[{\cite[Corollary 5.15, Proposition 5.19, Lemma 5.22]{GGV}}]\label{lem: comp phi iso}
The compositions \eqref{comp phi+} and \eqref{comp phi-2} are both quasi-isomorphism of $\w^{op}$-bimodules.
\end{lem} \qed

\begin{lem}\label{lem: continuation to chern}
The quasi-equivalence $\Phi: \rw \to \winf$ induces a quasi-isomorphism of cones of bimodules
\begin{equation}
\cone(\ck: \w_{-} \to \w^{op}_{\D}) \to \cone((\w^{op})^{!} \otimes_{\w^{op}} (\w^{op})^{\vee} \to \w^{op}_{\D}).
\end{equation}
\end{lem}
\begin{proof}
Note that there is a canonical quasi-isomorphism 
\[
\r{CC}^{*}(\w^{op}, (\w^{op})^{\vee} \otimes_{\K} \w_{\D}^{op}) \stackrel{\sim}\to (\w^{op})^{!} \otimes_{\w^{op}} (\w^{op})^{\vee}
\]
given by \eqref{1pt to 2pt}.
Then apply Lemma \ref{lem: comp phi iso} plus a filtration argument.
\end{proof}

\subsection{The geometric pairing}\label{section: rw taut pairing}

Finally, note that there is a tautological pairing on $\rw$ by evaluation,
\begin{equation}\label{taut pairing on rw}
\langle \cdot, \cdot \rangle_{taut}: \rw^{*}(K, L) \otimes \rw^{n-1-*}(L, K) \to \K,
\end{equation}
which has degree $1-n$.
This is a consequence of the general algebraic construction in \S\ref{section: taut pairing} together with the Poincar\'{e} duality isomorphism $I$ \eqref{PD for w-}.
The explicit formula is given as follows.
Since by definition we have
\begin{equation}
\hom_{\rw}(K, L) = RC^{*}(K, L) = \cone(c: CW^{*}(K, L; -H) \to CW^{*}(K, L; H)),
\end{equation}
it follows that we can define a pairing by
\begin{equation}
\langle (x_{-}, x_{+}), (y_{-}, y_{+}) \rangle_{taut} = I(x_{-})(y_{+}) + (-1)^{|x_{+}||y_{+}| + |x_{+}| + |y_{+}|} I(y_{-})(x_{+}).
\end{equation}
This is a form of a tautological pairing defined in \eqref{taut pairing} in \S\ref{section: taut pairing},
and Lemma \ref{lem: taut nondegenerate} implies

\begin{cor}\label{cor: rw taut nondegenerate}
The tautological pairing \eqref{taut pairing on rw} on $\rw$ is nondegenerate.
\end{cor} \qed

The above discussion is sufficient for showing that the cohomology category, 
or the idempotent split-closed derived category of the Rabinowitz Fukaya category, 
is a Calabi-Yau category as an ordinary category.
However, this pairing is too rigid for proving $\ainf$-structures to be compatible with it strictly using direct geometric arguments,
which often fail due to lack of enough symmetries.
Relatedly, the continuation map $\w_{-} \to \w^{op}_{\D}$ \eqref{bimodule continuation map} has too many higher order terms for which compatibility with pairing cannot be guaranteed in the way we choose Floer data.
The chain-level Calabi-Yau structure on the Rabinowitz Fukaya category,
to be constructed in the next section,
 will actually give rise to a slightly different yet closely related pairing.

\section{The Calabi-Yau structure on the Rabinowitz Fukaya category}\label{section: residue on rw}

The goal of this section is to construct a chain-level weak proper Calabi-Yau structure on the Rabinowitz Fukaya category and prove Theorem \ref{thm: rw cy}.
With the algebraic techniques developed previously in \S\ref{section: neighborhood of infinity},
it turns out that this weak proper Calabi-Yau structure is the obvious one that we may expect.

\subsection{The residue on the Rabinowitz Fukaya category}

Suppose $X$ is nondegenerate and has $c_{1}(X) = 0$.
By \cite{ganatra}, the wrapped Fukaya category $\w = \w(X)$ admits a weak smooth Calabi-Yau structure $\sigma \in \r{CC}_{n}(\w)$,
which is a canonical Hochschild cycle whose homology class is the preimage of the unit in symplectic cohomology under the open-closed map \eqref{usual oc}.
(In fact, \cite{ganatra_cyclic} also shows that the open-closed map has an $S^{1}$-equivariant lift,
which leads to a strong smooth Calabi-Yau stricture $\tilde{\sigma}$ lifting $\sigma$,
but we will not use that here.)

By the general algebra machinery developed in \S\ref{section:res}, 
we get a chain map
\begin{equation}\label{winf res}
\r{res}: \r{CC}_{*}(\winf) \to \K[1-n],
\end{equation}
which induces a map on Hochschild homology
\begin{equation}
\r{res}: \r{HH}_{*}(\winf) \to \K[1-n].
\end{equation}
Using this we also define a pairing on the Rabinowitz Fukaya category as follows.
Consider the quasi-equivalence $\Phi: \rw \stackrel{\sim}\to \winf$ given by \eqref{rw to winf2} from Theorem \ref{thm: rw=winf}.
The induced chain map
\begin{equation}
\Phi_{*}: \r{CC}_{*}(\rw) \to \r{CC}_{*}(\winf)
\end{equation}
is a quasi-isomorphism,
where both Hochschild chain complexes are defined with their own diagonal coefficients.

\begin{defn}\label{defn: rw res}
The residue on the Rabinowitz Fukaya category is defined to be
\begin{equation}\label{rw res}
\r{res}_{\rw} = \r{res} \circ \Phi_{*}: \r{CC}_{*}(\rw) \to \K[1-n].
\end{equation}
\end{defn}

\begin{prop}
The map $\r{res}_{\rw}$ \eqref{rw res} is a chain map. 
The induced map only depends on the homology class $[\sigma] \in \r{HH}_{-n}(\w)$ of the weak smooth Calabi-Yau structure $\sigma$ on the wrapped Fukaya category $\w$.
\end{prop}
\begin{proof}
By Proposition \ref{prop: res cy chain map}, $\r{res}$ \eqref{winf res} is a chain map,
so is $\r{res}_{\rw}$ \eqref{rw res} since $\Phi_{*}$ is one.
Independence of chain representatives of $[\sigma]$ follows from the fact that the induced map between Hochschild cohomology and Hochschild homology by the chain map $-\cap \sigma$ \eqref{cy for cc of Z},
which is a quasi-isomorphism by Lemma \ref{lem: capping with cy is iso},
 depends only on the homology class $[\sigma] \in \r{HH}_{-n}(\cc)$.
\end{proof}

\subsection{The induced pairing}

Define a pairing on $\rw$ by the following composition.
\begin{equation}\label{rw res pair}
\langle \cdot, \cdot \rangle_{\r{res}_{\rw}}: \rw^{*}(K, L) \otimes \rw^{n-1-*}(L, K) \stackrel{\mu^{2}_{\rw}}\to \rw^{n-1}(L, L) \stackrel{i_{L}} \to \r{CC}_{n-1}(\rw) \stackrel{\r{res}_{\rw}}\to \K.
\end{equation}
Call this the residue pairing on $\rw$.

\begin{lem}\label{lem: res pairing is taut pairing}
The pairing $\langle \cdot, \cdot \rangle_{\r{res}_{\rw}}$ \eqref{rw res pair} induces a nondegenerate pairing on cohomology groups
\begin{equation}\label{res pairing on cohomology}
\langle \cdot, \cdot \rangle_{\r{res}_{\rw}}: H^{*}(\rw)(K, L) \otimes H^{n-1-*}(\rw)(L, K) \to \K.
\end{equation}
\end{lem}
\begin{proof}
Consider the diagram
\begin{equation}\label{cd for pairing}
\begin{tikzcd}
\rw^{*}(K, L) \otimes \rw^{n-1-*}(L, K) \arrow["\mu^{2}_{\rw}", r] \arrow["\Phi^{1} \otimes \Phi^{1}", d] & \rw^{n-1}(L, L) \arrow["\Phi^{1}", d] \\
\winf^{*}(K, L) \otimes \winf^{n-1-*}(L, K) \arrow["\mu^{2}_{\winf}", r] & \winf^{n-1}(L, L) \arrow["i_{L}", r] &\r{CC}_{n-1}(\winf) \stackrel{\r{res}}\to \K.
\end{tikzcd}
\end{equation}
Since $\Phi: \rw \stackrel{\sim}\to \winf$ \eqref{rw to winf2} is an $\ainf$-functor (which is a quasi-equivalence), 
the left square in \eqref{cd for pairing} commutes up to chain homotopy, 
where the chain homotopy is precisely given by $\Phi^{2}$.
It follows that the induced pairing \eqref{res pairing on cohomology} agrees with the pairing on $H^{*}(\winf)$ induced by the residue on $\winf$, 
which is nondegenerate by Proposition \ref{prop: res pairing nondegenerate}.
\end{proof}

\begin{rem}
It is a straightforward but somewhat tedious computation to show that the pairing \eqref{rw res pair} is compatible with the tautological pairing \eqref{taut pairing on rw} in an $\ainf$-homotopic sense.
This would require one to go through details of the construction of the functor $\Phi$ in \cite{GGV}. 
Since we do not need this result in this paper, we will not carry out this discussion.
\end{rem}

\begin{proof}[Proof of Theorem \ref{thm: rw cy}]
We already have a chain map $\r{res}_{\rw}: \r{CC}_{*}(\rw) \to \K[1-n]$ \eqref{rw res}.
By Lemma \ref{lem: res pairing is taut pairing}, its induced pairing $\langle \cdot, \cdot \rangle_{\r{res}_{\rw}}$ is nondegenerate on the cohomology category $H(\rw)$.
\end{proof}

\section{Closed-string counterparts}\label{section: closed-string rfc}

In this section, we recall the definition of symplectic cohomology and provide a definition of the closed-string Rabinowitz Floer cohomology in a manner that is compatible with the geometric/analytic framework we used for constructing the Rabinowitz Fukaya category in \S\ref{section: rabinowitz fukaya category}.

\subsection{Popsicle structures for closed strings}\label{section: closed string popsicle}

Consider genus zero Riemann surfaces without boundary with two punctures.
Each of such surfaces is biholomorphic to $\P^{1}$ minus two points,
and topologically they can all be identified with a cylinder.
However, we want to think of them as carrying additional data similar to popsicle structures when we consider pseudoholomorphic maps from them.

Let $\mathbf{p}_{c}: F \to \{1\}$ be a map from a finite set $F$ to the singleton set $\{1\}$.
Technically all finite sets are allowed;
however for the purpose of extracting Floer-theoretic operations we will always consider injective maps, 
so that $F$ is either the empty set $\varnothing$ or a singleton set.
Let $w_{0}, w_{1} \in \Z_{\le 0}$ be non-positive integers, called weights, and put $\mathbf{w}_{c} = (w_{0}, w_{1})$.
These should satisfy the following condition
\begin{equation}\label{weight condition 3}
w_{0} = w_{1} + |F|.
\end{equation}
We shall only be interested in the case where $w_{0}, w_{1} \in \{-1, 0\}$,
which together with \eqref{weight condition 3} implies that the are only three possibilities
\begin{enumerate}[label=(\roman*)]

\item $w_{0} = w_{1} = 0$, and $F = \varnothing$;

\item $w_{0} = w_{1} = -1$, and $F = \varnothing$;

\item $w_{0} = 0, w_{1} = -1$, and $F$ is a singleton set so that $\mathbf{p}_{c}$ is a bijection.

\end{enumerate}

Let $S$ be a genus zero Riemann surface without boundary with two punctures $\zeta_{0}, \zeta_{1}$,
where for $i=0, 1$ we assign the weight $w_{i}$ to $\zeta_{i}$.
These determine the positive/negative choices cylindrical ends at $\zeta_{0}, \zeta_{1}$ as follows:
\begin{equation}
\k_{i}^{\sk_{i}}: C^{\sk_{i}} \to S, i = 0, 1,
\end{equation}
where $C^{+} = [0, +\infty) \times \R/2\pi\Z$ is the positive cylinder, $C^{-} = (-\infty, +] \times \R/2\pi\Z$ is the negative cylinder,
and the symbols $\sk_{i} \in \{+, -\}$ are defined as in \eqref{sign symbol}.

A choice of a cylindrical end $\k$ at a puncture $\zeta$ determines a tangent ray at $\zeta$,
by taking the limit as $s \to \pm \infty$ of the tangent rays containing the tangent vectors $\k_{*}\p_{s}$ along the line $\k(s, 0)$.
In other words, one gets a framing at the puncture $\zeta$.
Since we will not be dealing with additional algebraic structures governed by the moduli spaces of framed genus zero Riemann surfaces,
there are no issues in, and we will be, fixing the choices of cylindrical ends (and therefore framings) such that the geodesics on $\P^{1} = \bar{S}$ starting from $\zeta_{i}$ in the direction of the tangent rays determined by the cylindrical ends are contained in the same geometric image.
For $w_{0} = w_{1}$ and so $F=\varnothing$, this can simply be replaced with the requirement that the two cylindrical ends agree with the global trivialization of the cylinder $S \cong (-\infty, +\infty) \times \R/2\pi\Z$.
With such choices of cylindrical ends,
the tangent rays at both punctures $\zeta_{0}, \zeta_{1}$ 
In the last case where $w_{0} = 0, w_{1} = -1$,

\begin{defn}
A popsicle structure on $S$ of flavor $\mathbf{p}$ is a choice of a preferred point along the geodesic determined by the choices of cylindrical ends as above.
\end{defn}

Although the above definition is meaningful and useful,
there are not enough stable popsicles to form a `moduli space'.
The only stable case is when $|F|=1$, in which the moduli space $\M^{2, \{1\}, (0, -1)}$ is a singleton set.
Despite the insufficient supply of elements in the moduli spaces of domains,
we find that using popsicles make it easy for us to keep track of the choices of cylindrical ends and to study limits of pseudoholomorphic maps from such domains with different choices of cylindrical ends.

\begin{rem}
For genus zero Riemann surfaces with multiple inputs and outputs,
the notion of a popsicle structure is problematic, 
so we will not use this model to discuss higher structures in Rabinowitz Floer theory.
\end{rem}

\subsection{Pseudoholomorphic maps from twice-punctured spheres}

Let $S$ be a genus zero Riemann surface without boundary with two punctures,
with a popsicle structure $\eta$ of flavor $\mathbf{p}_{c}$ and weights $\mathbf{w}_{c}$.

\begin{defn}\label{def: Floer data for closed strings}
A Floer datum for $(S, \eta)$ of flavor $\mathbf{p}_{c}$ and weights $\mathbf{w}_{c}$ consists of
\begin{enumerate}[label=(\roman*)]

\item Cylindrical ends $\k_{i}: C^{\sk_{i}} \to S, i=0, 1$.

\item A domain-dependent family of Hamiltonians $H_{S}: S \to \mathcal{H}(X)$, 
which is compatible with cylindrical ends: $\k_{i}^{*} H_{S} = H$.

\item A domain-dependent family of $\omega$-compatible almost complex structures $J_{S}: S \to \mathcal{J}_{1}(X)$ (see \eqref{h rescaled contact type}),
which is compatible with cylindrical ends: $\k_{i}^{*} J_{S} = J_{t}$.

\end{enumerate}
In the unstable case where $w_{0}=w_{1}$ (so that the popsicle structure is trivial), 
$H_{S}$ and $J_{S}$ must be domain independent.
\end{defn}

In Floer theory, we would like to extend the choices of Floer data across the moduli spaces.
Since the only non-empty moduli space is $\M^{2, \{1\}, (0, -1)}$ is a point,
and there are no rescaling factors as required by Definition \ref{def: Floer data for closed strings} as opposed to Definition \ref{def: Floer data},
so a {\it universal and conformally consistent choice of Floer data} $\mathbf{D}_{\M}$ for all closed strings,
which could have been defined in a way similar to Definition \ref{def: universal and conformally consistent choice of Floer data},
is indeed the same as a choice of a Floer datum on $(S, \eta)$ identified with the fiber of the universal curve over $\M^{2, \{1\}, (0, -1)}$,
along with choices of Floer data for trivial cylinders subject to the requirement by Definition \ref{def: Floer data for closed strings}.

To define the closed-string invariants, we must in addition break $S^{1}$-symmetry for non-trivial Hamiltonian orbits,
which can be done by introducing a perturbation term as below.

\begin{defn}\label{def: perturbation datum on cylinder}
A perturbation datum for $S$ is a smooth non-negative function
\begin{equation}
F: S^{1} \times X \to \R
\end{equation}
such that $F$ and $\lambda(X_{F})$ are uniformly absolutely bounded, 
and all time-one $H_{F}$-orbits are nondegenerate, 
where $H_{F}(t, p) = H(p) + F(t, p)$ is called the total Hamiltonian.
\end{defn}

Let $\o(H_{F})$ and $\o(-H_{F})$ be the sets of time-one periodic orbits of $H_{F}$ and $-H_{F}$ respectively.
As is the case with chords, there is a natural one-to-one correspondence between $\o(H_{F})$ and $\o(-H_{F})$,
given by negative reparametrization.
Consider maps $u: S \to X$ satisfying
\begin{equation}\label{CR for cylinders}
\begin{cases}
&(du - dt \otimes X_{H_{F}})^{0, 1} = 0, \\
&\lim\limits_{s \to \sk_{i}\infty} u \circ \k_{i}(s, \cdot) = y_{i}(\cdot) \text{ exponentially for some } y_{i} \in \o(H). 
\end{cases}
\end{equation}
When $w_{0}=0, w_{1}=-1$, both the cylindrical ends $\k_{0}, \k_{1}$ are negative cylindrical ends.
Since the cylindrical end $\k_{1}$ is negative as opposed to its usual choice being positive,
we want to think of the orbit $y_{1} \in \o(H_{F})$ as being the negatively parametrized orbit for a unique orbit $y_{1}^{-} \in \o(-H_{F})$.

Denote by 
\begin{equation}
\M^{2, \mathbf{p}_{c}, \mathbf{w}_{c}}(y_{0}, y_{1})
\end{equation}
the moduli space of maps satisfying \eqref{CR for cylinders}.
It has a natural Gromov compactification by adding broken cylinders,
where the popsicle structure can go to any smooth component:
\begin{equation}
\bar{\M}^{2, \mathbf{p}_{c}, \mathbf{w}_{c}}(y_{0}, y_{1}) = \coprod_{\substack{m \\ y'_{1}, \ldots, y'_{m}}} \M^{2}(y_{0}, y'_{1}) \times \cdots \times \M^{2, \mathbf{p}_{c}, \mathbf{w}_{c}}(y'_{i-1}, y'_{i}) \times \cdots \times \M^{2}(y'_{m}, y_{1}).
\end{equation}
For generic choices of Floer data, this moduli space is a compact smooth manifold with boundary and corners.

\subsection{Symplectic cohomology and Rabinowitz Floer cohomology}

The {\it symplectic cochain complex} is defined to be
\begin{equation}\label{symplectic cochains}
SC^{*}(X; H_{F}) = \bigoplus_{y \in \o(H)} |o_{y}|_{\K},
\end{equation}
where the grading is given by
\begin{equation}
|y| = n - CZ(y),
\end{equation}
where $CZ(y)$ stands for the Conley-Zehnder index of $y$.
The differential $d: SC^{*}(X; H_{F})  \to SC^{*}(X; H_{F}) [1]$ is defined by
\begin{equation}
d([y_{1}]) = \sum_{\substack{y_{0}\\ |y_{0}| = |y_{1}|+1}} \sum_{u \in \M^{2, \varnothing, (0, 0)}(y_{0}, y_{1})} o_{u}([y_{1}]).
\end{equation}

The {\it symplectic chain complex} is defined to be
\begin{equation}\label{symplectic chains}
SC^{*}(X; -H_{F})  = \prod_{y \in \o(-H)} |o_{y}|_{\K}.
\end{equation}
The differential $\p: SC^{*}(X; -H_{F}) \to SC^{*}(X; -H_{F})[1]$ is defined by
\begin{equation}
\p ([y_{1}^{-}]) = \prod_{\substack{y_{0} \\ |y_{1}| = |y_{0}|+1}} \sum_{u \in \M^{2, \varnothing, (-1, -1)}(y_{0}, y_{1})} o_{u}([y_{1}^{-}]).
\end{equation} 
Note that, despite its name, $SC^{*}(X; -H_{F})$ is carrying a cohomological grading and the differential $\p$ increases grading.

Similar to \eqref{PD for w-}, there is a Poincar\'{e} duality chain-level isomorphism
\begin{equation}\label{PD for sc-}
I: SC^{*}(X; -H_{F})  \stackrel{\cong}\to SC^{*}(X; H_{F}) ^{\vee}[-2n] = \hom_{\K}(SC^{2n-*}(X; H_{F}), \K),
\end{equation}
by negative reparametrization of orbits.

The above two closed-string invariants also come with a continuation map
\begin{equation}\label{closed-string continuation map}
c: SC^{*}(X; -H_{F})  \to SC^{*}(X; H_{F}) ,
\end{equation}
which is defined by counting rigid elements in the moduli space $\M^{2, \{1\}, (0, -1)}(y_{0}, y_{1})$, i.e.,
\begin{equation}\label{closed-string continuation map formula}
c([y_{1}^{-}]) = \sum_{\substack{y_{0}\\|y_{0}| = 2n - |y_{1}|}} \sum_{u \in \M^{2, \{1\}, (0, -1)(y_{0}, y_{1})}} \cR^{2, \{1\}, (0, -1)}_{u}([y_{1}^{-}]).
\end{equation}

We define the {\it closed-string Rabinowitz Floer cochain complex} to be the mapping cone
\begin{equation}\label{closed-string rfc}
RFC^{*}(X) = \cone(c: SC^{*}(X; -H_{F}) \to SC^{*}(X; H_{F}) ).
\end{equation}

Similar to the tautological pairing \eqref{taut pairing on rw}, 
we get from the tautological pairing \eqref{taut pairing} defined in \S\ref{section: taut pairing} a tautological pairing
\begin{equation}\label{taut pairing on rfc}
\langle \cdot, \cdot \rangle_{taut}: RFC^{*}(X) \otimes RFC^{2n-1-*}(X) \to \K,
\end{equation}
using the isomorphism \eqref{PD for sc-}.
Just like the pairing on $\rw$ (Corollary \ref{cor: rw taut nondegenerate}), this pairing is nondegenerate as well.

\section{Open-closed relationships}\label{section: open-closed}

In this section, we study relationships between the Rabinowitz Fukaya category and the closed-string Rabinowitz Floer cohomology.
This relies on a geometric construction of a map from the Hochschild homology of the wrapped Fukaya category with coefficients in the Rabinowitz Fukaya category.

\subsection{Open-closed strings}\label{section: popsicles on open-closed strings}

We first extend the definition of popsicles to domains of pseudoholomorphic maps that are commonly used to define the open-closed map.
Let $k \ge 1$ be a positive integer.
Let $S$ be a bordered Riemann surface isomorphic to a disk with $k$ boundary punctures $z_{1}, \ldots, z_{k}$ ordered cyclically counterclockwise along the boundary, treated as inputs,
as well as one interior puncture $\zeta$ equipped with a framing, treated as an output.
We want to think of $z_{k}$ as a special input, 
and $z_{1}, \ldots, z_{k-1}$ auxiliary inputs.

Such a surface will be a typical domain of a pseudoholomorphic map defining the open-closed map,
which we call an {\it open-closed string}.
Since we will not define an $S^{1}$-equivariant version of the open-closed map which would otherwise require one to consider all possible framings at the puncture $\zeta$,
we specify the choice of the framing at $\zeta$ to be pointing to the direction of the special input $z_{k}$.

Define a popsicle structure on an open-closed string in a way similar to that discussed in \S\ref{section: closed string popsicle}.
Let $\mathbf{p}_{oc, k}: F \to \{d\}$ be a map from a finite set $F$.
Let $w, w_{k} \in \Z_{\le 0}$ be a pair of non-positive integers satisfying the following conditions
\begin{equation}\label{oc weight 1}
w = w_{k} + |F|.
\end{equation}
We shall be only interested in the case where $w, w_{k} \in \{-1, 0\}$,
such that by \eqref{oc weight 1} there are only three possibilities:
\begin{enumerate}[label=(\roman*)]

\item $w=w_{k}=0$, $F = \varnothing$.

\item $w=w_{k}=-1$, $F=\varnothing$.

\item $w=0, w_{k}=-1$, and $F$ is a singleton set so that $\mathbf{p}_{oc, k}$ is a bijection.

\end{enumerate}
Define the extended weights
\begin{equation}\label{extended weight}
\mathbf{w}_{oc, k} = (w, \underbrace{0, \ldots, 0}_{k-1 \text{ times }}, w_{k}),
\end{equation} 
which means that we assign the weight $w$ to $\zeta$, $w_{k}$ to $z_{k}$, and $0$ to $z_{i}, i=1, \ldots, k$.

\begin{defn}\label{def: popsicle on open-closed string}
A popsicle structure of flavor $\mathbf{p}_{oc, k}$ on an open-closed string $S$ is a choice of a preferred point along the unqiue geodesic connecting $\zeta$ and $z_{k}$ determined by the chosen framing.
\end{defn}

Let 
\begin{equation}
\M_{\oc}^{k, \mathbf{p}_{oc, k}, \mathbf{w}_{oc, k}}
\end{equation}
be the moduli space of stable open-closed strings equipped with popsicle structures of flavor $\mathbf{p}_{oc, k}$ and weights $\mathbf{w}_{oc, k}$.
When $\mathbf{p}_{oc, k} = \varnothing$ and $\mathbf{w}_{oc, k} = \mathbf{0} = (0, \ldots, 0)$,
so we get the usual moduli spaces of open-closed strings, i.e., the ones without popsicle structures.
The other case where the popsicle structure is trivial is when $\mathbf{w}_{oc, k} = (-1, 0, \ldots, 0, -1)$.
In this case, we get a different moduli space, 
but abstractly it can be identified with a copy of $\M_{\oc}^{k, \mathbf{p}_{oc, k}, \mathbf{w}_{oc, k}}$ if we forget the weights (note that we have not made choices of strip-like and cylindrical ends yet).

The moduli space $\M_{\oc}^{k, \mathbf{p}_{oc, k}, \mathbf{w}_{oc, k}}$ has a natural stable map compactification
\begin{equation}
\bar{\M}_{\oc}^{k, \mathbf{p}_{oc, k}, \mathbf{w}_{oc, k}}
\end{equation}
by adding boundary strata which consist of {\it stable broken open-closed strings} with a popsicle structure,
appearing as limits of open-closed strings with popsicle structures.
Each stable broken open-closed string with a popsicle structure has a {\it main component}, 
which is itself an open-closed string with a popsicle structure.
The other components are either closed strings with popsicle structures introduced in \S\ref{section: closed string popsicle},
or the usual disks with popsicle structures introduced in \S\ref{section: popsicles}.
In particular, the codimension one boundary strata,
which consist of stable broken open-closed strings with popsicles structures with two smooth surface components,
 are covered the following types of moduli spaces
\begin{align}
& \M_{\oc}^{k, \mathbf{p}'_{oc, k}, \mathbf{w}'_{oc, k}} \times \M^{2, \mathbf{p}_{c}, \mathbf{w}_{c}}, \label{oc domain boundary 1} \\ 
& \cR^{j+1} {}_{0}\times_{i+1} \M_{\oc}^{k-j, \mathbf{p}'_{oc, k-j}, \mathbf{w}'_{oc, k-j}}, 1 \le i < k, i+j<k, \label{oc domain boundary 2}\\ 
& \cR^{j+1, \mathbf{p}'_{k, i, j}, \mathbf{w}'_{k, i, j}} {}_{0}\times_{i+1} \M_{\oc}^{k-j, \mathbf{p}''_{oc, k-j}, \mathbf{w}''_{oc, k-j}}, 1 \le i < k, 1 \le j < k, i+ > k. \label{oc domain boundary 3}
\end{align}

Some explanations about the boundary strata \eqref{oc domain boundary 1}-\eqref{oc domain boundary 3},
including the various notations for flavors and weights, are in order:
\begin{enumerate}[label=(\roman*)]

\item The boundary strata of type \eqref{oc domain boundary 1} consist of stable broken open-closed strings with a main component and the other component being a closed-string with popsicle structure introduced in \S\ref{section: closed string popsicle},
such that the closed-string with popsicle structure itself is stable, 
which implies that $\mathbf{p}_{c} = (0, -1)$ and $\mathbf{p}_{c}: F_{c} \to \{1\}$ is a bijection.
Consequently the flavor $\mathbf{p}'_{oc, k}$ is determined by $\mathbf{p}_{oc, k}$ as follows.
Since $\mathbf{p}_{oc, k}$ is at most a singleton set, 
it follows that $\mathbf{p}'_{oc, k} = \varnothing$.
The weights $\mathbf{w}'_{oc, k}$ and $\mathbf{w}_{c}$ are automatically inherited from $\mathbf{w}_{oc, k}$,
such that both \eqref{weight condition 3} and \eqref{oc weight 1} are satisfied. 

\item The boundary strata of type \eqref{oc domain boundary 2} consist of stable broken open-closed strings with a main component and another disk component,
such that the special input $z_{0}$ remains on the main component.
The notation ${}_{0} \times_{i}$ in \eqref{oc domain boundary 2} means that the $0$-th boundary puncture of the disk component is identified with the $i$-th boundary puncture of the main component.
Since popsicle structures on open-closed strings are choices of points on the geodesic connecting $\zeta$ and the special input $z_{k}$,
it follows that the other disk component must carry a trivial popsicle structure.
In addition, all the weights at the boundary punctures at the disk component must be zero,
which imply that that disk is an ordinary boundary-punctured disk.
The flavor and weight for the popsicle structure on the main component do not essentially change (except for having fewer boundary punctures).
That is, $\mathbf{p}'_{oc, k-j}: F \to \{k-j\}$ is the same map as $\mathbf{p}_{oc, k}$ where we identify the singleton sets $\{k\}$ with $\{k-j\}$,
and
\begin{equation}
\mathbf{w}'_{oc, k-j} = (w, \underbrace{0, \ldots, 0}_{k-j-1 \text{times }}, w_{k}),
\end{equation}
where $w, w_{k}$ remain unchanged when inherited from $\mathbf{w}_{oc, k}$.

\item The boundary strata of type \eqref{oc domain boundary 3} are similar to \eqref{oc domain boundary 2},
except that the special input $z_{k}$ goes to the disk component other than the main component.
Since the original popsicle structure has at most one sprinkle, i.e., $|F| \le 1$,
in this case exactly one component of a stable broken open-closed string can have a non-trivial popsicle structure.
The weights $\mathbf{w}'_{k, i, j}$ on the disk component are inherited from the weights from the original smooth open-closed string, 
which takes the form
\begin{equation}\label{weight on bubble off disk}
\mathbf{w}'_{k, i, j} = (w'_{0}, 0, \ldots, 0, w_{k}, 0, \ldots, 0),
\end{equation}
where the weight $w_{k}$ (the one coming from the original smooth open-closed string) is placed at the $(k-i)$-th entry in \eqref{weight on bubble off disk}, 
and $w'_{0}$ is the new weight at the $0$-th boundary puncture of the disk component,
which is glued to the main component at its $(k-j+1)$-th special input.
If the disk component carries the non-trivial popsicle structure,
then we must have $w'_{0}=0$ and $w_{k} = -1$;
otherwise the weights must satisfy $w'_{0}=w_{k}=0$ or $w'_{0}=w_{k}=-1$.

\end{enumerate}

\subsection{Pseudoholomorphic maps from open-closed strings}

The definition of a pseudoholomorphic map from an open-closed string with a popsicle structure is similar to those from popsicles on disks as well as cylinders, 
for which we need to introduce Floer data.
Fix $H \in \mathcal{H}(X), J \in \mathcal{J}_{1}(X)$ as well as a perturbation datum $F \in C^{\infty}(X \times \R/2\pi\Z, \R)$ introduced in Definition \ref{def: perturbation datum on cylinder}.

\begin{defn}\label{def: Floer data for open-closed string}
A Floer datum on a open-closed string with a popsicle structure $(S, \eta)$ of flavor $\mathbf{p}_{oc, k}$ and weights $\mathbf{w}_{oc, k}$ consists of the following data.
\begin{enumerate}[label=(\roman*)]

\item Choices of cylindrical and strip-like ends at the interior and boundary punctures
\begin{align}
\k^{\sk}: & C^{\sk} \to S, \\
\e^{+}_{i}: & Z^{+} \to S, i = 1, \ldots, k,\\
\e^{\sk_{0}}_{0}: & Z^{\sk_{0}} \to S,
\end{align}
such that the framing determined by the cylindrical end $\k^{\sk}$ should be pointing to the puncture $z_{0}$.

\item A collection of real numbers $\nu, \nu_{0}, \ldots, \nu_{k} \ge 1$, 
called {\it rescaling factors},
which satisfy
\begin{equation}
(-1)^{\d} \nu + (-1)^{\d_{0}} \nu_{0} + \sum_{i=1}^{k} \nu_{i} \le 0.
\end{equation}

\item A function $\rho_{S}: S \to [1, +\infty)$ which is locally constant over cylindrical and strip-like ends,
equal to $\nu$ over the cylindrical end $\k$, $\nu_{i}$ over the the strip-like ends $\e_{i}, i = 1, \ldots, k$,
and $\nu_{k}$ over $\e_{k}$.

\item A sub-closed one-form $\alpha_{S}$ with $d\alpha_{S} \le 0$ which is which is compatible with cylindrical end strip-like ends:
\begin{align}
\k^{*} \alpha_{S} & = \nu dt, \\
\e_{i}^{*} \alpha_{S} & = \nu_{i}dt, i =1,\ldots, k,\\
\e_{0}^{*} \alpha_{S} & = \nu_{0} dt.
\end{align}

\item A domain-dependent family of Hamiltonians $H_{S}: S \to \mathcal{H}(X)$ which is compatible with cylindrical and strip-like ends

\item A domain-dependent family of perturbation data $F_{S}: S \to C^{\infty}(X \times \R/2\pi\Z, \R)$ which is locally constant outside of the image of the cylindrical end $\k$,
and is compatible with the cylindrical end 
\begin{equation}
\k^{*} F_{S} = \frac{F \circ \psi^{\nu}}{\nu^{2}} + C,
\end{equation}
for some constant $C$ that depends only on the cylindrical end $\k$.

\item A domain-dependent family of almost complex structures $J_{S}: S \to \mathcal{J}(X)$ which is compatible with cylindrical and strip-like ends

\end{enumerate}

\end{defn}

There is a notion of a universal and conformally consistent choice of Floer data, similar to Definition \ref{def: universal and conformally consistent choice of Floer data}.

\begin{defn}\label{def: universal and conformally consistent choice of Floer data for oc}
Suppose we have made a universal and conformally consistent choice of Floer data $\mathbf{D}_{\cR}$ for all disks with popsicle structures,
as well as a universal and conformally consistent choice of Floer data $\mathbf{D}_{\M}$ for all closed strings with popsicle structures.
A universal and conformally consistent choice of Floer data $\mathbf{D}_{\oc}$ for open-closed strings with popsicle structures is a choice of Floer data,
one for each representative of element in $\bar{\M}_{\oc}^{k, \mathbf{p}_{oc, k}, \mathbf{w}_{oc, k}}$, 
being $\r{Aut}(\mathbf{p}_{oc, k})$-invariant and smoothly varying over $\M_{\oc}^{k, \mathbf{p}_{oc, k}, \mathbf{w}_{oc, k}}$, 
which at boundary strata covered by \eqref{oc domain boundary 1}-\eqref{oc domain boundary 3} agree with a product Floer data chosen for lower dimensional moduli spaces up to conformal equivalence,
parts of which come from $\mathbf{D}_{\cR}$ and $\mathbf{D}_{\M}$.
\end{defn}

Again, since all the choices involved in a Floer datum as in Definition \ref{def: Floer data for open-closed string} form a contractible space, 
a universal and conformally consistent choice of Floer data $\mathbf{D}_{\M}$ for all closed strings with popsicle structures exists.

Let $y \in \o(H_{F})$ and $x_{i} \in \chi(L_{i-1}, L_{i}; H)$ for $i = 1, \ldots, k$ and $x_{0} \in \chi(L_{k}, L_{0}; H)$.
The inhomogeneous Cauchy-Riemann equation for maps $u: S \to X$ defined with respect to a choice of a Lagrangian label $L_{0}, \ldots, L_{k}$, 
a choice of a Floer datum as in Definition \ref{def: Floer data for open-closed string},
as well as asymptotic conditions given by $y, x_{0}, x_{1}, \ldots, x_{k}$,
is the following
\begin{equation}\label{cauchy-riemann for oc}
\begin{cases}
& (du-X_{H_{S}} \otimes \alpha_{S})^{0, 1} = \frac{1}{2} [ (du-X_{H_{S}} \otimes \alpha_{S}) + J_{S} \circ (du-X_{H_{S}} \otimes \alpha_{S}) \circ j] = 0, \\
& u(z) \in (\psi^{\rho_{S}(z)})^{*} L_{i}, z \in \p S, i = 0, \ldots, k, \\
& \lim\limits_{s \to \sk \infty} u \circ \k(s, \cdot) = (\psi^{\nu})^{*} y \text{ exponentially},  \\
& \lim\limits_{s \to +\infty} u \circ \e_{i}(s, \cdot) = (\psi^{\nu_{i}})^{*} x_{i} \text{ exponentially}, i=1,\ldots, k-1, \\
& \lim\limits_{s \to \sk_{k} \infty} u \circ \e_{k}(s, \cdot) = (\psi^{\nu_{0}})^{*} x_{k} \text{ exponentially}.
\end{cases}
\end{equation}

Suppose we have made a universal and conformally consistent choice of Floer data $\mathbf{D}_{\M}$ for all closed strings with popsicle structures.
Here we denote by $\mathbf{x} = (x_{1}, \ldots, x_{k})$.
Let
\begin{equation}\label{moduli space for oc}
\M_{\oc}^{k, \mathbf{p}_{oc, k}, \mathbf{w}_{oc, k}}(y; \mathbf{x})
\end{equation}
be the moduli space of pseudoholomorphic maps from open-closed strings with popsicle structures of flavor $\mathbf{p}_{oc, k}$ and weights $\mathbf{w}_{oc, k}$, 
i.e., the space of equivalence classes of maps $u: S \to X$ satisfying the equation \eqref{cauchy-riemann for oc}.

\begin{lem}\label{lem: dim of oc moduli space}
The virtual dimension of the moduli space $\M_{\oc}^{k, \mathbf{p}_{oc, k}, \mathbf{w}_{oc, k}}(y; \mathbf{x})$ \eqref{moduli space for oc} is
\begin{equation}\label{dim of oc moduli space}
v-\dim \M_{\oc}^{k, \mathbf{p}_{oc, k}, \mathbf{w}_{oc, k}}(y; \mathbf{x}) = k + |F| + n(1+2\d+\sum_{i=1}^{k}\d_{i}) - (-1)^{\d}|y| - \sum_{i=1}^{k} (-1)^{\d_{i}}|x_{i}|,
\end{equation}
where the symbols $\d, \d_{i}$ are defined in \eqref{sign function}.
\end{lem}
\begin{proof}
This is a standard index calculation, e.g. following the argument of Proposition 11.13 of \cite{seidel_book}.
The only difference is there is an interior puncture $\zeta$ with asymptotic condition $y$, 
but we can follow the same idea of the argument by gluing a Cauchy-Riemann operator on a sphere with one puncture with the  asymptotic condition given by the negatively parametrized orbit $\bar{y}$.
Then \eqref{dim of oc moduli space} follows.
\end{proof}

The moduli space $\M_{\oc}^{k, \mathbf{p}_{oc, k}, \mathbf{w}_{oc, k}}(y; \mathbf{x})$ \eqref{moduli space for oc} has a natural Gromov compactification
\begin{equation}\label{compactified moduli space for oc}
\bar{\M}_{\oc}^{k, \mathbf{p}_{oc, k}, \mathbf{w}_{oc, k}}(y; \mathbf{x})
\end{equation}
by adding stable maps from broken open-closed strings with popsicle structures.
The underlying domain of a representative of an element in $\bar{\M}_{\oc}^{k, \mathbf{p}_{oc, k}, \mathbf{w}_{oc, k}}(y; \mathbf{x})$ \eqref{compactified moduli space for oc} is not necessarily stable;
and in addition to the stable maps coming from domain degenerations,
there are also stable maps coming from strip and cylinder breakings.
The codimension-one boundary strata of $\bar{\M}_{\oc}^{k, \mathbf{p}_{oc, k}, \mathbf{w}_{oc, k}}(y; \mathbf{x})$ are therefore covered by the following types of product moduli spaces:
\begin{align}
& \M_{\oc}^{k, \mathbf{p}'_{oc, k}, \mathbf{w}'_{oc, k}}(y'; \mathbf{x}) \times \M^{2, \mathbf{p}_{c}, \mathbf{w}_{c}}(y', y), \label{oc boundary 1} \\ 
& \cR^{j+1}(x', x_{i+1}, \ldots, x_{i+j}) {}_{0}\times_{i+1} \M_{\oc}^{k-j, \mathbf{p}'_{oc, k-j}, \mathbf{w}'_{oc, k-j}}(y; x_{1}, \ldots, x_{i}, x', x_{i+j+1}, \ldots, x_{k})  \label{oc boundary 2}\\ 
& \cR^{j+1, \mathbf{p}'_{k, i, j}, \mathbf{w}'_{k, i, j}}(x', x_{i+1}, \dots, x_{k}, x_{1}, \ldots, x_{i+j-k}) {}_{0}\times_{i+1} \M_{\oc}^{k-j, \mathbf{p}''_{oc, k-j}, \mathbf{w}''_{oc, k-j}}(y; x_{i+j-k+1}, \ldots, x_{i}, x'), \label{oc boundary 3}\\
& \M_{\oc}^{k, \mathbf{p}_{oc, k}, \mathbf{w}_{oc, k}}(y'; \mathbf{x}) \times \M^{2, \varnothing, (w, w)}(y', y) \label{oc boundary 4} \\
& \cR^{2, \varnothing, (0, 0)}(x'_{i}, x_{i})  {}_{0}\times_{i+1} \M_{\oc}^{k, \mathbf{p}_{oc, k}, \mathbf{w}_{oc, k}}(y; x_{0}, \ldots, x'_{i}, \ldots, x_{k})  \label{oc boundary 5}\\
& \cR^{2, \varnothing, (w_{k}, w_{k})}(x'_{k}, x_{k})  {}_{0}\times_{i+1} \M_{\oc}^{k, \mathbf{p}_{oc, k}, \mathbf{w}_{oc, k}}(y; x_{1}, \ldots, x'_{k}) \label{oc boundary 6}\\
\end{align}

In comparison to \eqref{oc domain boundary 1}-\eqref{oc domain boundary 3},
we give some addition explanations of the notations here:
\begin{enumerate}[label=(\roman*)]

\item The boundary strata \eqref{oc boundary 1} occur because of the domain degeneration corresponding to the boundary strata \eqref{oc domain boundary 1} in the compactified moduli space of domains.
Here $y'$ is a new orbit as the asymptotic condition at the interior puncture for a map in $\M_{\oc}^{k, \mathbf{p}_{oc, k}, \mathbf{w}_{oc, k}}(y'; \mathbf{x})$,
and the asymptotic condition at $\zeta_{2, 1}$ for a map in $\M^{2, \varnothing, (w, w)}(y', y)$.

\item The boundary strata \eqref{oc boundary 2} occur because of the domain degeneration corresponding to the boundary strata \eqref{oc domain boundary 2} in the compactified moduli space of domains.
Here $x'$ is a new chord as the asymptotic condition at the $0$-th boundary puncture (output) for a map in $\cR^{j+1}(x', x_{i+1}, \ldots, x_{i+j})$,
and the asymptotic condition at the $(i+1)$-th boundary puncture for a map in $\M_{\oc}^{k-j, \mathbf{p}'_{oc, k-j}, \mathbf{w}'_{oc, k-j}}$.

\item The boundary strata \eqref{oc boundary 3} occur because of the domain degeneration corresponding to the boundary strata \eqref{oc domain boundary 3} in the compactified moduli space of domains.
The role of the new chord $x'$ is similar to the previous case.

\item The boundary strata \eqref{oc boundary 4} occur because of cylinder breakings at the puncture $\zeta$,
where $y'$ is a new orbit as the output for a map in $\M_{\oc}^{k, \mathbf{p}_{oc, k}, \mathbf{w}_{oc, k}}(y'; \mathbf{x})$,
and the asymptotic condition at $\zeta_{2, 1}$ for a map in $\M^{2, \varnothing, (w, w)}(y', y)$.

\item The boundary strata \eqref{oc boundary 5} and \eqref{oc boundary 6} occur because of strip breakings at the puncture $z_{i}$ and $z_{0}$,
where $x'_{i}$ is a new chord as the asymptotic condition at $z_{1, 1}$ for a map in $\cR^{2, \varnothing, (0, 0)}(x'_{i}, x_{i})$ or $\cR^{2, \varnothing, (w_{0}, w_{0})}(x'_{0}, x_{0})$, 
and the asymptotic condition at $z_{2, i}$ for a map $\M_{\oc}^{k, \mathbf{p}_{oc, k}, \mathbf{w}_{oc, k}}(y; x_{0}, \ldots, x'_{i}, \ldots, x_{k})$.
Such strip breakings do not affect the popsicle structure on the main component.

\end{enumerate}

The following result for transversality and compactness of the moduli space $\bar{\M}_{\oc}^{k, \mathbf{p}_{oc, k}, \mathbf{w}_{oc, k}}(y; \mathbf{x})$ \eqref{compactified moduli space for oc} is similar to Proposition \ref{prop: transversality and compactness of popsicles},
where we do not need the relevant result for moduli spaces of dimension two.
The proof for transversality is a standard Sard-Smale argument (essentially following \cites{abouzaid_gc, seidel_book}),
which is possible since the Floer data are chosen to break $S^{1}$-symmetry over cylindrical ends,
and compactness follows from maximum principle plus Gromov compactness.

\begin{prop}\label{prop: transversality and compactness of oc moduli space}
There exists a universal and conformally consistent choice of Floer data $\mathbf{D}_{\oc}$ such that the following holds.
\begin{enumerate}[label=(\roman*)]

\item If $k + |F| + n(1+2\d+\sum_{i=1}^{k}\d_{i}) - (-1)^{\d}|y| - \sum_{i=1}^{k} (-1)^{\d_{i}}|x_{i}|=0$, 
the moduli space $\bar{\M}_{\oc}^{k, \mathbf{p}_{oc, k}, \mathbf{w}_{oc, k}}(y; \mathbf{x})$ \eqref{compactified moduli space for oc} is a compact smooth manifold of dimension zero.

\item If $k + |F| + n(1+2\d+\sum_{i=1}^{k}\d_{i}) - (-1)^{\d}|y| - \sum_{i=1}^{k} (-1)^{\d_{i}}|x_{i}|=1$, 
the moduli space $\bar{\M}_{\oc}^{k, \mathbf{p}_{oc, k}, \mathbf{w}_{oc, k}}(y; \mathbf{x})$ \eqref{compactified moduli space for oc} is a compact smooth manifold-with-boundary of dimension one.

\end{enumerate}
\end{prop}

\subsection{The open-closed map}

The diagonal bimodule $\rw_{\D}$ of $\rw$ can be regarded as a $\w$-bimodule via the pullback by the pair of functors $(j, j)$,
where $j: \w \to \rw$ is the canonical functor \eqref{w to rw}.
For simplicity of notation, we put
\begin{equation}
\r{CC}_{*}(\w, \rw) := \r{CC}_{*}(\w, (j, j)^{*}\rw_{\D}),
\end{equation}
and call it the Hochschild chain complex of $\w$ with coefficients in $\rw$.

Define the {\it Rabinowitz open-closed map}
\begin{equation}\label{wrw oc}
\oc_{R}: \r{CC}_{*-n}(\w, \rw) \to RFC^{*}(X)
\end{equation}
by counting rigid elements in the moduli spaces $\M_{\oc}^{k, \mathbf{p}_{oc, k}, \mathbf{w}_{oc, k}}(y; \mathbf{x})$.
That is, the map \eqref{wrw oc} is a direct sum of maps from $k$-fold tensor products
\begin{equation}
\oc_{R} = \bigoplus_{k \ge 0} \oc_{R}^{k},
\end{equation}
 where
 \begin{equation}
 \oc_{R}^{k}: \bigoplus_{L_{0}, \ldots, L_{k-1} \in \L} RC^{*}(L_{k-1}, L_{1}) \otimes CW^{*}(L_{k-2}, L_{k-1}; H) \otimes \cdots \otimes CW^{*}(L_{0}, L_{1}; H) \to RFC^{*}(X).
 \end{equation}
 This is further decomposed into components 
 \begin{equation}
 \oc_{R}^{k} = \bigoplus_{\mathbf{p}_{oc, k}, \mathbf{w}_{oc, k}} \oc_{R}^{k, \mathbf{p}_{oc, k}, \mathbf{w}_{oc, k}},
 \end{equation}
 where the sum is over all possible injective maps $\mathbf{p}_{oc, k}: F \to \{k\}$ and weights $\mathbf{w}_{oc, k}$ defined in \eqref{extended weight}, 
 satisfying \eqref{oc weight 1}, such that the component
 \begin{equation}
 \begin{split}
 \oc_{R}^{k, \mathbf{p}_{oc, k}, \mathbf{w}_{oc, k}}: \bigoplus_{L_{0}, \ldots, L_{k-1} \in \L} & CW^{*}(L_{k-1}, L_{1}; \sk_{k}H) \otimes CW^{*}(L_{k-2}, L_{k-1}; H)\\
 & \otimes \cdots \otimes CW^{*}(L_{0}, L_{1}; H) \to SC^{*}(X; \sk H_{F})
 \end{split}
 \end{equation}
 is defined on a basis of elements by
\begin{equation}\label{components of rw oc}
\begin{split}
&\oc_{R}^{k, \mathbf{p}_{oc, k}, \mathbf{w}_{oc, k}}([x_{k}^{\sk_{k}}] \otimes [x_{k-1}] \otimes \cdots \otimes [x_{1}])\\
= & \prod_{\substack{y\\ (-1)^{\d}|y| = k + |F| + n(1+2\d+\sum_{i=1}^{k}\d_{i}) - \sum_{i=1}^{k} (-1)^{\d_{i}}|x_{i}|}} \sum_{u \in \M_{\oc}^{k, \mathbf{p}_{oc, k}, \mathbf{w}_{oc, k}}(y; \mathbf{x})} \\
&(-1)^{*_{k, \mathbf{p}_{oc, k}, \mathbf{w}_{oc, k}} + \diamond_{k, \mathbf{p}_{oc, k}, \mathbf{w}_{oc, k}} + |x_{k}^{\sk_{k}}|} o_{u}([x_{k}^{\sk_{k}}] \otimes [x_{k-1}] \otimes \cdots \otimes [x_{1}]).
\end{split}
\end{equation}
Here the symbols $\sk, \sk_{i}$ are defined in \eqref{sign symbol},
and the signs are \eqref{popsicle sign 1}, \eqref{popsicle sign 2}.
In this case, since $w_{1} = \ldots = w_{k-1} = 0$, 
we have by \eqref{oc weight 1} that $\mathbf{p}_{oc, k}^{-1}(i) = \varnothing, i = 1, \ldots, k-1$.
It follows that the signs get simplified to
\begin{align}
*_{k, \mathbf{p}_{oc, k}, \mathbf{w}_{oc, k}} = & \sum_{i=1}^{k} i|x_{i}^{\sk_{i}}|, \\
\diamond_{k, \mathbf{p}_{oc, k}, \mathbf{w}_{oc, k}} = 0.
\end{align}
The formula \eqref{components of rw oc} can be rewritten as
\begin{equation}
\begin{split}
&\oc_{R}^{k, \mathbf{p}_{oc, k}, \mathbf{w}_{oc, k}}([x_{0}^{\sk_{0}}] \otimes [x_{k}] \otimes \cdots \otimes [x_{1}])\\
= & \prod_{\substack{y\\ (-1)^{\d}|y| = k + |F| + n(1+2\d+\sum_{i=1}^{k}\d_{i}) - \sum_{i=1}^{k} (-1)^{\d_{i}}|x_{i}|}} \sum_{u \in \M_{\oc}^{k, \mathbf{p}_{oc, k}, \mathbf{w}_{oc, k}}(y; \mathbf{x})} \\
&(-1)^{\sum_{i=1}^{k-1} i |x_{i}| + (k+1)|x_{k}^{\sk_{k}}|} o_{u}([x_{k}^{\sk_{k}}] \otimes [x_{k}] \otimes \cdots \otimes [x_{1}]).
\end{split}
\end{equation}

If we restrict the special input from $\rw$ to $\w$ via the functor $j: \w \to \rw$,
which, in terms of open-closed strings with popsicle structures introduced in \S\ref{section: popsicles on open-closed strings}, means that we require $w_{k}=0$,
then \eqref{oc weight 1} implies $w = 0$ and $F = \varnothing$ for the moduli spaces $\M_{\oc}^{k, \mathbf{p}_{oc, k}, \mathbf{w}_{oc, k}}(y; \mathbf{x})$ with $w_{k}=0$.
Therefore, all such moduli spaces are the usual moduli spaces of open-closed strings $\M_{\oc}^{k}(y; \mathbf{x})$
In this case, the Rabinowitz open-closed string restricts to the (usual) {\it open-closed map}
\begin{equation}\label{usual oc}
\oc: \r{CC}_{*-n}(\w, \w) \to SC^{*}(X; H_{F}).
\end{equation}

\begin{prop}\label{lem: wrw oc chain map}
The map $\oc_{R}$ \eqref{wrw oc}, whose components $\oc_{R}^{k}$ are defined in \eqref{components of rw oc}, is a chain map.
\end{prop} 
\begin{proof}
This is a standard degeneration-gluing argument,
which follows from the study of the codimension-one boundary strata of the moduli spaces $\bar{\M}_{\oc}^{k, \mathbf{p}_{oc, k}, \mathbf{w}_{oc, k}}(y; \mathbf{x})$,
where the boundary strata are described in \eqref{oc boundary 1} - \eqref{oc boundary 6}.
together with Proposition \ref{prop: transversality and compactness of oc moduli space}.
\end{proof}

\begin{prop}\label{prop: wrw oc iso}
Whenever the open-closed map $\oc$ is a quasi-isomorphism,
the Rabinowitz open-closed map \eqref{wrw oc} is also a quasi-isomorphism.
\end{prop}
\begin{proof}
By Proposition \ref{prop: rw as cone of bimodule}, we have a quasi-isomorphism 
\[
\iota: \cone(\ck: \w_{-} \to (\w^{op})_{\D}) \to \rw
\]
of $\w^{op}$-bimodules.
This implies that the diagonal bimodule of $\rw$, viewed as a $\w$-bimodule via the pullback by the pair of functors $(j, j)$, comes with a quasi-isomorphism of $\w$-bimodules 
\begin{equation}\label{rw as w bimodule 2}
\iota: \cone(\ck: \w_{-}^{\top} \to \w_{\D}) \to (j, j)^{*}\rw_{\D},
\end{equation}
where $\w_{-}^{\top}$ is defined as the $\w$-bimodule that is the transpose of $\w_{-}$, 
whose underlying chain complex is
\begin{equation}
\w_{-}^{\top}(K, L) = \w_{-}(L, K) = CW^{*}(L, K; -H).
\end{equation}
The underlying map of $\iota$ is the canonical chain-level identity map on chain complexes underlying both sides of \eqref{rw as w bimodule 2}.
Thus we get a quasi-isomorphism of chain complexes
\begin{equation}\label{cc wrw as cone}
\iota_{*}: \r{CC}_{*}(\w, \cone(\w_{-}^{\top} \stackrel{\ck}\to \w_{\D})) = \cone(\r{CC}_{*}(\w, \w_{-}^{\top}) \stackrel{\ck_{*}}\to \r{CC}_{*}(\w, \w_{\D}))  \stackrel{\sim}\to \r{CC}_{*}(\w, (j, j)^{*}\rw_{\D}).
\end{equation}
The composition of \eqref{cc wrw as cone} with the open-closed map $\oc_{R}$ yields a chain map
\begin{equation}
\cone(\r{CC}_{*}(\w, \w_{-}^{\top}[-n]) \stackrel{\ck_{*}}\to \r{CC}_{*}(\w, \w_{\D}[-n])) \to RFC^{*}(X) = \cone(SC^{*}(X; -H_{F}) \stackrel{c}\to SC^{*}(X; H_{F})).
\end{equation}
To prove such a chain map on mappings cones is quasi-isomorphism, it suffices to prove that the filtered pieces
\begin{align}
\oc_{-}[1]: \r{CC}_{*}(\w, \w_{-}^{\top}[-n])[1] &\to SC^{*}(X; -H_{F})[1], \label{oc-} \\
\r{CC}_{*}(\w, \w_{\D}[-n]) &\to SC^{*}(X; H_{F}), \label{oc}
\end{align}
are both quasi-isomorphisms.
The second map \eqref{oc} is precisely the usual open-closed map $\oc$ \eqref{usual oc} (see the paragraph above \eqref{usual oc} on how it is defined).

The first map \eqref{oc-} is defined in a similar way, 
by counting rigid elements in the moduli space $\M_{\oc}^{k, \mathbf{p}_{oc, k}, \mathbf{w}_{oc, k}}(y; \mathbf{x})$ subject to the weight condition $w = w_{0}=-1$.
Note that by Propostion \ref{prop: bimodule PD}, Poincar\'{e} duality \eqref{PD for w-} extends to a quasi-isomorphism of $\w$-bimodules 
\begin{equation}\label{bimodule PD 1}
\mathcal{I}: \w_{-}^{\top} \stackrel{\sim}\to \w^{\vee}[-n].
\end{equation}
Since quasi-isomorphisms of $\ainf$-bimodules are invertible, we also get an inverse
\begin{equation}\label{bimodule PD inverse}
\mathcal{K}: \w^{\vee}[-n] \stackrel{\sim}\to \w_{-}^{\top},
\end{equation}
such that the $(0,0)$-th order terms of both \eqref{bimodule PD 1} and \eqref{bimodule PD inverse} are the identity map on the chain complexes.
Thus we get a quasi-isomorphism
\begin{equation}
\mathcal{I}_{*}: \r{CC}_{*}(\w, \w_{-}^{\top}) \stackrel{\sim}\to \r{CC}_{*}(\w, \w^{\vee}[-n])
\end{equation}
as well as its chain homotopy inverse
\begin{equation}\label{cc PD inverse}
\mathcal{K}_{*}: \r{CC}_{*}(\w, \w^{\vee}[-n]) \stackrel{\sim}\to \r{CC}_{*}(\w, \w_{-}^{\top}).
\end{equation}
Since $\w$ has a weak smooth Calabi-Yau structure $\sigma$,
we by Lemma \ref{lem: capping with cy is iso} a quasi-isomorphism
\begin{equation}\label{capping with cy for w dual}
-\cap \sigma: \r{CC}^{*}(\w, \w^{\vee}) \stackrel{\sim}\to \r{CC}_{*}(\w, \w^{\vee}[-n]).
\end{equation}
Note that we have a canonical evaluation pairing
\begin{equation}\label{evaluation on hochschild}
\r{CC}^{*}(\w, \w^{\vee}) \otimes \r{CC}_{*}(\w, \w) \to \K
\end{equation}
giving a chain-level identification 
\begin{equation}\label{hochschild homology dual}
\r{CC}^{*}(\w, \w^{\vee}) \cong \r{CC}_{*}(\w, \w)^{\vee}.
\end{equation}
Thus we get the following composition
\begin{equation}\label{oc dual}
\begin{split}
& \r{CC}_{*}(\w, \w)^{\vee}[-n] = \r{CC}^{*}(\w, \w^{\vee}[-n]) \stackrel{-\cap \sigma}\to \r{CC}_{*}(\w, \w^{\vee}[-2n])\\
& \stackrel{\mathcal{K}_{*}}\to \r{CC}_{*}(\w, \w_{-}^{\top}[-n]) \stackrel{\oc_{-}}\to SC^{*}(X; -H_{F}) \stackrel{I^{-1}}\to SC^{*}(X; H_{F})^{\vee}[-2n].
\end{split}
\end{equation}
Since $\oc$ is assumed to be a quasi-isomorphism, so is its linear dual.
On the level of homology groups, 
this composition \eqref{oc dual} induces the same map as the inverse of the linear dual of $\oc$ \eqref{usual oc},
because of the way we count rigid elements in $\M_{\oc}^{k, \mathbf{p}_{oc, k}, \mathbf{w}_{oc, k}}(y; \mathbf{x})$ when $w=w_{0}=-1$:
we think of $y^{-} \in \o(-H_{F})$ as an input that corresponds to the actual geometric output $y \in \o(H_{F})$,
and similarly for $x_{0}^{-}$.
Thus $\oc_{-}$ \eqref{oc-} is a quasi-isomorphism,
which implies that the Rabinowitz open-closed map $\oc_{R}$ \eqref{rw oc} is a quasi-isomorphism.
\end{proof}

Summarizing the Proof of Proposition \ref{prop: wrw oc iso},
we can regard the Rabinowitz open-closed map $\oc_{R}$ \eqref{wrw oc}, 
in view of the identification \eqref{cc wrw as cone},
as a chain map of mapping cones:
\begin{equation}
\cone(\r{CC}_{*}(\w, \w_{-}^{\top}[-n]) \stackrel{\ck_{*}}\to \r{CC}_{*}(\w, \w_{\D}[-n])) \to \cone(SC^{*}(X; -H_{F}) \stackrel{c}\to SC^{*}(X; H_{F})).
\end{equation}
Such a map has three terms, two of which are already seen as $\oc$ \eqref{oc} and $\oc_{-}$ \eqref{oc-},
defined by counting rigid elements in the moduli spaces $\M_{\oc}^{k, \mathbf{p}_{oc, k}, \mathbf{w}_{oc, k}}(y; \mathbf{x})$ when $w=w_{0}=0$ and when $w=w_{0}=-1$ respectively.
There is another term
\begin{equation}\label{oc-+}
\oc_{-,+}: \r{CC}_{*}(\w, \w_{-}^{\top}[-n])[1] \to SC^{*}(X; H_{F}),
\end{equation}
which is precisely given by counting rigid elements in the moduli space $\M_{\oc}^{k, \mathbf{p}_{oc, k}, \mathbf{w}_{oc, k}}(y; \mathbf{x})$ when $w=0, w_{0}=-1$ with non-trivial popsicle structures.
This is not a chain map, but rather a chain homotopy.
In other words, the way in which we define $\oc_{R}$ \eqref{wrw oc} is a way of packaging the three maps $\oc$ \eqref{oc}, $\oc_{-}$ \eqref{oc-} and $\oc_{-, +}$ \eqref{oc-+}.

Going one step beyond, we wonder what structures $\r{HH}_{*}(\rw, \rw)$ and $\r{HH}^{*}(\rw, \rw)$ themselves carry and how they are related to $RFH^{*}(X)$.
In an upcoming work, we shall show that the open-closed map factors through a map
\begin{equation} \label{rw oc}
\r{HH}_{*-n}(\rw(X), \rw(X)) \to RFH^{*}(X).
\end{equation}
The reader should be warned that the current method by using the moduli spaces of pseudoholomorphic disks considered in this paper is not appropriate for defining this map.
The map \eqref{rw oc} is always surjective, but not always an isomorphism.
However, it will play an important role in understanding the relation between the residue on $\rw$ and the trace on $RFH^{*}(X)$.

\subsection{Relating pairings}\label{section: relating pairings}

We close this section by establishing one last property of the Rabinowitz open-closed map $\oc_{R}$,
which has to do with pairings.
First, we shall construct a pairing of degree $1-2n$ on $\r{CC}_{*-n}(\w, \rw)$:
\begin{equation}\label{cc wrw pairing}
\langle \cdot, \cdot \rangle_{\sigma}: \r{CC}_{*-n}(\w, \rw) \otimes \r{CC}_{(2n-1-*)-n}(\w, \rw) \to \K[1-2n]
\end{equation}
which will depend on the weak smooth Calabi-Yau structure $\sigma$ on $\w$.
Using the identification \eqref{cc wrw as cone},
it suffices to construct a pairing of degree $1$
\begin{equation}\label{cc wrw pairing 1}
\cone(\r{CC}_{*}(\w, \w_{-}^{\top}) \to \r{CC}_{*}(\w, \w)) \otimes \cone(\r{CC}_{*}(\w, \w_{-}^{\top}) \to \r{CC}_{*}(\w, \w)) \to \K[1].
\end{equation}
Moreover, via the terms in the composition \eqref{oc dual} before $\oc_{-}$ (without degree shift by $-n$), 
i.e., via the following composition
\begin{equation}\label{w dual is w-}
\mathcal{K}_{*} \circ (-\cap \sigma): \r{CC}_{*}(\w, \w)^{\vee} = \r{CC}^{*}(\w, \w^{\vee}) \stackrel{-\cap \sigma}\to \r{CC}_{*}(\w, \w^{\vee}[-n]) \stackrel{\mathcal{K}_{*}}\to \r{CC}_{*}(\w, \w_{-}^{\top}),
\end{equation}
it suffices to construct a pairing
\begin{equation}\label{cc wrw pairing 2}
\cone(\r{CC}_{*}(\w, \w)^{\vee} \to \r{CC}_{*}(\w, \w)) \otimes \cone(\r{CC}_{*}(\w, \w)^{\vee} \to \r{CC}_{*}(\w, \w)) \to \K[1].
\end{equation}
We define this pairing to be the tautological pairing $\langle \cdot, \cdot \rangle_{taut}$ \eqref{taut pairing}.
Thus we obtain the pairing \eqref{cc wrw pairing}, 
which depends on the weak smooth Calabi-Yau structure in the way that the map \eqref{w dual is w-} does.

\begin{prop}\label{prop: oc intertwines pairing}
The map $\oc_{R}$ \eqref{wrw oc} respects pairings on the level of cohomology groups.
\end{prop}
\begin{proof}
Upon the identifications \eqref{w dual is w-} and \eqref{cc wrw as cone} as well as the Poincar\'{e} duality isomorphism \eqref{PD for sc-}, 
the map $\oc_{R}$ \eqref{wrw oc} is equivalent to a chain map between mapping cones 
\begin{equation}
\cone(\r{CC}_{*}(\w, \w)^{\vee}[-n] \to \r{CC}_{*}(\w, \w)[-n]) \to \cone(SC^{*}(X; H_{F})^{\vee}[-2n] \to SC^{*}(X; H_{F}))
\end{equation}
which is equivalent to a diagram
\begin{equation}\label{cd for oc}
\begin{tikzcd}
& \r{CC}_{*}(\w, \w)^{\vee}[-n] \arrow[r, "\bar{\ck}_{*}"] \arrow[d, "\tilde{\oc}_{-}"] \arrow[rd, "\tilde{\oc}_{-+}"] & \r{CC}_{*}(\w, \w)[-n] \arrow[d, "\oc"] \\
& SC^{*}(X; H_{F})^{\vee}[-2n] \arrow[r, "\bar{c}"] & SC^{*}(X; H_{F}).
\end{tikzcd}
\end{equation}
Here $\bar{\ck}_{*} = \ck_{*} \circ \mathcal{K}_{*} \circ (-\cap \sigma)$, 
 $\mathcal{K}_{*}$ is \eqref{cc PD inverse}, 
 $\mathcal{K}_{*} \circ (-\cap \sigma)$ is \eqref{w dual is w-}, $\ck: \w_{-}^{\top} \to \w_{\D}$ is the bimodule continuation map \eqref{bimodule continuation map};
$\tilde{\oc}_{-} = I^{-1} \circ \oc_{-} \circ \mathcal{K}_{*} \circ (-\cap \sigma)$ is the composition \eqref{oc dual},
 $\bar{c} = c \circ I^{-1}$,
 $c: SC_{*}(X) \to SC^{*}(X)$ is the continuation map and $I^{-1}$ is the inverse of \eqref{PD for sc-},
  $\tilde{\oc}_{-+} = \oc_{-+} \circ \mathcal{K}_{*} \circ (-\cap \sigma)$,
and $\oc_{-+}$ is \eqref{oc-+}.

The induced map on homology by $\tilde{\oc}_{-}$ is 
\begin{equation}
(\tilde{\oc}_{-})_{*}: \r{HH}_{*}(\w, \w)^{\vee}[-n] \to SH^{*}(X)^{\vee}[-2n],
\end{equation}
which is the inverse of the induced map $\oc^{\vee}_{*}$,
as argued in the Proof of Proposition \ref{prop: wrw oc iso},
specifically for the map \eqref{oc dual}.

The map 
\[
\tilde{\oc}_{-+}: \r{CC}_{*}(\w, \w)^{\vee}[-n] \to SC^{*}(X; H_{F})[1]
\]
has degree $1$ and is not a chain map, but rather a chain homotopy between the two sides of the square \eqref{cd for oc}.
Thus we get a commutative diagram of homology groups:
\begin{equation}\label{cd for oc homology}
\begin{tikzcd}
& \r{HH}_{*}(\w, \w)^{\vee}[-n] \arrow[r, "\bar{\ck}_{*}"] \arrow[d, "(\tilde{\oc}_{-})_{*}"]  & \r{HH}_{*}(\w, \w)[-n] \arrow[d, "(\oc)_{*}"] \\
& SH^{*}(X)^{\vee}[-2n] \arrow[r, "\bar{c}"] & SH^{*}(X),
\end{tikzcd}
\end{equation}
where both vertical arrows are isomorphisms,
and the left arrow is the inverse to the linear dual of the right arrow.

Both the pairing $\langle \cdot, \cdot \rangle_{\sigma}$ \eqref{cc wrw pairing} on $\r{CC}_{*-n}(\w, \rw)$ and the pairing $\langle \cdot, \cdot \rangle_{taut}$ \eqref{taut pairing on rfc} on $RFC^{*}(X)$ are defined in terms of the tautological pairing \eqref{taut pairing},
induced by the evaluation pairing $ev$ \eqref{evaluation on a}.
Thus, on the level of homology, commutativity of \eqref{cd for oc homology} along with the fact that the left arrow is the inverse to the linear dual of the right arrow,
implies that $(\oc_{R})_{*}$ respects pairings on homology groups.
\end{proof}

Now we are ready to prove Theorem \ref{thm: oc respects pairing}.

\begin{proof}[Proof of Theorem \ref{thm: oc respects pairing}]
The map $\oc_{R}$ is a chain map by Lemma \ref{lem: wrw oc chain map},
and moreover a quasi-isomorphism by Proposition \ref{prop: wrw oc iso}.
It respects pairings by Proposition \ref{prop: oc intertwines pairing}.
\end{proof}

One thing to notice is that the domain of the homology level open-closed map $\oc_{R}$ indeed carries a ring structure.
When $X$ is nondegenerate with $c_{1}(X)=0$, the existence of a weak smooth Calabi-Yau structure $\sigma$ on $\w$ implies that there is an isomorphism
\begin{equation}
\r{HH}_{*-n}(\w, \rw) \cong \r{HH}_{*-n}(\w, \winf) \cong \r{HH}^{*}(\w, \winf),
\end{equation}
where the first arrow is induced by $\Phi_{*}$ and the second arrow is induced by capping with $\sigma$.
Finally, notice that the Hochschild cohomology $\r{HH}^{*}(\w, \winf)$ carries a product from \eqref{extended cup product},
which is in fact graded commutative.
This answers the first part of Question 1.8 in \cite{GGV}.
For the second part of that question, 
general structural properties of Fukaya categories could be used to prove that $\oc_{R}$ is ring homomorphism ((e.g. \cite{perutz-sheridan}),
but  issues with duality (e.g. as address in \cite{CO_tate}) remain to be resolved for $\ainf$-categories enriched in topological vector spaces.

\bibliography{residue}
 
\end{document}